\documentclass[10pt]{amsart}
\usepackage{amsmath,amssymb,amscd,amsfonts}
\usepackage{mathtools,adforn,mathabx}
\usepackage{graphicx,xcolor,color}
\usepackage{tikz}
\usepackage[all,cmtip]{xy}
\usepackage[hidelinks]{hyperref}
\hypersetup{
  colorlinks,
  linkcolor={red!85!black},
  citecolor={blue!85!black},
  urlcolor={blue!95!black}
}

\usepackage[T1]{fontenc}
\usepackage{lmodern}
\usepackage{mathrsfs,upgreek}

\makeatletter
\@namedef{subjclassname@2020}{\textup{2020} Mathematics Subject Classification}
\makeatother

\numberwithin{equation}{section} 
\numberwithin{figure}{section} 
\mathtoolsset{showonlyrefs} 
\usepackage{enumitem} 
\setlist[enumerate]{label=$(\roman*)$, ref=$(\roman*)$}

\theoremstyle{plain}
\newtheorem{theoalph}{Theorem}

\newtheorem{theorem}{Theorem}[section]
\newtheorem{coro}[theorem]{Corollary}
\newtheorem{proposition}[theorem]{Proposition}
\newtheorem{lemma}[theorem]{Lemma}
\newtheorem{question}[theorem]{Question}
\newtheorem{conjecture}[theorem]{Conjecture}

\newtheorem{innercustomthm}{Theorem}
\newenvironment{custtheo}[1]
{\renewcommand\theinnercustomthm{#1}\innercustomthm}
{\endinnercustomthm}

\theoremstyle{definition}
\newtheorem{defi}[theorem]{Definition}

\theoremstyle{remark}
\newtheorem{remark}[theorem]{Remark}

\newtheoremstyle{citing}
{3pt}
{3pt}
{\itshape}
{}
{\bfseries}
{.}
{.5em}
{\thmnote{#3}}

\theoremstyle{citing}
\newtheorem*{generic}{}

\newcommand{\A}{\mathbb{A}}

\newcommand{\C}{\mathbb{C}}

\newcommand{\F}{\mathbb{F}}

\newcommand{\N}{\mathbb{N}}
\newcommand{\Q}{\mathbb{Q}}
\newcommand{\R}{\mathbb{R}}

\newcommand{\Z}{\mathbb{Z}}

\renewcommand{\H}{\mathbb{H}}
\renewcommand{\P}{\mathbb{P}}

\newcommand{\cA}{\mathcal{A}}

\newcommand{\cD}{\mathcal{D}}

\newcommand{\cF}{\mathcal{F}}

\newcommand{\cK}{\mathcal{K}}

\newcommand{\cM}{\mathcal{M}}

\newcommand{\cO}{\mathcal{O}}

\newcommand{\cR}{\mathcal{R}}

\newcommand{\sK}{\mathscr{K}}

\newcommand{\sM}{\mathscr{M}}

\newcommand{\sZ}{\mathscr{Z}}

\newcommand{\hE}{\widehat{E}}

\newcommand{\hK}{\widehat{K}}

\newcommand{\halpha}{\widehat{\alpha}}

\newcommand{\hgamma}{\widehat{\gamma}}
\newcommand{\hGamma}{\widehat{\Gamma}}

\newcommand{\hphi}{\widehat{\phi}}
\newcommand{\hvarphi}{\widehat{\varphi}}
\newcommand{\hPhi}{\widehat{\Phi}}

\newcommand{\tE}{\widetilde{E}}

\newcommand{\tvarphi}{\widetilde{\varphi}}

\newcommand{\tpsi}{\widetilde{\psi}}

\newcommand{\partn}[1]{{\smallskip \noindent \textbf{#1.}}}
\renewcommand{\=}{\coloneqq}
\newcommand{\dd}{\hspace{1pt}\operatorname{d}\hspace{-1pt}}

\DeclareMathOperator{\dist}{dist}
\DeclareMathOperator{\supp}{supp} 

\newcommand{\bfB}{\mathbf{B}}
\newcommand{\bfD}{\mathbf{D}}
\newcommand{\bfG}{\mathbf{G}}

\newcommand{\bfL}{\mathbf{L}}

\newcommand{\bfR}{\mathbf{R}}

\newcommand{\bfX}{\mathbf{X}}
\newcommand{\bfY}{\mathbf{Y}}

\DeclareMathOperator{\oO}{O}

\newcommand{\whf}{\widehat{f}}

\newcommand{\whj}{\widehat{j}}

\newcommand{\whw}{\widehat{w}}

\newcommand{\tcF}{\widetilde{\cF}}

\newcommand{\ssetminus}{\smallsetminus}

\DeclareMathOperator{\Aut}{Aut}
\DeclareMathOperator{\End}{End}
\DeclareMathOperator{\Hom}{Hom}
\DeclareMathOperator{\Iso}{Iso}
\DeclareMathOperator{\Ker}{Ker}

\DeclareMathOperator{\Div}{Div}

\DeclareMathOperator{\Gal}{Gal}

\DeclareMathOperator{\SL}{SL}

\renewcommand{\N}{\Z_{> 0}}

\newcommand{\Fp}{\F_p}
\newcommand{\Fpalg}{\overline{\F}_p}
\newcommand{\Qalg}{\overline{\Q}}
\newcommand{\Qp}{\Q_p}
\newcommand{\Qpalg}{\overline{\Q}_p}

\newcommand{\OQpalg}{\cO_{\Qpalg}}

\newcommand{\Cp}{\C_p}
\newcommand{\Op}{\cO_p}

\newcommand{\OK}{\cO_{\cK}}

\DeclareMathOperator{\ord}{ord}

\DeclareMathOperator{\sups}{sups}
\newcommand{\Ell}{Y}

\newcommand{\Sups}{Y_{\sups}(\Cp)}
\newcommand{\tSups}{Y_{\sups}(\Fpalg)}

\newcommand{\CM}{CM}
\newcommand{\odelta}{\overline{\delta}}

\DeclareMathOperator{\Berk}{Berk}
\newcommand{\AKber}{\A^1_{\Berk}}
\DeclareMathOperator{\can}{can}
\newcommand{\xcan}{x_{\can}}

\renewcommand{\t}{\mathbf{t}}

\newcommand{\kval}{v_p}

\newcommand{\FE}{\cF_{E}}

\newcommand{\tFE}{\tcF_{E}}

\newcommand{\pd}{\mathfrak{D}} 
\newcommand{\pfd}{\mathfrak{d}} 

\newcommand{\Qpfd}{\Qp(\sqrt{\pfd})}

\newcommand{\dBss}{\dist_{\Bss}}

\DeclareMathOperator{\nr}{nr}
\DeclareMathOperator{\tr}{tr}

\newcommand{\rf}{\Bbbk}
\newcommand{\rfk}{\rf_0}

\renewcommand{\ss}{e}

\newcommand{\Bss}{\bfB_{\ss}}

\newcommand{\Dss}{\bfD_{\ss}}

\newcommand{\hDss}{\widehat{\bfD}_{\ss}}

\newcommand{\Fss}{\cF_{\ss}}

\newcommand{\Gss}{\bfG_{\ss}}

\newcommand{\Piss}{\Pi_{\ss}}

\newcommand{\Rss}{\bfR_{\ss}}

\newcommand{\Xss}{\bfX_{\ss}}

\DeclareMathOperator{\Fix}{Fix}
\newcommand{\Fixss}{\Fix_{\ss}}

\newcommand{\Nr}{\mathbf{Nr}}
\newcommand{\NE}{\Nr_{E}}
\newcommand{\Npfd}{\Nr_{\pfd}}
\newcommand{\coset}{\mathfrak{N}}
\DeclareMathOperator{\Orb}{Orb}
\newcommand{\corbit}{\Orb_{\coset}(E)}
\newcommand{\corbitc}{\overline{\corbit}}

\newcommand{\VQ}{M_{\Q}}
\newcommand{\Cv}{\C_v}
\newcommand{\Cw}{\C_w}
\newcommand{\nh}{\operatorname{h}_{\operatorname{W}}}
\newcommand{\sm}{\mathfrak{j}}
\newcommand{\fsm}{\mathfrak{f}} 
\newcommand{\hsm}{\mathfrak{h}} 

\newcommand{\VK}{M_{K}}
\newcommand{\VhK}{M_{\hK}}

\newcommand{\Kalg}{\overline{K}}
\newcommand{\wf}{\mathbf{f}}

\title{There are at most finitely many singular moduli that are $S$-units}

\author{Sebasti\'an Herrero}
\address{
  Instituto de Matem\'aticas, Pontificia Universidad Cat\'olica de Valpara\'iso, Blanco Viel 596, Cerro Bar\'on, Valpara\'iso,
  Chile.}
\email{sebastian.herrero.m@gmail.com}

\author{Ricardo Menares}

\address{
  Facultad de Matem\'aticas, Pontificia Universidad Cat\'olica de Chile, Vicu\~na Mackenna 4860, Santiago, Chile.}

\email{rmenares@mat.uc.cl}

\author{Juan Rivera-Letelier}
\address{Department of Mathematics, University of Rochester.
  Hylan Building, Rochester, NY~14627, U.S.A.}
\email{riveraletelier@gmail.com}
\urladdr{\url{http://rivera-letelier.org/}}

\usepackage{microtype}
\begin{document}


\begin{abstract}
  We show that for every finite set of prime numbers~$S$, there are at most finitely many singular moduli that are $S$\nobreakdash-units.
  The key new ingredient is that for every prime number~$p$, singular moduli are $p$\nobreakdash-adically disperse.
  We prove analogous results for the Weber modular functions, the $\lambda$\nobreakdash-invariants and the McKay--Thompson series associated with the elements of the monster group.
  Finally, we also obtain that a modular function that specializes to infinitely many algebraic units at quadratic imaginary numbers must be a weak modular unit.
\end{abstract}

\maketitle

\section{Introduction}
\label{s:introduction}

A \emph{singular modulus} is the $j$\nobreakdash-invariant of an elliptic curve with complex multiplication.
These algebraic numbers lie at the heart of the theory of abelian extensions of imaginary quadratic fields, as they generate the ring class fields of quadratic imaginary orders.
This was predicted by Kronecker and referred to by himself as his \emph{liebsten Jugendtraum}.

A result going back at least to Weber, states that every singular modulus is an algebraic integer \cite[\S115, \emph{Satz~VI}]{Web08}.
Thus, the absolute norm of a singular modulus is a rational integer, and the same holds for a difference of singular moduli.
Gross and Zagier gave an explicit formula for the factorization of the absolute norms of differences of singular moduli~\cite{GroZag85}.
Roughly speaking, this formula shows that these absolute norms are highly divisible numbers.
In fact, Li showed recently that the absolute norm of every difference of singular moduli is divisible by at least one prime number~\cite{Li21}.
Equivalently, that no difference of singular moduli is an algebraic unit.
Li's work extends previous results of Habegger~\cite{Hab15} and of Bilu, Habegger and K{\"u}hne~\cite{BilHabKuh20}.
These results answered a question raised by Masser in 2011, which was motivated by results of Andr{\'e}--Oort type.

In view of these results, one is naturally led to look at differences of singular moduli whose absolute norms are only divisible by a given set of prime numbers.
To be precise, recall that for a set of prime numbers~$S$, an algebraic integer is an \emph{$S$\nobreakdash-unit} if the only prime numbers dividing its absolute norm are in~$S$.
The following is our main result.

\begin{generic}[Main Theorem]
  Let~$S$ be a finite set of prime numbers and~$\sm_0$ a singular modulus.
  Then, there are at most finitely many singular moduli~$\sm$ such that ${\sm-\sm_0}$ is an $S$\nobreakdash-unit.
\end{generic}

To prove this result we follow Habegger's original strategy in the case where ${S = \emptyset}$ in~\cite{Hab15}.
The main new ingredient is that for every prime number~$p$, singular moduli are $p$\nobreakdash-adically disperse (Theorem~\ref{t:disperse} in Section~\ref{ss:disperse}).
We also prove analogous results for a more general class of modular functions that includes the Weber modular functions, the $\lambda$\nobreakdash-invariants and the McKay--Thompson series associated with the elements of the monster group (Theorem~\ref{t:S-units-genus-0} in Section~\ref{ss:S-units}).
In the course of the proof of these results, we obtain that a modular function that specializes to infinitely many algebraic units at quadratic imaginary numbers must be a weak modular unit (Theorem~\ref{t:units-are-rare} in Section~\ref{ss:units-are-rare}).

We also propose a conjecture whose affirmative solution would yield a vast generalization of the Main Theorem.
The conjecture is that for every prime number~$p$, every algebraic number is $p$\nobreakdash-adically badly approximable by singular moduli (Conjecture~\ref{c:badly-approximable} in Section~\ref{ss:badly-approximable-Hauptmodul}).
We show that an affirmative solution to this conjecture, would imply a version of the Main Theorem for every nonconstant modular function~$f$ for a congruence or genus zero group and every algebraic value of~$f$ (Corollary~\ref{c:S-units-conditional-non-unit} in Section~\ref{s:proof-S-units}).

\subsection{Singular moduli that are $S$-units}
\label{ss:S-units}
Consider the usual action of~$\SL(2, \R)$ on the upper-half plane~$\H$ and consider the $j$\nobreakdash-invariant as a holomorphic function defined on~$\H$ that is invariant under~$\SL(2, \Z)$.
Moreover, denote by~$\Qalg$ the algebraic closure of~$\Q$ inside~$\C$.

A subgroup~$\Gamma$ of~$\SL(2, \R)$ is \emph{commensurable to~$\SL(2, \Z)$}, if the intersection ${\Gamma \cap \SL(2, \Z)}$ has finite index in~$\Gamma$ and in~$\SL(2, \Z)$.
For such a group, denote by~$X(\Gamma)$ the Riemann surface obtained by compactifying the quotient~$\Gamma \backslash \H$.
The \emph{genus of~$\Gamma$} is the genus of~$X(\Gamma)$.
A \emph{modular function for~$\Gamma$} is a meromorphic function defined on~$\H$ that is obtained by lifting the restriction to~$\Gamma \backslash \H$ of a meromorphic function defined on~$X(\Gamma)$.
A meromorphic function defined on~$\H$ is a modular function if and only if it is algebraically dependent with the $j$\nobreakdash-invariant over~$\C$ (Proposition~\ref{p:modular-functions}).

A modular function is \emph{defined over~$\Qalg$}, if it is algebraically dependent with the $j$\nobreakdash-invariant over~$\Qalg$.
In this case, a \emph{singular modulus of~$f$} is a finite value that~$f$ takes at a quadratic imaginary number.
Every singular modulus of~$f$ is in~$\Qalg$ (Proposition~\ref{p:singular-moduli}$(i)$ in Section~\ref{ss:singular-moduli}).
We show that every modular function whose Fourier series expansion at~$i \infty$ has coefficients in~$\Qalg$ is defined over~$\Qalg$ (Proposition~\ref{p:arithmetically-modular} in Appendix~\ref{s:arithmetically-modular}).

Recall that for a set of prime numbers~$S$, a number in~$\Qalg$ is an \emph{$S$\nobreakdash-unit} if the leading and constant coefficients of its minimal polynomial in~$\Z[X]$ have all their prime factors in~$S$.

\begin{theoalph}
  \label{t:S-units-genus-0}
  Let~$f$ be a nonconstant modular function defined over~$\Qalg$ for a genus zero group.
  Moreover, let~$\fsm_0$ be a singular modulus of~$f$ and let~$S$ be a finite set of prime numbers.
  Then, there are at most finitely many singular moduli~$\fsm$ of~$f$ such that ${\fsm -\fsm_0}$ is an $S$\nobreakdash-unit.
\end{theoalph}

Since~$\SL(2, \Z)$ is of genus zero and the $j$\nobreakdash-invariant is a nonconstant modular function for~$\SL(2, \Z)$ defined over~$\Qalg$, the Main Theorem is Theorem~\ref{t:S-units-genus-0} applied to the $j$\nobreakdash-invariant.
Theorem~\ref{t:S-units-genus-0} also applies to the Weber modular functions, the $\lambda$\nobreakdash-invariants and the McKay--Thompson series associated with the elements of the monster group.
See Section~\ref{ss:notes} for details.

The following corollary is a direct consequence of Theorem~\ref{t:S-units-genus-0} with~$f$ equal to the $j$\nobreakdash-invariant and~$\fsm_0 = 0$, which is the $j$\nobreakdash-invariant of every elliptic curve whose endomorphism ring is isomorphic to~${\Z \left[ \tfrac{1 + \sqrt{3}i}{2} \right]}$.

\begin{coro}
  \label{c:S-units-j}
  For every finite set of prime numbers~$S$, there are at most finitely many singular moduli of the $j$\nobreakdash-invariant that are $S$\nobreakdash-units.
\end{coro}

When restricted to the $j$\nobreakdash-invariant and~$S=\emptyset$, Theorem~\ref{t:S-units-genus-0} is a particular instance of~\cite[Theorem~2]{Hab15} and of \cite[Corollary~1.3 with~${m = 1}$]{Li21}.\footnote{See Theorem~\ref{t:units-are-rare} in Section~\ref{ss:units-are-rare} for an extension of Habegger's result to a general modular function defined over~$\Qalg$ and Section~\ref{ss:notes} for further comments on Li's result.}
This last result extends the main result of~\cite{BilHabKuh20}, that in the case where ${S = \emptyset}$ the set of singular moduli in Corollary~\ref{c:S-units-j} is empty.
In contrast to these results, the proof of Theorem~\ref{t:S-units-genus-0}, which follows the strategy of proof of \cite[Theorem~2]{Hab15} in the case where~$f$ is the $j$\nobreakdash-invariant and ${S = \emptyset}$, does not give an effectively computable upper bound.

The number~$- 2^{15}$ is an example of a singular modulus of the $j$\nobreakdash-invariant that is a $\{2\}$\nobreakdash-unit.
In fact, $- 2^{15}$ is the $j$\nobreakdash-invariant of every elliptic curve whose endomorphism ring is isomorphic to ${\Z \left[ \tfrac{1 + \sqrt{11}i}{2} \right]}$.
Numerical computations suggest an affirmative answer to the following question, see, \emph{e.g.}, \cite{Sut00}.

\begin{question}
  Is $- 2^{15}$ the unique singular modulus of the $j$\nobreakdash-invariant that is a $\{2\}$\nobreakdash-unit?
\end{question}

For ${\sm_0 = 0}$ or~$1728$ and for the infinite set of prime numbers~$S$ for which every elliptic curve with $j$\nobreakdash-invariant equal to $\sm_0$ has potential ordinary reduction, \mbox{Campagna} shows the following in~\cite{Cam21}: If~$\sm$ is a singular modulus of the $j$\nobreakdash-invariant such that ${\sm - \sm_0}$ is an $S$\nobreakdash-unit, then ${\sm - \sm_0}$ is in fact an algebraic unit.
A combination of the Main Theorem and the arguments of \mbox{Campagna}, shows that when~$\sm_0 = 0$ or~$1728$ the conclusion of the Main Theorem holds for some infinite sets of prime numbers~$S$, see Section~\ref{ss:notes}.

\subsection{Singular moduli are disperse}
\label{ss:disperse}
Denote by~$\VQ$ the set of all prime numbers together with~$\infty$, put~$\C_{\infty} \= \C$ and denote by~$| \cdot |_{\infty}$ the usual absolute value on~$\C$.
Moreover, for each prime number~$p$ let~$(\Cp, | \cdot |_p)$ be a completion of an algebraic closure of the field of $p$\nobreakdash-adic numbers~$\Qp$, and identify the algebraic closure of~$\Q$ inside~$\Cp$ with~$\Qalg$.
For all~$v$ in~$\VQ$, $\alpha$ in~$\Cv$ and~$r > 0$, put
\begin{displaymath}
  \bfD_v(\alpha, r)
  \=
  \{ z \in \Cv \colon |z - \alpha|_v < r \}.
\end{displaymath}

For a finite extension~$K$ of~$\Q$ inside~$\Qalg$, consider the Galois group~$\Gal(\Qalg|K)$ and for each~$\alpha$ in~$\Qalg$ denote by~$\oO_K(\alpha)$ its orbit by~$\Gal(\Qalg|K)$.
The following result is stated for a modular function that is ``defined over~$K$'' in the sense of Definition~\ref{d:definition-field} in Section~\ref{ss:singular-moduli}.
For a modular function to be defined over~$K$, it is sufficient that its Fourier series expansion at~$i \infty$ has coefficients in~$K$ (Proposition~\ref{p:arithmetically-modular} in Appendix~\ref{s:arithmetically-modular}).

\begin{theoalph}[Singular moduli are disperse]
  \label{t:disperse}
  Let~$K$ be a finite extension of~$\Q$ inside~$\C$ and let~$f$ be a nonconstant modular function defined over~$K$.
  Then, for all~$v$ in~$\VQ$, $\alpha$ in~$\Cv$ and~$\varepsilon > 0$, there is~$r > 0$ such that the following property holds.
  For every singular modulus~$\fsm$ of~$f$ such that~$\# \oO_K(\fsm)$ is sufficiently large, we have
  \begin{equation}
    \label{eq:1}
    \# \left( \oO_K(\fsm) \cap \bfD_v(\alpha, r) \right)
    \le
    \varepsilon \cdot \# \oO_K(\fsm).
  \end{equation}
\end{theoalph}

We first establish this result for the $j$\nobreakdash-invariant, and then deduce the general case from this special case.
The case where ${v = \infty}$ and~$f$ is the~$j$\nobreakdash-invariant is a direct consequence of the fact that the asymptotic distribution of the singular moduli of the $j$\nobreakdash-invariant is given by a nonatomic measure~\cite{Duk88,CloUll04}.
In the case where~$v$ is a prime number~$p$, there are infinitely many measures that describe the $p$\nobreakdash-adic asymptotic distribution of the singular moduli of the $j$\nobreakdash-invariant.
The main ingredient in the proof of Theorem~\ref{t:disperse} is that none of these measures has an atom in~$\Cp$ (Theorem~\ref{t:CM-nonatomic} in Section~\ref{s:nonatomic}).
We also prove an analogous result for the Hecke orbit of every point in~$\Cp$ (Theorem~\ref{t:Hecke-orbits-nonatomic} in Section~\ref{ss:Hecke-orbits-nonatomic}).
As a consequence, we obtain that a Hecke orbit cannot have a significant proportion of good approximations of a given point in~$\Cp$ (Corollary~\ref{c:Hecke-orbits-nonatomic} in Section~\ref{ss:Hecke-orbits-nonatomic}), thus improving a result of Charles in~\cite{Cha18}.
The proofs of these results are based on the description of all the measures describing the $p$\nobreakdash-adic asymptotic distribution of singular moduli and Hecke orbits, given in the companion papers~\cite{HerMenRiv20,HerMenRivII}.

\subsection{Approximation by singular moduli}
\label{ss:badly-approximable-Hauptmodul}
Let~$f$ be a nonconstant modular function, denote by~$\Gamma$ its stabilizer in~$\SL(2, \R)$ and denote by~$f_0$ the meromorphic function defined on~$X(\Gamma)$ induced by~$f$.
A complex number is a \emph{cuspidal value of~$f$}, if it is a value that~$f_0$ takes at a cusp of~$X(\Gamma)$.
Note that the number of cuspidal values is finite.
Moreover, $f$ is a \emph{Hauptmodul} if~$X(\Gamma)$ is of genus zero and~$f_0$ is a biholomorphism from~$X(\Gamma)$ onto the Riemann sphere.

Let~$f$ be a nonconstant modular function defined over~$\Qalg$ and let~$v$ be in~$\VQ$.
A number~$\alpha$ in~$\Cv$ is \emph{badly approximable in~$\Cv$ by the singular moduli of~$f$}, if there are constants~$A > 0$ and~$B$ such that for every singular modulus~$\fsm$ of~$f$ different from~$\alpha$ we have
\begin{equation}
  \label{eq:2}
  - \log |\fsm - \alpha|_v
  \le
  A \log (\# \oO_{\Q}(\fsm)) + B.
\end{equation}
If this property does not hold, then~$\alpha$ is \emph{well approximated in~$\Cv$ by the singular moduli of~$f$}.

\begin{theoalph}
  \label{t:badly-approximable-Hauptmodul}
  Let~$f$ be a nonconstant modular function defined over~$\Qalg$, let~$\fsm_0$ be a singular modulus of~$f$ and let~$v$ be in~$\VQ$.
  In the case where ${v = \infty}$, assume that~$\fsm_0$ is a non-cuspidal value of~$f$ and in the case where~$v$ is a prime number, assume that~$f$ is a \emph{Hauptmodul}.
  Then, $\fsm_0$ is badly approximable in~$\Cv$ by the singular moduli of~$f$.
\end{theoalph}

In the case where ${v = \infty}$, the hypothesis that~$\fsm_0$ is a non-cuspidal value of~$f$ is necessary, see Proposition~\ref{p:cuspidal-are-approximated}$(i)$.

In the case where~$f$ is the $j$\nobreakdash-invariant and ${v = \infty}$, the theorem above is a direct consequence of a result of Habegger \cite[Lemmas~5 and~8 and formula~(11), or the proof of Lemma~6]{Hab15}.
In fact, using results of David and Hirata-Kohno in~\cite{DavHir09}, Habegger proved the stronger result that every algebraic number is badly approximable in~$\C$ by the singular moduli of the $j$\nobreakdash-invariant.
It is unclear to us whether the analogous result holds in the $p$\nobreakdash-adic setting.

\begin{conjecture}
  \label{c:badly-approximable}
  Let~$p$ be a prime number.
  Then, every algebraic number is badly approximable in~$\Cp$ by the singular moduli of the $j$\nobreakdash-invariant.
\end{conjecture}

We show that an affirmative solution to this conjecture would yield a version of Theorem~\ref{t:S-units-genus-0} for a general congruence or genus zero group and a general algebraic value (Corollary~\ref{c:S-units-conditional-non-unit} in Section~\ref{s:proof-S-units}).

\subsection{Weak modular units are the only source of singular units}
\label{ss:units-are-rare}
A \emph{modular unit} is a modular function without zeros or poles in~$\H$.
A \emph{weak modular unit} is a modular function~$u$ for which~$0$ is a cuspidal value of~$u$ and of~$\frac{1}{u}$.
Note that every nonconstant modular unit is a weak modular unit.

The singular moduli of modular units defined over~$\Qalg$ are a natural source of algebraic units, see, \emph{e.g.}, \cite{KubLan81}.
Roughly speaking, the following result asserts that among modular functions defined over~$\Qalg$, weak modular units are the only source of singular moduli that are algebraic units.

\begin{theoalph}
  \label{t:units-are-rare}
  Let~$f$ be a nonconstant modular function defined over~$\Qalg$ that is not a weak modular unit.
  Then, there are at most finitely many singular moduli of~$f$ that are algebraic units.
\end{theoalph}

We also show that an affirmative solution to Conjecture~\ref{c:badly-approximable} would imply a version of Theorem~\ref{t:units-are-rare} for $S$\nobreakdash-units (Corollary~\ref{c:S-units-conditional-non-weak} in Section~\ref{s:proof-S-units}) and a version of Theorem~\ref{t:units-are-rare} that holds under the weaker hypothesis that~$f$ is not a modular unit, but that is restricted to congruence or to genus zero groups (Corollary~\ref{c:S-units-conditional-non-unit} in Section~\ref{s:proof-S-units}).
Note that for every modular unit~$f$ defined over~$\Qalg$, there is a finite set of prime numbers~$S$ such that every singular modulus of~$f$ is an $S$\nobreakdash-unit, see Corollary~\ref{c:elliptic-S-units}$(ii)$ in Section~\ref{ss:cuspidal-values}.

An \emph{elliptic unit} is an algebraic unit that is the value of a modular unit defined over~$\Qalg$ at a quadratic imaginary number.
A natural problem that arises from Theorem~\ref{t:units-are-rare} is to determine those modular units that specialize to infinitely many elliptic units at quadratic imaginary numbers.
Examples of such can be easily extracted from \mbox{Weber}'s book \cite{Web08}.
Recall that the \emph{Weber modular functions}~$\wf$, $\wf_1$ and~$\wf_2$ are given in terms of Dedekind's~$\eta$ function by
\begin{equation}
  \label{eq:3}
  \wf(\tau)
  \=
  \exp \left( - \frac{\pi i}{24} \right) \frac{\eta \left( \frac{\tau + 1}{2} \right)}{\eta(\tau)},
  \wf_1(\tau)
  \=
  \frac{\eta \left( \frac{\tau}{2} \right)}{\eta(\tau)}
  \text{ and }
  \wf_2(\tau)
  \=
  \sqrt{2}
  \frac{\eta(2\tau)}{\eta(\tau)},
\end{equation}
see, \emph{e.g.}, \cite[\S34, (9)]{Web08}.
If~$p$ is a prime number satisfying ${p \equiv -1 \mod 8}$ and we put ${\tau_p \= \sqrt{p}i}$, then the singular modulus~$\wf \left( \frac{\tau_p - 1}{\tau_p + 1} \right)$ of~$\wf$ is equal to~${\frac{\sqrt{2}}{\wf(\tau_p)}}$ by~\cite[\S34, (18)]{Web08} and it is an algebraic unit by~\cite[\S142, p. 540]{Web08}.
Together with~\cite[\S34, (13), (14)]{Web08}, this implies that the singular moduli~$\wf_1 \left( - \frac{2}{\tau_p + 1} \right)$ and~$\wf_2 \left( \frac{\tau_p + 1}{2} \right)$ of~$\wf_1$ and~$\wf_2$ are both algebraic units.

Other examples of modular units that specialize to infinitely many elliptic units can be found in \cite{KubLan81}.
We mention the \emph{$\lambda$\nobreakdash-invariants} or \emph{modular~$\lambda$ functions}.
These are six \emph{Hauptmoduln} for the principal congruence group of level two, which can be defined as the roots of
\begin{equation}
  \label{eq:4}
  256 (1 - X + X^2)^3 - j X^2(1 - X)^2 = 0,
\end{equation}
see, \emph{e.g.}, \cite[Chapter~18, Section~6]{Lan87}.
Clearly, every singular modulus of a $\lambda$\nobreakdash-invariant is a $\{2 \}$\nobreakdash-unit.
By, \emph{e.g.}, \cite[Chapter~12, Section~2, Corollary of Theorem~5]{Lan87} or the more recent results of Yang, Yin and Yu \cite[Theorem~1.1]{YanYinYu21}, each of the six $\lambda$\nobreakdash-invariants has infinitely many singular moduli that are algebraic units.\footnote{Although these results only apply directly to one of the six $\lambda$\nobreakdash-invariants, they automatically imply analogous results for each of the remaining five $\lambda$\nobreakdash-invariants.
  Note that for every pair of $\lambda$\nobreakdash-invariants~$\lambda_0$ and~$\lambda_1$, there is~$\gamma$ in~$\SL(2, \Z)$ such that ${\lambda_1 = \lambda_0 \circ \gamma}$.}

To prove Theorem~\ref{t:units-are-rare}, we follow the strategy of proof of \cite[Theorem~2]{Hab15}.
In particular, we use \cite[Lemmas~5 and~8 and formula~(11)]{Hab15}, whose proof is based on results of David and Hirata--Kohno in \cite{DavHir09}.

\subsection{Notes and references}
\label{ss:notes}

Theorem~\ref{t:S-units-genus-0} applies to the Weber modular functions~$\wf$, $\wf_1$ and~$\wf_2$, see Section~\ref{ss:units-are-rare} for the definition.
In fact, $\wf_2$ is a \emph{Hauptmodul} by \cite[Theorem~1.3(2)(a) and p.~19]{YanYin16}, and therefore so are~$\wf$ and~$\wf_1$ by \cite[\S34, (13) and (14)]{Web08}.
On the other hand, each of these functions is defined over~$\Q$ in the sense of Definition~\ref{d:definition-field} in Section~\ref{ss:modular-functions} because it is a root of either
\begin{equation}
  \label{eq:5}
  (X^{24} - 16)^3 - X^{24} j = 0
  \text{ or }
  (X^{24} + 16)^3 - X^{24} j = 0,
\end{equation}
see, \emph{e.g.}, \cite{YuiZag97} or \cite[\S126, (1)]{Web08}.
The singular moduli of Weber modular functions provide generators of ring class fields of quadratic imaginary orders, see~\cite[\S126]{Web08} and \cite[\emph{Satz}~4.2]{Sch76}.
Additionally, the arithmetic complexity of these generators is sometimes significantly lower than that of the corresponding singular moduli of the $j$\nobreakdash-invariant, see~\cite{YuiZag97} and~\cite{EngSut10} for a computational perspective.
Since each of the functions~$\wf$, $\wf_1$ and~$\wf_2$ is a root of one of the polynomials in~\eqref{eq:5}, each of their singular moduli is an algebraic integer and a $\{ 2 \}$\nobreakdash-unit.
Moreover, as mentioned in Section~\ref{ss:units-are-rare}, the results of Weber imply that the singular moduli of~$\wf$, $\wf_1$ and~$\wf_2$ are often algebraic units, in contrast to the singular moduli of the $j$\nobreakdash-invariant.
Note that~$0$ is a singular modulus of the $j$\nobreakdash-invariant, but not of~$\wf$, $\wf_1$ or~$\wf_2$.
In fact, $\wf$, $\wf_1$ and~$\wf_2$ are all modular units, see, \emph{e.g.}, Corollary~\ref{c:elliptic-S-units}$(ii)$ in Section~\ref{ss:cuspidal-values}.

Theorem~\ref{t:S-units-genus-0} also applies to each of the six $\lambda$\nobreakdash-invariants, see Section~\ref{ss:units-are-rare} for the definition.
This solves affirmatively a conjecture of Habegger~\cite{Hab2102}.
Note that each of these functions is defined over~$\Q$, because it is a root of~\eqref{eq:4}.
As mentioned in Section~\ref{ss:units-are-rare}, the singular moduli of each of the six $\lambda$\nobreakdash-invariants are often algebraic units, in contrast to those of the $j$\nobreakdash-invariant.
Note that each of the $\lambda$\nobreakdash-invariants is a modular unit, see, \emph{e.g.}, Corollary~\ref{c:elliptic-S-units}$(ii)$ in Section~\ref{ss:cuspidal-values}.

The representation theory of the monster group provides a wealth of modular functions satisfying the hypotheses of Theorem~\ref{t:S-units-genus-0}.
In fact, by Borcherds' solution \cite[Theorem~1.1]{Bor92} of the monstrous moonshine conjecture of Conway and Norton \cite{ConNor79}, the McKay--Thompson series associated with a given element of the monster group is a \emph{Hauptmodul} defined over~$\Q$, see Proposition~\ref{p:arithmetically-modular} in Appendix~\ref{s:arithmetically-modular}.
Moreover, the singular moduli of a fundamental McKay--Thompson series are often algebraic integers \cite[Theorem~I]{CheYui96}.

For distinct singular moduli~$\sm$ and~$\sm'$ of the $j$\nobreakdash-invariant, Li gives in~\cite{Li21} an explicit lower bound for the absolute norm of~$\sm - \sm'$ that implies that this algebraic integer is not an algebraic unit.
When restricted to ${\sm' = 0}$, this is~\cite[Theorem~1.1]{BilHabKuh20}.
In fact, Li proves a stronger result for the values of modular polynomials at pairs of singular moduli of the $j$\nobreakdash-invariant.
Li's approach makes use of (extensions of) the work of Gross and Zagier in \cite{GroZag85}, and it is different from those in~\cite{Hab15,BilHabKuh20}.
Li does not treat the case of $S$\nobreakdash-units in~\cite{Li21}.

In the case where~$\sm_0 = 0$ (resp.~$1728$), the conclusion of the Main Theorem holds for certain classes of infinite sets of prime numbers~$S$.
In fact, if we put
\begin{multline*}
  S_0 \= \{ q \colon \text{ prime number, $q \equiv 1 \mod 3$} \}
  \\
  \text{ (resp. $\{ q \colon \text{ prime number, $q \equiv 1 \mod 4$} \}$),}
\end{multline*}
then the conclusion of the Main Theorem holds for every set of prime numbers~$S$ such that~$S \ssetminus S_0$ is finite and does not contain ${\{2, 3, 5 \}}$ (resp.~${\{2, 3, 7\}}$).
This is a direct consequence of the Main Theorem and the proof of \cite[Theorems~1.2 and~6.1]{Cam21}.

\subsection{Organization}
\label{ss:organization}

In Section~\ref{s:modular-functions} we establish general properties of modular functions (Section~\ref{ss:modular-functions}), their singular moduli (Section~\ref{ss:singular-moduli}) and their cuspidal and omitted values (Section~\ref{ss:cuspidal-values}).

In Section~\ref{s:nonatomic} we prove Theorem~\ref{t:disperse}.
We first establish it in the special case of the $j$\nobreakdash-invariant.
The main ingredient in the proof of this special case, is that no measure describing the $v$\nobreakdash-adic asymptotic distribution of the singular moduli of the $j$\nobreakdash-invariant has an atom in~$\Cv$.
This follows from \cite[\emph{Th{\'e}or{\`e}me}~2.4]{CloUll04} if~${v = \infty}$ and is stated as Theorem~\ref{t:CM-nonatomic} in the case where~$v$ is a prime number.
Together with~\cite[Theorems~A and~B]{HerMenRivII}, this implies Theorem~\ref{t:disperse} in the case of the $j$\nobreakdash-invariant as a direct consequence.
The proof of Theorem~\ref{t:CM-nonatomic} is based on the description of all these measures given in the companion papers~\cite{HerMenRiv20,HerMenRivII}.
We also use an analogous description for Hecke orbits given in \emph{loc.~cit.}
We first establish a result analogous to Theorem~\ref{t:CM-nonatomic} for Hecke orbits (Theorem~\ref{t:Hecke-orbits-nonatomic}) in Section~\ref{ss:Hecke-orbits-nonatomic}, and in Section~\ref{ss:CM-nonatomic} we deduce Theorem~\ref{t:CM-nonatomic} from this result.
To prove Theorem~\ref{t:Hecke-orbits-nonatomic}, we first show that the images of a point under Hecke correspondences associated with different prime numbers are nearly disjoint (Lemma~\ref{l:Hecke-independence}).
We use this to show that an atom in~$\Cp$ of an accumulation measure of a Hecke orbit would replicate indefinitely, thus creating infinite mass.\footnote{See Remark~\ref{r:homogeneous-limits} for a different strategy of proof.}
In Section~\ref{ss:proof-disperse} we deduce Theorem~\ref{t:disperse} in the general case from the special case of the $j$\nobreakdash-invariant, using the results about modular functions in Section~\ref{s:modular-functions}.

In Section~\ref{s:badly-approximable-Hauptmodul} we prove Theorem~\ref{t:badly-approximable-Hauptmodul}.
We first establish it in the special case of the $j$\nobreakdash-invariant, which is stated in a slightly different form as Proposition~\ref{p:badly-approximable-j}.
After a brief review of the work of Gross and Hopkins on deformation spaces of formal modules in Section~\ref{ss:deformations}, in Section~\ref{ss:proof-badly-approximable-j} we give the proof of Proposition~\ref{p:badly-approximable-j}.
First, we use that singular moduli are isolated in the ordinary reduction locus \cite[Corollary~B]{HerMenRiv20}, to restrict to the case where~$\sm$ and~$\sm_0$ are both in the supersingular reduction locus.
In the case where the conductors of~$D_{\sm}$ and~$D_{\sm_0}$ are both $p$\nobreakdash-adic units, we use an idea in the proof of \cite[Proposition~5.11]{Cha18}.
To extend this estimate to the general case, we use a formula in~\cite{HerMenRivII} that shows how the canonical branch of the Hecke correspondence~$T_p$ relates \CM{} points whose conductors differ by a power of~$p$.
In Section~\ref{ss:proof-badly-approximable-Hauptmodul} we deduce Theorem~\ref{t:badly-approximable-Hauptmodul} in the general case from the special case of the $j$\nobreakdash-invariant, using the results about modular functions in Section~\ref{s:modular-functions}.

In Section~\ref{s:proof-S-units} we prove Theorems~\ref{t:S-units-genus-0} and~\ref{t:units-are-rare}.
We follow Habegger's original strategy in the case of the $j$\nobreakdash-invariant and ${S = \emptyset}$ in~\cite{Hab15}, to prove a more general result that we state as Theorem~\ref{t:S-units-conditional}.
In particular, we use in a crucial way Colmez' bound~\cite[\emph{Th{\'e}or{\`e}me}~1]{Col98} in the form of~\cite[Lemma~3]{Hab15}.
The main new ingredient to implement Habegger's strategy is Theorem~\ref{t:disperse}.
Theorem~\ref{t:S-units-conditional} implies Theorem~\ref{t:units-are-rare} as a direct consequence.
Another direct consequence of Theorem~\ref{t:S-units-conditional} is that an affirmative solution to Conjecture~\ref{c:badly-approximable} would yield a version of Theorem~\ref{t:units-are-rare} for $S$\nobreakdash-units (Corollary~\ref{c:S-units-conditional-non-weak}) and a version of Theorem~\ref{t:S-units-genus-0} for a general congruence or genus zero group and a general algebraic value (Corollary~\ref{c:S-units-conditional-non-unit}).
We prove Theorem~\ref{t:S-units-conditional} in Section~\ref{ss:proof-S-units-conditional} and derive Theorem~\ref{t:S-units-genus-0} and Corollary~\ref{c:S-units-conditional-non-unit} from Theorem~\ref{t:S-units-conditional} in Section~\ref{ss:proof-S-units-genus-0}.

\subsection*{Acknowledgments}
The authors would like to thank Shouwu Zhang, for pointing out that a statement like the Main Theorem might be obtained from knowledge about the $p$\nobreakdash-adic asymptotic behavior of \CM{} points, and Philipp Habegger for sharing his conjecture on the $\lambda$\nobreakdash-invariants that prompted Theorem~\ref{t:S-units-genus-0}.
Finally, we thank the referee for insightful comments.

During the preparation of this work the first named author was partially supported by FONDECYT grant 11220567.
The second named author was partially supported by FONDECYT grant 1211858.
The third named author acknowledges partial support from NSF grant DMS-1700291.
The authors would like to thank the Pontificia Universidad Cat{\'o}lica de Valpara{\'{\i}}so, the University of Rochester and Universitat de Barcelona for hospitality during the preparation of this work.

\section{Modular functions and their special values}
\label{s:modular-functions}
In this section we prove general properties of modular functions, their singular moduli and their cuspidal and omitted values.
In Section~\ref{ss:modular-functions} we establish some general properties of modular functions (Proposition~\ref{p:modular-functions}).
In Section~\ref{ss:singular-moduli} we study arithmetic properties of singular moduli of modular functions defined over~$\Qalg$ (Proposition~\ref{p:singular-moduli}).
Finally, in Section~\ref{ss:cuspidal-values} we study cuspidal and omitted values of modular functions.

\subsection{Modular functions}
\label{ss:modular-functions}

The goal of this section is to prove the following proposition.

\begin{proposition}
  \label{p:modular-functions}
  Every modular function is algebraically dependent with the $j$\nobreakdash-invariant over~$\C$.
  Conversely, let~$K$ be a subfield of~$\C$ and let~$f$ be a nonconstant meromorphic function defined on~$\H$ that is algebraically dependent with the $j$\nobreakdash-invariant over~$K$.
  Then, $f$ is a modular function and there is a polynomial~$\Phi(X, Y)$ with coefficients in a finite extension of~$K$ inside~$\C$ that is irreducible over~$\C$ and such that~$\Phi(j, f)$ vanishes identically.
  Furthermore, $\Phi(X, Y)$ depends on both~$X$ and~$Y$, and it satisfies the following properties.
  \begin{enumerate}
  \item
    For every~$(z, w)$ in the zero set of~$\Phi$ in~$\C \times \C$, there is~$\tau$ in~$\H$ satisfying
    \begin{equation}
      \label{eq:6}
      z = j(\tau)
      \text{ and }
      w = f(\tau).
    \end{equation}
  \item
    Up to a constant factor, $\Phi$ is the unique irreducible polynomial in~$\C[X, Y]$ such that~$\Phi(j, f)$ vanishes identically.
  \end{enumerate}
\end{proposition}

In the proof of this proposition, which is below, and in the rest of the paper, we use the following property: For every subfield~$K$ of~$\C$, a polynomial in~$K[X, Y]$ is irreducible over~$\C$ if and only if it is irreducible over an algebraic closure of~$K$.

\begin{defi}
  \label{d:definition-field}
  Let~$K$ be a subfield of~$\C$.
  A modular function~$f$ is \emph{defined over~$K$}, if there is a polynomial~$\Phi(X, Y)$ in~$K[X, Y]$ that is irreducible over~$\C$ and such that~$\Phi(j, f)$ vanishes identically.
  In this case, $\Phi(X, Y)$ is \emph{a modular polynomial of~$f$}.
\end{defi}

In view of Proposition~\ref{p:modular-functions}, in the case where ${K = \Qalg}$ this definition coincides with the one given in Section~\ref{ss:S-units}.
Note that if~$K$ is a subfield of~$\C$, then every modular function having a modular polynomial in~$K[X, Y]$ is defined over~$K$.
On the other hand, by Proposition~\ref{p:modular-functions} for every modular function~$f$ there is a modular polynomial of~$j$ and~$f$ in~$\C[X, Y]$ and, if in addition~$f$ is algebraically dependent with the $j$\nobreakdash-invariant over~$K$, then there is a modular polynomial of~$j$ and~$f$ with coefficients in a finite extension of~$K$.
Furthermore, a modular polynomial in~$K[X, Y]$ of~$j$ and a modular function is unique up to a multiplicative constant in~$K^{\times}$.

\begin{proof}[Proof of Proposition~\ref{p:modular-functions}]
  Let~$f$ be a modular function.
  The case where~$f$ is constant being trivial, assume~$f$ is nonconstant.
  Let~$\Gamma$ be a subgroup of~$\SL(2, \R)$ that is commensurable to~$\SL(2, \Z)$ and such that~$f$ is invariant under~$\Gamma$.
  Replacing~$\Gamma$ by ${\Gamma \cap \SL(2, \Z)}$ if necessary, assume that~$\Gamma$ is a finite index subgroup of~$\SL(2, \Z)$.
  Then, the $j$\nobreakdash-invariant and~$f$ induce meromorphic functions~$j_0$ and~$f_0$ defined on~$X(\Gamma)$.
  Since the field of meromorphic functions defined on~$X(\Gamma)$ has transcendence degree one over~$\C$, there is a nonzero polynomial~$\Phi(X, Y)$ in~$\C[X, Y]$ such that~$\Phi(j_0, f_0)$ vanishes identically.
  It follows that the function~$\Phi(j, f)$ vanishes identically.
  This implies that the $j$\nobreakdash-invariant and~$f$ are algebraically dependent over~$\C$.

  To prove the second assertion, let~$K$ be a subfield of~$\C$ and let~$f$ be a nonconstant meromorphic function defined on~$\H$ that is algebraically dependent with the $j$\nobreakdash-invariant over~$K$.
  Then, there is a polynomial~$\Phi_0(X, Y)$ in~$K[X, Y]$ such that~$\Phi_0(j, f)$ vanishes identically.
  Suppose~$\Phi_0$ is not irreducible over~$\C$.
  Then, we can find a finite extension~$\hK$ of~$K$ inside~$\C$ and a finite family~$(\Phi_i)_{i \in I}$ of polynomials in~$\hK[X, Y]$ that are irreducible over~$\C$ and whose product is equal to~$\Phi_0$.
  It follows that at least one of the meromorphic functions in~$\{ \Phi_i(j, f) \colon i \in I \}$ vanishes identically.
  This proves that in all the cases there is a polynomial~$\Phi$ with coefficients in a finite extension of~$K$ that is irreducible over~$\C$ and such that~$\Phi(j, f)$ vanishes identically.
  Note also that, since the $j$\nobreakdash-invariant is nonconstant and~$f$ is nonconstant by assumption, the polynomial~$\Phi$ depends on both variables.

  To prove that~$f$ is a modular function, denote by~$\sM(\H)$ the field of all meromorphic functions defined on~$\H$.
  Note that the polynomial~$\Phi(j, Y)$ in~$\sM(\H)[Y]$ is nonconstant and denote by~$\sZ$ its finite number of zeros in~$\sM(\H)$.
  The set~$\sZ$ contains~$f$ and is invariant under the action of~$\SL(2, \Z)$ on~$\sM(\H)$ given by ${(\gamma, g) \mapsto g \circ \gamma}$.
  Since~$\sZ$ is finite, it follows that the stabilizer~$\Gamma$ of~$f$ in~$\SL(2, \R)$ is a finite index subgroup of~$\SL(2, \Z)$.
  Thus, to prove that~$f$ is a modular function, it is sufficient to prove that~$f$ is the lift of the restriction to ${\Gamma \backslash \H}$ of a meromorphic function defined on~$X(\Gamma)$.
  To do this, it is sufficient to show that for every~$\gamma$ in~$\SL(2, \Z)$ the function~$f \circ \gamma$ is meromorphic at~$i \infty$, see \emph{e.g.}, \cite[Proposition~1.30 and Section~1.4]{Shi71}.
  Note that the functions ${f \circ \gamma}$ and ${j \circ \gamma = j}$ are algebraically dependent over~$\C$.
  Thus, replacing~$f \circ \gamma$ by~$f$ if necessary, it is sufficient to prove that~$f$ is meromorphic at~$i \infty$.
  To do this, denote by~$\delta$ the degree of~$\Phi(X, Y)$ in~$Y$ and for each~$k$ in~$\{0, \ldots, \delta \}$ denote by~$P_k(X)$ the coefficient of~$Y^k$ in~$\Phi(X, Y)$ and by~$d_k$ the degree of~$P_k$.
  Then, there are constants ${R > 0}$ and~$C > 0$, such that for every~$z$ in~$\C$ satisfying ${|z| \ge R}$ we have~${|P_{\delta}(z)| \ge C^{-1} |z|^{d_{\delta}}}$, and such that every~$k$ in ${\{0, \ldots, \delta - 1 \}}$ we have ${|P_k(z)| \le C|z|^{d_k}}$.
  Thus, if we put
  \begin{equation}
    \label{eq:7}
    d
    \=
    - d_{\delta} + \max \{ d_1, \ldots, d_{\delta - 1} \},
  \end{equation}
  then for every~$\tau$ in~$\H$ satisfying
  \begin{equation}
    \label{eq:8}
    |j(\tau)| \ge R
    \text{ and }
    |f(\tau)| \ge 1,
  \end{equation}
  we have
  \begin{equation}
    \label{eq:9}
    |f(\tau)|^{\delta}
    =
    \left| P_{\delta} (j(\tau)) \right|^{-1} \left| \sum_{k = 0}^{\delta - 1} P_k (j(\tau)) f(\tau)^k \right|
    \le
    C^2 \delta |j(\tau)|^d |f(\tau)|^{\delta - 1}.
  \end{equation}
  Since the $j$\nobreakdash-invariant has a pole at~$i \infty$, see, \emph{e.g.}, \cite[Chapter~4, Section~1]{Lan87}, we conclude that for every~$\tau$ in~$\H$ whose imaginary part is sufficiently large we have ${|f(\tau)| \le C^2 \delta |j(\tau)|^d}$.
  This implies that~$f$ is meromorphic at~$i \infty$ and completes the proof that~$f$ is a modular function.

  To prove item~$(i)$, let~$j_0$ and~$f_0$ be as above, note that the set of poles of~$j_0$ is equal to the complement of~$\Gamma \backslash \H$ in~$X(\Gamma)$ and denote by~$P$ the set of poles of~$f_0$ in~$X(\Gamma)$.
  Moreover, denote by~$\hPhi$ the homogenization of~$\Phi$ in~$\C[X,Y,Z]$, so that ${\hPhi(X, Y, 1) = \Phi(X, Y)}$, and let~$Z(\hPhi)$ its zero set in~$\P^2(\C)$.
  Then,
  \begin{equation}
    \begin{array}{rcl}
      \label{eq:10}
      (\Gamma \backslash \H) \ssetminus P & \to & Z(\hPhi)
      \\
      x & \mapsto & \iota(x) \= [j_0(x) : f_0(x) : 1]
    \end{array}
  \end{equation}
  has a unique continuous extension ${\iota \colon X(\Gamma) \to Z(\hPhi)}$ and this map is surjective, see, \emph{e.g.}, \cite[Chapter~II, Proposition~6.8]{Har77}.
  Thus, for every~$(z, w)$ in~$\C \times \C$ in the zero set of~$\Phi$ there is~$x$ in ${(\Gamma \backslash \H) \ssetminus P}$ such that ${\psi(x) = [z : w : 1]}$.
  This implies item~$(i)$.

  To prove item~$(ii)$, let~$\widecheck{\Phi}(X, Y)$ be an irreducible polynomial in~$\C[X, Y]$ such that~$\widecheck{\Phi}(j, f)$ vanishes identically.
  Then, by item~$(i)$ the polynomial~$\widecheck{\Phi}$ vanishes on the zero set of~$\Phi$ and therefore~$\widecheck{\Phi}$ is a multiple of~$\Phi$.
  Since by assumption~$\widecheck{\Phi}$ is irreducible, it follows that it is a constant multiple of~$\Phi$.
  This completes the proof of item~$(ii)$ and of the proposition.
\end{proof}

\subsection{Singular moduli of modular functions}
\label{ss:singular-moduli}

The goal of this section is to establish some arithmetic properties of singular moduli of modular functions defined over~$\Qalg$.
These are gathered in Proposition~\ref{p:singular-moduli} below.
To state it, we introduce some terminology.

For a finite extension~$K$ of~$\Q$ inside~$\Qalg$, consider the Galois group~$\Gal(\Qalg|K)$ and for each~$\alpha$ in~$\C$ denote by~$\oO_K(\alpha)$ its orbit by~$\Gal(\Qalg|K)$.
Note that, with the notation introduced in Section~\ref{ss:disperse} we have ${\oO(\alpha) = \oO_{\Q}(\alpha)}$.

In this paper, a \emph{discriminant} is the discriminant of an order in a quadratic imaginary extension of~$\Q$.
For every discriminant~$D$ denote by~$h(D)$ the class number of the order of discriminant~$D$ in~$\Q(\sqrt{D})$.
For a singular modulus~$\sm$, the discriminant of the endomorphism ring of an elliptic curve over~$\C$ whose $j$\nobreakdash-invariant is equal to~$\sm$ only depends on~$\sm$.
Denote it by~$D_{\sm}$.
We use that for every singular modulus~$\sm$ of the $j$\nobreakdash-invariant, we have
\begin{align}
  \label{eq:11}
  \oO_{\Q}(\sm)
  & =
    \{ \text{singular modulus~$\sm'$ of the $j$\nobreakdash-invariant with ${D_{\sm'} = D_{\sm}}$} \}
    \intertext{ and }
    \label{eq:12}
    \# \oO_{\Q}(\sm)
  & =
    h(D_{\sm}),
\end{align}
see, \emph{e.g.}, \cite[Chapter~10, Theorem~5]{Lan87}.

\begin{proposition}
  \label{p:singular-moduli}
  Let~$K$ be a finite extension of~$\Q$ inside~$\C$.
  Then, for every nonconstant modular function~$f$ defined over~$K$ the following properties hold.
  \begin{enumerate}
  \item
    Every singular modulus~$\fsm_0$ of~$f$ is in~$\Qalg$ and every element of~$\oO_K(\fsm_0)$ is also a singular modulus of~$f$.
  \item
    There is a constant ${C_0 > 0}$ such that for every quadratic imaginary number~$\tau_0$ in~$\H$ that is not a pole of~$f$, the singular modulus~$f(\tau_0)$ of~$f$ satisfies
    \begin{equation}
      \label{eq:13}
      C_0^{-1} \cdot \# \oO_{\Q}(j(\tau_0))
      \le
      \# \oO_{K}(f(\tau_0))
      \le
      \# \oO_{\Q}(f(\tau_0))
      \le
      C_0 \cdot \# \oO_{\Q}(j(\tau_0)).
    \end{equation}
  \item
    For every ${\varepsilon > 0}$ there is a constant ${C_1 > 0}$ such that the following property holds.
    For every quadratic imaginary number~$\tau$ in~$\H$ that is not a pole of~$f$, we have
    \begin{equation}
      \label{eq:14}
      C_1^{-1} \cdot \# \oO_K(f(\tau))^{2 - \varepsilon}
      \le
      |D_{j(\tau)}|
      \le
      C_1 \cdot \# \oO_{K}(f(\tau))^{2 + \varepsilon}.
    \end{equation}
  \item
    For every ${\varepsilon > 0}$ there is a constant ${C_2 > 0}$, such that for every~$R > 0$ we have
    \begin{equation}
      \label{eq:15}
      \# \{ \text{singular modulus~$\fsm$ of~$f$ such that $\#\oO_{K}(\fsm) \le R$} \}
      \le
      C_2 R^{3 + \varepsilon}.
    \end{equation}
  \end{enumerate}
\end{proposition}

\begin{proof}
  Let~$\Phi(X, Y)$ be a modular polynomial of~$f$ in~$K[X, Y]$.
  Denote by~$\delta_X$ (resp.~$\delta_Y$) the degree of~$\Phi(X, Y)$ in~$X$ (resp.~$Y$).

  To prove items~$(i)$ and~$(ii)$, let~$\fsm_0$ be a singular modulus of~$f$ and let~$\tau_0$ be a quadratic imaginary number in~$\H$ such that ${\fsm_0 = f(\tau_0)}$.
  Then, ${\sm_0 \= j(\tau_0)}$ is a singular modulus of the $j$\nobreakdash-invariant and therefore it is in~$\Qalg$.
  On the other hand, the polynomial~$\Phi(\sm_0, Y)$ is nonzero because~$\Phi$ is irreducible over~$\C$ and~$j$ is nonconstant.
  Since~$\fsm_0$ is a root of~$\Phi(\sm_0, Y)$, it is in an extension of~$K(\sm_0)$ of degree at most~$\delta_Y$.
  In particular, $\fsm_0$ is in~$\Qalg$.
  To complete the proof of item~$(i)$, let~$\sigma$ in~$\Gal(\Qalg|K)$ be given, and note that ${\Phi(\sigma(\sm_0), \sigma(\fsm_0)) = 0}$.
  Since~$\Phi$ is irreducible over~$\C$, by Proposition~\ref{p:modular-functions} there is~$\tau$ in~$\H$ such that
  \begin{equation}
    \label{eq:16}
    \sigma(\sm_0) = j(\tau)
    \text{ and }
    \sigma(\fsm_0) = f(\tau).
  \end{equation}
  By~\eqref{eq:11}, the number~$\sigma(\sm_0)$ is a singular modulus of the $j$\nobreakdash-invariant and therefore~$\tau$ is a quadratic imaginary number.
  It follows that~$\sigma(\fsm_0)$ is a singular modulus of~$f$.
  This completes the proof of item~$(i)$.
  To prove item~$(ii)$, note that by~\eqref{eq:12} and the fact that~$\fsm_0$ is in an extension of~$K(\sm_0)$ of degree at most~$\delta_Y$ we have
  \begin{equation}
    \label{eq:17}
    \# \oO_{\Q}(\fsm_0)
    \le
    \delta_Y [K(\sm_0) : \Q]
    \le
    \delta_Y [K : \Q] \cdot \# \oO_{\Q}(\sm_0).
  \end{equation}
  On the other hand, the polynomial~$\Phi(X, \fsm_0)$ is nonzero because~$\Phi$ is irreducible over~$\C$ and~$f$ is nonconstant.
  Since~$\sm_0$ is a root of~$\Phi(X, \fsm_0)$, it is in an extension of~$K(\fsm_0)$ of degree at most~$\delta_X$, and we have
  \begin{equation}
    \label{eq:18}
    \# \oO_{\Q}(\sm_0)
    \le
    \delta_X [K(\fsm_0) : \Q]
    =
    \delta_X [K : \Q] \cdot \# \oO_{K}(\fsm_0).
  \end{equation}
  This completes the proof of item~$(ii)$ with ${C_0 = [K : \Q] \max \{ \delta_X, \delta_Y \}}$.

  Item~$(iii)$ is a direct consequence of \eqref{eq:11}, \eqref{eq:12}, item~$(ii)$ and of the following estimate: For every~$\varepsilon > 0$ there is~$C > 0$ such that for every discriminant~$D$, we have
  \begin{equation}
    \label{eq:19}
    C^{-1} |D|^{\frac{1}{2} - \varepsilon}
    \le
    h(D)
    \le
    C |D|^{\frac{1}{2} + \varepsilon}.
  \end{equation}
  In the case where~$D$ is fundamental this is Siegel's estimate \cite[(1)]{Sie35}.
  To deduce the general case from the fundamental case, see, \emph{e.g.}, \cite[Chapter~8, Section~1, Theorem~7]{Lan87} or \cite[(5.12) and Lemma~5.12]{HerMenRivII}.

  To prove item~$(iv)$, let~$C_0$ and~$C_1$ be the constants given by items~$(ii)$ and~$(iii)$, respectively.
  Moreover, for each singular modulus~$\fsm$ of~$f$ choose a quadratic imaginary number~$\tau$ in~$\H$ such that ${f(\tau) = \fsm}$, and put ${\sm(\fsm) \= j(\tau)}$.
  For every singular modulus~$\sm$ of the $j$\nobreakdash-invariant there are at most~$\delta_Y$ singular moduli~$\fsm$ of~$f$ such that ${\sm(\fsm) = \sm}$.
  Thus, by items~$(ii)$ and~$(iii)$ the left side of~\eqref{eq:15} is bounded from above by
  \begin{multline}
    \label{eq:20}
    \delta_Y \cdot \# \left\{ \text{singular modulus~$\sm$ of the $j$\nobreakdash-invariant such that}
    \right. \\ \left. \text{
        ${\# \oO_{\Q}(\sm) \le C_0 R}$ and ${|D_{\sm}| \le C_1 R^{2 + \varepsilon}}$} \right\}
    \le
    \delta_Y C_0 C_1 R^{3 + \varepsilon}.
  \end{multline}
  This proves item~$(iv)$, and completes the proof of the proposition.
\end{proof}

\subsection{Cuspidal and omitted values of modular functions}
\label{ss:cuspidal-values}

Let~$f$ be a nonconstant modular function.
A complex number~$\alpha$ is a \emph{value of~$f$} if there is~$\tau$ in~$\H$ such that ${f(\tau) = \alpha}$, and it is an \emph{omitted value of~$f$} if it is not a value of~$f$.

For a modular function~$f$, the following proposition gives a characterization of the cuspidal and omitted values of~$f$.
It shows in particular that every omitted value is cuspidal, see also Remark~\ref{r:cuspidal-vs-omitted}.
Moreover, in Proposition~\ref{p:cuspidal-are-approximated} below we show that in the case where~$f$ is defined over~$\Qalg$, every cuspidal value of~$f$ is well approximated in~$\C$ by the singular moduli of~$f$ and that for every prime number~$p$, every omitted value of~$f$ is badly approximable in~$\Cp$ by the singular moduli of~$f$.

\begin{proposition}
  \label{p:omitted-and-cuspidal}
  Let~$f$ be a nonconstant modular function and let~$\Phi(X, Y)$ be a modular polynomial of~$f$.
  Then, for every complex number~$\alpha$ the polynomial~$\Phi(X, \alpha)$ is nonzero and the following properties hold.
  \begin{enumerate}
  \item
    The number~$\alpha$ is an omitted value of~$f$ if and only if the polynomial~$\Phi(X, \alpha)$ is constant.
  \item
    The number~$\alpha$ is a cuspidal value of~$f$ if and only if the degree of the polynomial~$\Phi(X, \alpha)$ is strictly smaller than the degree of~$\Phi(X, Y)$ in~$X$.
  \end{enumerate}
  In particular, every omitted value is cuspidal.
  Furthermore, if~$f$ is defined over a subfield~$K$ of~$\C$, then every cuspidal value of~$f$ is in the algebraic closure of~$K$ inside~$\C$.
\end{proposition}

The following corollary is an immediate consequence of this proposition.
Note that a modular function~$f$ is holomorphic if and only if~$0$ is an omitted value of~$\frac{1}{f}$, and~$f$ is a modular unit if and only if~$0$ is an omitted value of~$f$ and of~$\frac{1}{f}$.

\begin{coro}
  \label{c:elliptic-S-units}
  Let~$f$ and~$\Phi$ be as in Proposition~\ref{p:omitted-and-cuspidal}.
  If we consider~$\Phi(X, Y)$ as a polynomial in~$Y$ with coefficients in~$\C[X]$, then the following properties hold.
  \begin{enumerate}
  \item
    The modular function~$f$ is holomorphic if and only if the leading coefficient of~$\Phi(X, Y)$ does not depend on~$X$.
    In particular, for every holomorphic modular function~$f$ defined over~$\Qalg$, there is a finite set of prime numbers~$S$ such that every singular modulus of~$f$ is an $S$\nobreakdash-integer.
  \item
    The modular function~$f$ is a modular unit if and only if neither the constant nor the leading coefficients of~$\Phi(X, Y)$ depend on~$X$.
    In particular, for every modular unit~$f$ defined over~$\Qalg$, there is a finite set of prime numbers~$S$ such that every singular modulus of~$f$ is an $S$\nobreakdash-unit.
  \end{enumerate}
\end{coro}

\begin{proof}[Proof of Proposition~\ref{p:omitted-and-cuspidal}]
  If the polynomial~$\Phi(X, \alpha)$ were zero, then~$\Phi(X, Y)$ would be divisible by ${Y - \alpha}$.
  This is impossible since~$\Phi(X, Y)$ is irreducible in~$\C[X, Y]$ and it depends on both variables (Proposition~\ref{p:modular-functions}).
  This proves that~$\Phi(X, \alpha)$ is nonzero.

  To prove item~$(i)$, let~$\alpha$ be a complex number such that~$\Phi(X, \alpha)$ is nonconstant and let~$\beta$ be a root of this polynomial.
  Then, by Proposition~\ref{p:modular-functions}$(i)$ there is~$\tau$ in~$\H$ such that ${j(\tau) = \beta}$ and ${f(\tau) = \alpha}$.
  In particular, $\alpha$ is a value of~$f$ and therefore it is not an omitted value of~$f$.
  To prove the reverse implication, let~$\tau$ in~$\H$ be such that~$f(\tau)$ is finite.
  Then, the number~$j(\tau)$ is a zero of the polynomial~$\Phi(X, f(\tau))$.
  Since~$\Phi(X, f(\tau))$ is nonzero, it follows that it is nonconstant.
  This completes the proof of item~$(i)$.

  To prove item~$(ii)$, denote by~$d$ the degree of~$\Phi(X, Y)$ in~$X$ and let~$P(Y)$ be the coefficient of~$X^d$ in~$\Phi(X, Y)$, seen as a polynomial in~$X$ with coefficients in~$\C[Y]$.
  Furthermore, put
  \begin{equation}
    \label{eq:21}
    \Delta(X, Y)
    \=
    P(Y) X^d - \Phi(X, Y)
  \end{equation}
  and note that the degree in~$X$ of this polynomial is strictly less than~$d$.
  Let~$\Gamma$ be the stabilizer of~$f$ in~$\SL(2, \R)$ and denote by~$f_0$ the meromorphic function defined on~$X(\Gamma)$ induced by~$f$.
  Suppose that~$\alpha$ is a cuspidal value of~$f$.
  That is, $\alpha$ is a value that~$f_0$ takes at a point in ${\Gamma \backslash \P^1(\Q)}$.
  Since~$\SL(2, \Z)$ acts transitively on~$\P^1(\Q)$, there is~$\gamma$ in~$\SL(2, \Z)$ such that ${f \circ \gamma(\tau) \to \alpha}$ as ${\Im(\tau) \to \infty}$, see, \emph{e.g.} \cite[Proposition~1.30]{Shi71}.
  Combined with~\eqref{eq:21}, this implies
  \begin{equation}
    \label{eq:22}
    |P(\alpha)|
    =
    \lim_{\Im(\tau) \to \infty} |P((f \circ \gamma)(\tau))|
    =
    \lim_{\Im(\tau) \to \infty} \frac{|\Delta(j(\tau), (f \circ \gamma)(\tau))|}{|j(\tau)|^d}
    =
    0.
  \end{equation}
  This proves ${P(\alpha) = 0}$ and therefore that the degree of~$\Phi(X, \alpha)$ is strictly less than~$d$.
  To prove the reverse implication, suppose that~$\alpha$ is a non-cuspidal value of~$f$ and let~$A$ be a finite subset of~$\H$ such that ${f_0^{-1}(\alpha) = \Gamma \backslash (\Gamma \cdot A)}$.
  Let ${r > 0}$ be sufficiently small so that there is a compact neighborhood~$N$ of~$A$ in~$\H$ such that
  \begin{equation}
    \label{eq:23}
    f_0^{-1} \left(\overline{\bfD_{\infty}(\alpha, r)} \right) = \Gamma \backslash (\Gamma \cdot N).
  \end{equation}
  Reducing~$r$ if necessary, suppose that for every~$\alpha'$ in ${B(\alpha, r) \ssetminus \{ \alpha \}}$ we have ${P(\alpha') \neq 0}$ and let~$(\alpha_i)_{i = 1}^{\infty}$ be a sequence in ${B(\alpha, r) \ssetminus \{ \alpha \}}$ converging to~$\alpha$.
  Then, for every~$i$ the polynomial~$\Phi(X, \alpha_i)$ is of degree~$d$ and therefore by Proposition~\ref{p:modular-functions}$(i)$ there are~$\tau_i^{(1)}$, \ldots, $\tau_i^{(d)}$ in~$N$ such that
  \begin{equation}
    \label{eq:24}
    \Phi(X, \alpha_i)
    =
    P(\alpha_i) \prod_{\ell = 1}^d (X - j(\tau_i^{(\ell)})).
  \end{equation}
  Taking a subsequence if necessary, suppose that for every~$\ell$ in ${\{ 1, \ldots, d \}}$ the sequence~$(\tau_i^{(\ell)})_{i = 1}^{\infty}$ converges to an element~$\tau_i$ of~$N$.
  Letting ${i \to \infty}$ in~\eqref{eq:24}, we obtain
  \begin{equation}
    \label{eq:25}
    \Phi(X, \alpha)
    =
    P(\alpha) \prod_{\ell = 1}^d (X - j(\tau^{(\ell)})).
  \end{equation}
  Since~$\Phi(X, \alpha)$ is nonzero, it follows that~$P(\alpha)$ is nonzero and therefore that the degree of~$\Phi(X, \alpha)$ is~$d$.
  This completes the proof of item~$(ii)$.

  To prove the remaining assertions, note that by combining items~$(i)$ and~$(ii)$ we obtain that every omitted value is cuspidal.
  On the other hand, by item~$(ii)$ the cuspidal values of~$f$ are precisely the zeros of~$P(Y)$.
  In particular, there are at most finitely many cuspidal values of~$f$.
  If~$f$ is defined over a subfield~$K$ of~$\C$, then we can assume that the polynomial~$\Phi(X, Y)$ is in~$K[X, Y]$.
  This implies that~$P(Y)$ is in~$K[Y]$ and therefore that all of the cuspidal values of~$f$ are in the algebraic closure of~$K$ inside~$\C$.
\end{proof}

\begin{remark}
  \label{r:cuspidal-vs-omitted}
  The modular function
  \begin{equation}
    \label{eq:26}
    g
    \=
    \frac{j}{j^2 - 1}
  \end{equation}
  provides an example of a cuspidal value that is not omitted.
  In fact, this function is invariant under~$\SL(2, \Z)$ and the meromorphic function~$g_0$ on~$X(\SL(2, \Z))$ induced by~$g$ vanishes at the cusp~$i \infty$.
  But~$0$ is not an omitted value of~$g$, because ${g \left( \frac{1 + \sqrt{3}i}{2} \right) = 0}$.
\end{remark}

\begin{proposition}
  \label{p:cuspidal-are-approximated}
  For every nonconstant modular function~$f$ defined over~$\Qalg$, the following properties hold.
  \begin{enumerate}
  \item
    Every cuspidal value of~$f$ is well approximated in~$\C$ by the singular moduli of~$f$.
    In particular, every omitted value of~$f$ is well approximated in~$\C$ by the singular moduli of~$f$.
  \item
    Let~$p$ be a prime number and let~$\alpha$ be an omitted value of~$f$.
    Then, there is ${r > 0}$ such that~$\bfD_p(\alpha, r)$ contains no singular modulus of~$f$.
    In particular, $\alpha$ is badly approximable in~$\Cp$ by the singular moduli of~$f$.
  \end{enumerate}
\end{proposition}

The proof of this proposition is after the following lemma.

\begin{lemma}
  \label{l:reduction-to-j}
  Let~$v$ be in~$\VQ$ and let~$\Phi(X, Y)$ be an irreducible polynomial in~$\Cv[X, Y]$ depending on both variables.
  Then, for every~$\alpha$ in~$\Cv$ there are constants
  \begin{equation*}
    C_3 > 1,
    \theta > 0,
    \eta > 0
    \text{ and }
    \eta' > 0
  \end{equation*}
  such that for every~$z$ in~$\Cv$ and every~$w$ in ${\Cv \ssetminus \{ \alpha \}}$ sufficiently close to~$\alpha$ and such that ${\Phi(z, w) = 0}$, exactly one of the following properties holds.
  \begin{enumerate}
  \item
    The polynomial~$\Phi(X, \alpha)$ is nonconstant and, denoting by~$Z$ its finite set of zeros in~$\Cv$, we have
    \begin{equation}
      \label{eq:27}
      \min \{ |z - z_0|_v \colon z_0 \in Z \}
      <
      C_3 |w - \alpha|_v^{\theta}.
    \end{equation}
  \item
    The degree of~$\Phi(X, \alpha)$ is strictly smaller than that of~$\Phi(X, Y)$ in~$X$ and we have
    \begin{equation}
      \label{eq:28}
      C_3^{-1} |w - \alpha|_v^{- \eta}
      <
      |z|_v
      <
      C_3 |w - \alpha|_v^{- \eta'}.
    \end{equation}
  \end{enumerate}
\end{lemma}

\begin{proof}
  Put ${Q_0(X) \= \Phi(X, \alpha)}$ and note that our hypotheses that~$\Phi(X, Y)$ is irreducible in~$\Cv[X, Y]$ and that it depends on both variables, implies that~$Q_0(X)$ is nonzero.
  Denote by~$\ell_0$ the degree of~$Q_0(X)$ and let ${R_0 > 1}$ and~$M_0$ in~$]0, 1[$ be constants so that for every~$z$ in~$\Cv$ satisfying ${|z|_v \ge R_0}$, we have
  \begin{equation}
    \label{eq:29}
    M_0 |z|_v^{\ell_0}
    \le
    |Q_0(z)|_v
    \le
    M_0^{-1} |z|_v^{\ell_0}.
  \end{equation}
  Reducing~$M_0$ if necessary, suppose that in the case where~$Q_0(X)$ is constant we have
  \begin{equation}
    \label{eq:30}
    |Q_0(0)|_v
    \ge
    M_0,
  \end{equation}
  and that in the case where~$Q_0(X)$ is nonconstant for every~$z$ in~$\Cv$ we have
  \begin{equation}
    \label{eq:31}
    |Q_0(z)|_v
    \ge
    M_0 \min \{ |z - z_0|_v \colon z_0 \in Z \}^{\ell_0}.
  \end{equation}

  Note that ${Y - \alpha}$ divides ${\Phi(X, Y) - Q_0(X)}$.
  Let~$m_0$ in~$\N$ be the largest integer such that~$(Y - \alpha)^{m_0}$ divides ${\Phi(X, Y) - Q_0(X)}$, and let~$\Psi(X, Y)$ be the polynomial in~$\Cv[X, Y]$ such that
  \begin{equation}
    \label{eq:32}
    \Phi(X, Y) - Q_0(X)
    =
    (Y - \alpha)^{m_0} \Psi(X, Y).
  \end{equation}
  Denote by~$\delta$ the degree of~$\Psi(X, Y)$ in~$X$.
  Regarding~$\Psi(X, Y)$ as a polynomial in~$X$ with coefficients in~$\Cv[Y]$, for each~$i$ in~$\{0, \ldots, \delta \}$ let~$P_i(Y)$ be the coefficient of~$X^i$ in~$\Psi(X, Y)$.
  Furthermore, denote by~$m_1$ the order of~$P_\delta(Y)$ at~$\alpha$.
  Then, there is a constant~$M_1 > 1$ such that for every~$w$ in ${\Cv \ssetminus \{ \alpha \}}$ that is sufficiently close to~$\alpha$, we have
  \begin{equation}
    \label{eq:33}
    |P_{\delta}(w)|_v > M_1^{-1} |w - \alpha|_v^{m_1},
  \end{equation}
  and such that for every~$i$ in~$\{0, \ldots, \delta \}$ we have ${|P_i(w)|_v \le M_1}$.
  Thus, for every~$z$ in~$\Cv$ such that ${\Phi(z, w) = 0}$, we have
  \begin{equation}
    \label{eq:34}
    |\Psi(z, w)|_v
    \le
    (\delta + 1) M_1 \max \{ 1, |z|_v \}^{\delta}
  \end{equation}
  and
  \begin{equation}
    \label{eq:35}
    |\Psi(z, w) - P_{\delta}(w) z^{\delta}|_v
    \le
    \delta M_1 \max \{ 1, |z|_v \}^{\delta - 1}.
  \end{equation}

  To prove the desired assertion, put
  \begin{equation}
    \label{eq:36}
    M_2
    \=
    \frac{M_0}{(\delta + 1) M_1},
  \end{equation}
  and let~$w$ in ${\Cv \ssetminus \{ \alpha \}}$ be sufficiently close to~$\alpha$ so that~\eqref{eq:33}, \eqref{eq:34} and~\eqref{eq:35} hold and so that
  \begin{equation}
    \label{eq:37}
    |w - \alpha|_v^{m_0}
    <
    M_2 R_0^{-\delta}.
  \end{equation}
  Furthermore, let~$z$ in~$\Cv$ be such that ${\Phi(z, w) = 0}$.

  \partn{Case~1} ${|z|_v < R_0}$.
  If~$Q_0(X)$ were constant, then by~\eqref{eq:30}, \eqref{eq:32} and~\eqref{eq:34} we would have
  \begin{equation}
    \label{eq:38}
    |w - \alpha|_v^{m_0}
    =
    \left| \frac{Q_0(z)}{\Psi(z, w)} \right|_v
    >
    M_2 R_0^{-\delta},
  \end{equation}
  which contradicts~\eqref{eq:37} and proves that~$Q_0(X)$ is nonconstant.
  Denoting by~$Z$ the nonempty set of zeros of~$Q_0(X)$ in~$\Cv$, by~\eqref{eq:31}, \eqref{eq:32} and~\eqref{eq:34} we have
  \begin{equation}
    \label{eq:39}
    |w - \alpha|_v^{m_0}
    =
    \left| \frac{Q_0(z)}{\Psi(z, w)} \right|_v
    >
    M_2 R_0^{- \delta} \min \{ |z - z_0|_v \colon z_0 \in Z \}^{\ell_0}.
  \end{equation}
  This proves~\eqref{eq:27} with ${C_3 = M_2^{- \frac{1}{\ell_0}} R_0^{\frac{\delta}{\ell_0}}}$ and ${\theta = \frac{m_0}{\ell_0}}$ and completes the proof that property~$(i)$ holds.

  \partn{Case~2} ${|z|_v \ge R_0}$.
  By~\eqref{eq:29}, \eqref{eq:32} and~\eqref{eq:34}, in this case we have
  \begin{multline}
    \label{eq:40}
    M_0 |z|_v^{\ell_0} \cdot |w - \alpha|_v^{- m_0}
    \le
    |Q_0(z)| \cdot |w - \alpha|_v^{- m_0}
    =
    |\Psi(z, w)|_v
    \\ \le
    (\delta + 1) M_1 |z|_v^{\delta}.
  \end{multline}
  If we had ${\ell_0 \ge \delta}$, then we would obtain ${|w - \alpha|_v^{m_0} \ge M_2}$.
  This contradicts~\eqref{eq:37} and proves that the degree~$\ell_0$ of~$Q_0(X)$ is strictly less than the degree~$\delta$ of~$\Phi(X, Y)$ in~$X$.
  Together with~\eqref{eq:40}, this implies the first inequality in~\eqref{eq:28} with ${C_3 = M_2^{\frac{1}{\delta - \ell_0}}}$ and ${\eta = \frac{m_0}{\delta - \ell_0}}$.
  To prove the second inequality in~\eqref{eq:28}, suppose
  \begin{equation}
    \label{eq:41}
    |z|_v
    \ge
    2 \delta M_1^2 |w - \alpha|_v^{- m_1}.
  \end{equation}
  Then, by~\eqref{eq:33} and~\eqref{eq:35} we have
  \begin{multline}
    \label{eq:42}
    |P_{\delta}(w) z^{\delta}|_v
    \ge
    2 \delta M_1^2 |w - \alpha|_v^{- m_1} |P_{\delta}(w) z^{\delta - 1}|_v
    >
    2 \delta M_1 |z|_v^{\delta - 1}
    \\ \ge
    2 |\Psi(z, w) - P_{\delta}(w) z^{\delta}|_v.
  \end{multline}
  Together with the triangle inequality, \eqref{eq:29}, \eqref{eq:32} and~\eqref{eq:33}, this implies
  \begin{multline}
    \label{eq:43}
    M_0^{-1} |z|_v^{\ell_0} \cdot |w - \alpha|_v^{- m_0}
    \ge
    |Q_0(z)| \cdot |w - \alpha|_v^{- m_0}
    =
    |\Psi(z, w)|_v
    >
    \frac{1}{2} |P_\delta(w) z^{\delta}|_v
    \\ \ge
    \frac{1}{2} M_1^{-1} |w - \alpha|_v^{m_1} |z|_v^{\delta}.
  \end{multline}
  Rearranging, we obtain the second inequality in~\eqref{eq:28} with
  \begin{equation}
    \label{eq:44}
    C_3
    =
    \max \left\{ 2\delta M_1^2, (2 M_0^{-1} M_1)^{\frac{1}{\delta - \ell_0}} \right\}
    \text{ and }
    \eta'
    =
    \max \left\{ m_1, \frac{m_0 + m_1}{\delta - \ell_0} \right\}.
  \end{equation}
  This completes the proof that property~$(ii)$ holds.

  Finally, note that for~$w$ in ${\Cv \ssetminus \{ \alpha \}}$ that is sufficiently close to~$\alpha$, the inequality~\eqref{eq:27} and the first inequality in~\eqref{eq:28} cannot hold at the same time.
  This proves that properties~$(i)$ and~$(ii)$ cannot hold simultaneously and completes the proof of the lemma.
\end{proof}

\begin{proof}[Proof of Proposition~\ref{p:cuspidal-are-approximated}]
  Let~$\Phi(X, Y)$ be a modular polynomial of~$f$ in~$\Qalg[X, Y]$ and note that for every~$v$ in~$\VQ$ the polynomial~$\Phi(X, Y)$ is irreducible in~$\C_v[X, Y]$.

  To prove item~$(i)$, let~$\alpha$ be a cuspidal value of~$f$ and let~$C_3$ and~$\eta'$ be the constants given by Lemma~\ref{l:reduction-to-j} with ${v = \infty}$.
  That is, if we denote by~$\Gamma$ the stabilizer of~$f$ in~$\SL(2, \R)$, then~$\alpha$ is a value that~$f$ takes at a point in ${\Gamma \backslash \P^1(\Q)}$.
  Since~$\SL(2, \Z)$ acts transitively on~$\P^1(\Q)$, there is~$\gamma$ in~$\SL(2, \Z)$ such that ${f \circ \gamma(\tau) \to \alpha}$ as ${\Im(\tau) \to \infty}$, see, \emph{e.g.} \cite[Proposition~1.30]{Shi71}.
  Let~$C > 0$ be a constant such that for every~$\tau$ in~$\H$ such that~$\Im(\tau)$ is sufficiently large, we have
  \begin{equation}
    \label{eq:45}
    |j(\tau)|
    \ge
    C \exp(2 \pi \Im(\tau)),
  \end{equation}
  see, \emph{e.g.}, \cite[Chapter~4, Section~1]{Lan87}.
  Given a prime number~$p'$ satisfying ${p' \equiv 1 \mod 4}$, put
  \begin{equation}
    \label{eq:46}
    \tau_{p'}
    \=
    i \sqrt{p'},
    \sm(p')
    \= j(\tau_{p'})
    \text{ and }
    \fsm(p')
    \= f \circ \gamma (\tau_{p'}),
  \end{equation}
  and note that~$\sm(p')$ is a singular modulus of the $j$\nobreakdash-invariant satisfying ${D_{\sm(p')} = - 4p'}$ and that~$\fsm(p')$ is a singular modulus of~$f$.
  If~$p'$ is sufficiently large, then by~\eqref{eq:45} with ${\tau = \tau_{p'}}$ property~$(i)$ in Lemma~\ref{l:reduction-to-j} cannot be satisfied with ${z = \sm(p')}$ and ${w = \fsm(p')}$.
  So, property~$(ii)$ holds and we have
  \begin{equation}
    \label{eq:47}
    - \log |\fsm(p') - \alpha|
    >
    \frac{1}{\eta'} \log |\sm(p')| - \frac{1}{\eta'} \log C_3
    \ge
    \frac{\pi}{\eta'} \sqrt{|D_{\sm(p')}|} - \frac{1}{\eta'} \log \frac{C_3}{C}.
  \end{equation}
  In view of Proposition~\ref{p:singular-moduli}$(iii)$, this implies that~$\alpha$ is well approximated in~$\C$ by the singular moduli of~$f$.
  The second assertion of item~$(i)$ follows from the first and from the fact that every omitted value is cuspidal (Proposition~\ref{p:omitted-and-cuspidal}).

  To prove item~$(ii)$, let~$C_3$ and~$\eta$ be the constants given by Lemma~\ref{l:reduction-to-j} with ${v = p}$ and put ${r \= C_3^{- \frac{1}{\eta}}}$.
  By Proposition~\ref{p:omitted-and-cuspidal}$(ii)$, our hypothesis that~$\alpha$ is an omitted value of~$f$ implies that the polynomial~$\Phi(X, \alpha)$ is constant.
  Thus, if there were a quadratic imaginary number~$\tau$ in~$\H$ such that~$f(\tau)$ is sufficiently close to~$\alpha$ in~$\Cp$, then~$f(\tau)$ would be in~$\bfD_p(\alpha, r)$ and by Lemma~\ref{l:reduction-to-j} the singular modulus~$j(\tau)$ of the $j$\nobreakdash-invariant would satisfy
  \begin{equation}
    \label{eq:48}
    |j(\tau)|_p
    >
    C_3^{-1} |f(\tau) - \alpha|_p^{-\eta}
    >
    1.
  \end{equation}
  This is absurd, since~$j(\tau)$ is an algebraic integer.
  This completes the proof of item~$(ii)$ and of the proposition.
\end{proof}

\section{$p$-Adic limits of \CM{} points}
\label{s:nonatomic}

The goal of this section is to prove Theorem~\ref{t:disperse}.
The main ingredient is Theorem~\ref{t:CM-nonatomic} below.
Together with \cite[Theorems~A and~B]{HerMenRivII}, which are summarized in Theorem~\ref{t:CM} in Section~\ref{ss:CM-nonatomic}, Theorem~\ref{t:CM-nonatomic} implies Theorem~\ref{t:disperse} in the case of the $j$\nobreakdash-invariant as a direct consequence.
The general case is deduced from this special case in Section~\ref{ss:proof-disperse}.

Throughout this section, fix a prime number~$p$ and let ${(\Cp, | \cdot |_p)}$ be as in the introduction.
Denote by~$\Ell(\Cp)$ the coarse moduli space of elliptic curves over~$\Cp$.
We consider~$\Ell(\Cp)$ as a subspace of the Berkovich affine line~$\AKber$ over~$\Cp$, using the $j$\nobreakdash-invariant to identify~$\Ell(\Cp)$ with the subspace~$\Cp$ of~$\AKber$.
We endow the space of Borel measures on~$\AKber$ with the weak topology with respect to the space of bounded and continuous real functions.
Denote by~$\xcan$ the ``Gauss'' or ``canonical'' point of~$\AKber$.
For~$x$ in~$\AKber$ denote by~$\delta_x$ the Dirac measure at~$x$.
An \emph{atom} of a Borel measure~$\nu$ on~$\AKber$ is a point~$x$ in~$\AKber$ such that ${\nu(\{ x \}) > 0}$.
A measure is \emph{nonatomic} if it has no atoms.

The endomorphism ring of an elliptic curve over~$\Cp$ only depends on the corresponding class~$E$ in~$\Ell(\Cp)$ of the elliptic curve.
It is isomorphic to~$\Z$ or to an order in a quadratic imaginary extension of~$\Q$.
In the latter case~$E$ is a \emph{\CM{} point} and its \emph{discriminant} is the discriminant of its endomorphism ring.
For every discriminant~$D$, the set
\begin{equation}
  \label{eq:49}
  \Lambda_D
  \=
  \{ E \in \Ell(\Cp) \colon \text{\CM{} point of discriminant~$D$} \}
\end{equation}
is finite and nonempty.
Denote by~$\odelta_{D, p}$ the Borel probability measure on~$\Ell(\Cp)$, defined by
\begin{equation}
  \label{eq:50}
  \odelta_{D, p}
  \=
  \frac{1}{\# \Lambda_D} \sum_{E \in \Lambda_D} \delta_E.
\end{equation}
In contrast to the complex case, as the discriminant~$D$ tends to~$- \infty$ the measure~$\odelta_{D, p}$ does not converge in the weak topology.
In fact, there are infinitely many different accumulation measures \cite[Corollary~1.1]{HerMenRivII}.

\begin{theorem}
  \label{t:CM-nonatomic}
  Let~$p$ be a prime number.
  Then every accumulation measure of
  \begin{equation}
    \label{eq:51}
    \left\{ \odelta_{D, p} \colon D \text{ discriminant} \right\}
  \end{equation}
  in the weak topology that is different from~$\delta_{\xcan}$, is nonatomic.
  In particular, no accumulation measure of~\eqref{eq:51} in the weak topology has an atom in~$\Ell(\Cp)$.
\end{theorem}

One of the main ingredients in the proof of this result is the description of all accumulation measures of~\eqref{eq:51} given in the companion papers~\cite{HerMenRiv20,HerMenRivII}.
We also use an analogous description for Hecke orbits given in \emph{loc.~cit.}
We first establish a result analogous to Theorem~\ref{t:CM-nonatomic} for Hecke orbits (Theorem~\ref{t:Hecke-orbits-nonatomic}) in Section~\ref{ss:Hecke-orbits-nonatomic}, and in Section~\ref{ss:CM-nonatomic} we deduce Theorem~\ref{t:CM-nonatomic} from this result.

Denote by~$\Qpalg$ the algebraic closure of~$\Qp$ inside~$\Cp$, and by~$\Op$ and~$\OQpalg$ the ring of integers of~$\Cp$ and~$\Qpalg$, respectively.
For~$E$ in~$\Ell(\Cp)$ represented by a Weierstrass equation with coefficients in~$\OQpalg$ having smooth reduction, denote by~$\FE$ the formal group of~$E$ and by~$\End(\FE)$ the ring of endomorphisms of~$\FE$ that are defined over the ring of integers of a finite extension of~$\Qp$.
Then~$\End(\FE)$ is either isomorphic to~$\Z_p$, or to a $p$\nobreakdash-adic quadratic order, see, \emph{e.g.}, \cite[Chapter~IV, Section~1, Theorem~1$(iii)$]{Fro68}.
In the latter case, $E$ is said to have \emph{formal complex multiplication} or to be a \emph{formal \CM{} point}.

An elliptic curve class~$E$ in~$\Ell(\Cp)$ has \emph{supersingular reduction}, if there is a representative elliptic curve over~$\Op$ whose reduction is smooth and supersingular.
Denote by~$\Sups$ the set of all elliptic curve classes in~$\Ell(\Cp)$ with supersingular reduction.

\subsection{On the limit measures of Hecke orbits}
\label{ss:Hecke-orbits-nonatomic}
The goal of this section is to prove Theorem~\ref{t:Hecke-orbits-nonatomic} below, which is the main ingredient in the proof of Theorem~\ref{t:CM-nonatomic}.
To state it, we introduce some notation.

A \emph{divisor on~$\Ell(\Cp)$} is an element of the free abelian group
\begin{displaymath}
  \Div(\Ell(\Cp))
  \=
  \bigoplus_{E\in \Ell(\Cp)} \Z E.
\end{displaymath}
For a divisor ${\cD = \sum_{E \in \Ell(\Cp)} n_EE}$ in~$\Div(\Ell(\Cp))$, the \emph{degree} and \emph{support} of~$\cD$ are defined by
\begin{displaymath}
  \deg(\cD)
  \=
  \sum_{E \in \Ell(\Cp)} n_E
  \text{ and }
  \supp(\cD)
  \=
  \{ E \in \Ell(\Cp) \colon n_E \neq 0 \},
\end{displaymath}
respectively.
If in addition ${\deg(\cD) \ge 1}$ and for every~$E$ in~$\Ell(\Cp)$ we have ${n_E \ge 0}$, then
\begin{displaymath}
  \odelta_{\cD, p}
  \=
  \frac{1}{\deg(\cD)}\sum_{E \in \Ell(\Cp)} n_E \delta_E
\end{displaymath}
is a Borel probability measure on~$\Ell(\Cp)$.

For~$n$ in~$\N$, the $n$\nobreakdash-th \emph{Hecke correspondence} is the linear map
\begin{displaymath}
  T_n \colon \Div(\Ell(\Cp)) \to \Div(\Ell(\Cp))
\end{displaymath}
defined for~$E$ in~$\Ell(\Cp)$ by
\begin{displaymath}
  T_n(E)
  \=
  \sum_{C\leq E \text{ of order } n}E/C,
\end{displaymath}
where the sum runs over all subgroups~$C$ of~$E$ of order~$n$.
For background on Hecke correspondences, see \cite[Sections~7.2 and~7.3]{Shi71} for the general theory, or the survey~\cite[Part~II]{DiaIm95}.

\begin{theorem}
  \label{t:Hecke-orbits-nonatomic}
  For each~$E$ in~$\Ell(\Cp)$, every accumulation measure of~$(\odelta_{T_n(E), p})_{n = 1}^{\infty}$ in the weak topology that is different from~$\delta_{\xcan}$, is nonatomic.
  In particular, no accumulation measure of~$(\odelta_{T_n(E), p})_{n = 1}^{\infty}$ in the weak topology has an atom in~$\Ell(\Cp)$.
\end{theorem}

To prove Theorem~\ref{t:Hecke-orbits-nonatomic}, we first recall some results in~\cite{HerMenRivII}.
For~$E$ in~$\Sups$, define a subgroup~$\NE$ of~$\Z_p^{\times}$ as follows.
If~$E$ is not a formal \CM{} point, then ${\NE \= (\Z_p^{\times})^2}$.
In the case where~$E$ is a formal \CM{} point, denote by~$\Aut(\FE)$ the group of isomorphisms of~$\FE$ defined over~$\OQpalg$, and by~$\nr$ the norm map of the field of fractions of~$\End(\FE)$ to~$\Qp$.
Then,
\begin{displaymath}
  \NE
  \=
  \left\{ \nr \left( \varphi \right) \colon \varphi \in \Aut(\FE) \right\}.
\end{displaymath}
In all the cases~$\NE$ is a multiplicative subgroup of~$\Z_p^{\times}$ containing~$(\Z_p^{\times})^2$.
In particular, the index of~$\NE$ in~$\Z_p^{\times}$ is at most two if~$p$ is odd, and at most four if~$p=2$.

For a coset~$\coset$ in~$\Qp^{\times}/ \NE$ contained in~$\Z_p$, the \emph{partial Hecke orbit of~$E$ along~$\coset$}~is
\begin{displaymath}
  \corbit
  \=
  \bigcup_{n \in \coset \cap \N} \supp(T_n(E)).
\end{displaymath}

In the following theorem we use the action of Hecke correspondences on compactly supported measures, see, \emph{e.g.}, \cite[Section~2.8]{HerMenRivII}.
For~$n$ in~$\N$, put
\begin{equation*}
  \sigma_1(n)
  \=
  \sum_{d \ge 1, d \mid n} d.
\end{equation*}

\begin{theorem}[\textcolor{black}{\cite[Theorem~C and Corollary~6.1]{HerMenRivII}}]
  \label{t:Hecke-orbits}
  For every~$E$ in~$\Sups$ and all cosets~$\coset$ and~$\coset'$ in~$\Qp^{\times}/ \NE$ contained in~$\Z_p$, the following properties hold.
  \begin{enumerate}
  \item[$(i)$]
    The closure~$\corbitc$ of~$\corbit$ in~$\Sups$ is compact.
    Moreover, there is a Borel probability measure~$\mu_{\coset}^E$ on~$\Ell(\Cp)$ whose support is equal to~$\corbitc$, and such that for every sequence~$(n_j)_{j = 1}^{\infty}$ in~$\coset \cap \N$ tending to~$\infty$, we have the weak convergence of measures
    \begin{displaymath}
      \odelta_{T_{n_j}(E), p} \to \mu_{\coset}^E
      \text{ as }
      j \to \infty.
    \end{displaymath}
  \item[$(ii)$]
    For every~$E'$ in~$\overline{\Orb_{\NE}(E)}$ and every~$n$ in~$\coset \cap \N$, we have
    \begin{displaymath}
      T_n \left( \overline{\Orb_{\coset'}(E')} \right)
      =
      \overline{\Orb_{\coset \cdot \coset'}(E)}
      \text{ and }
      \frac{1}{\sigma_1(n)} (T_n)_* \mu_{\coset'}^{E'}
      =
      \mu_{\coset \cdot \coset'}^E.
    \end{displaymath}
  \end{enumerate}
\end{theorem}

The following corollary is an immediate consequence of Theorems~\ref{t:Hecke-orbits-nonatomic} and~\ref{t:Hecke-orbits} and~\cite[Theorem~C]{HerMenRiv20}.

\begin{coro}
  \label{c:Hecke-orbits-nonatomic}
  For every~$E$ in~$\Ell(\Cp)$, $\alpha$ in~$\Cp$, and ${\varepsilon > 0}$, there exists ${r > 0}$ such that the following set is finite:
  \begin{displaymath}
    \left\{n\in \N \colon \deg(T_n(E)|_{\bfD_p(\alpha, r)}) \ge \varepsilon \sigma_1(n)\right\}.
  \end{displaymath}
\end{coro}

Previously, Charles showed that the set above with~$\N$ replaced by ${\N \ssetminus p\N}$ has zero density \cite[Proposition~3.2]{Cha18}.

The proof of Theorem~\ref{t:Hecke-orbits-nonatomic} is given after the following lemma.

\begin{lemma}
  \label{l:Hecke-independence}
  Let~$E_0$ be in~$\Ell(\Cp)$.
  If for distinct prime numbers~$q$ and~$q'$ we put
  \begin{displaymath}
    I \= \supp(T_q(E_0)) \cap \supp(T_{q'}(E_0)),
  \end{displaymath}
  then we have~$\deg(T_q(E_0)|_I) \le 24$.
\end{lemma}

\begin{proof}
  Given an elliptic curve~$\hE$ over~$\Cp$, denote by~$\End(\hE)$ and~$\Aut(\hE)$ the set of all endomorphisms and the set of all automorphisms of~$\hE$, respectively.
  Furthermore, for every elliptic curve~$\hE'$ over~$\Cp$ and every~$m$ in~$\N$, denote by~$\Hom_m(\hE, \hE')$ the set of all isogenies from~$\hE$ to~$\hE'$ of degree~$m$.

  Choose an elliptic curve~$\hE_0$ representing~$E_0$ and for each~$E$ in~$I$ choose an elliptic curve~$\hE$ representing~$E$ and an isogeny~$\phi_E$ in ${\Hom_{q'}(\hE, \hE_0)}$.
  Let~$\phi$ be in ${\Hom_{q}(\hE_0, \hE)}$ and set ${\psi \= \phi_E \circ \phi}$.
  The isogeny~$\psi$ determines both~$E$ in~$I$ and~$\phi$.
  Indeed, suppose that there are~$E'$ in~$I$ and~$\phi'$ in~$\Hom_{q}(\hE_0, \hE')$ with ${\phi_{E'}\circ \phi' = \psi}$.
  The group~$\Ker (\psi)$ has~$qq'$ elements, so it has a unique subgroup of order~$q$.
  Since~$\Ker (\phi)$ and~$\Ker (\phi')$ are two such subgroups, we have ${\Ker (\phi)=\Ker (\phi')}$.
  Then ${E = E'}$ by \cite[Chapter~III, Proposition~4.12]{Sil09}, and from the equality ${\phi_E \circ \phi = \phi_E \circ \phi'}$ we deduce ${\phi = \phi'}$.
  We thus have
  \begin{equation}
    \label{eq:52}
    \begin{split}
      \deg(T_q(E_0)|_{I})
      & =
        \sum_{E \in I} \# \Hom_q(\hE_0, \hE) / \# \Aut(\hE)
      \\ & \le
           \sum_{E \in I} \# \Hom_q(\hE_0, \hE)
      \\ & \le
           \sum_{E \in I} \# \{ \phi_E \circ \phi \colon \phi \in \Hom_q(\hE_0, \hE) \}
      \\ & \le
           \# \{ \psi \in \End(\hE_0) \colon \deg(\psi) = qq' \}.
    \end{split}
  \end{equation}
  If~$E_0$ is not a \CM{} point, then this last number is equal to zero and the lemma follows in this case.
  Suppose~$E_0$ is a \CM{} point, so the field of fractions~$K$ of~$\End(\hE_0)$ is a quadratic imaginary extension of~$\Q$.
  Denote by~$\cO_K$ the ring of integers of~$K$.
  Since each of the ideals~$q \cO_K$ and~$q' \cO_K$ is either prime or a product of two conjugate prime ideals, there are at most four ideals of~$\cO_K$ of norm~$qq'$.
  We thus have
  \begin{displaymath}
    \# \{ \psi \in \End(\hE_0) \colon \deg(\psi) = qq' \}
    \le
    \# \{ x \in \cO_K \colon x \overline{x} = qq' \}
    \le
    4 \# \cO_K^{\times}
    \le
    24.
  \end{displaymath}
  Together with~\eqref{eq:52} this completes the proof of the lemma.
\end{proof}

\begin{proof}[Proof of Theorem~\ref{t:Hecke-orbits-nonatomic}]
  By \cite[Theorem~C]{HerMenRiv20}, it is sufficient to assume that~$E$ is in~$\Sups$.
  Moreover, using Theorem~\ref{t:Hecke-orbits}$(i)$ and \cite[Theorem~C]{HerMenRiv20} again, it is sufficient to prove that for every coset~$\coset$ in~$\Qp^{\times} / \Nr_E$ contained in~$\Z_p$ the measure~$\mu_{\coset}^E$ has no atom in~$\corbitc$.

  Fix~$E_0$ in~$\corbitc$ and let~$N \ge 1$ be a given integer.
  Choose a set~$P$ of~$2N$ prime numbers that are contained in~$(\Z_p^{\times})^2$ and that are larger than~$100 N$.
  Note that every~$q$ in~$P$ is a $p$\nobreakdash-adic square and that by Theorem~\ref{t:Hecke-orbits}$(ii)$ we have
  \begin{equation}
    \label{eq:53}
    \frac{1}{\sigma_1(q)} (T_q)_* \mu_{\coset}^E
    =
    \mu_{\coset}^E.
  \end{equation}
  Moreover, for all distinct~$q$ and~$q'$ in~$P$ denote by~$I(q, q')$ the set~$I$ in Lemma~\ref{l:Hecke-independence} and put
  \begin{displaymath}
    S_q
    \=
    \supp(T_q(E_0)) \ssetminus \bigcup_{q' \in P, q' \neq q} I(q, q').
  \end{displaymath}
  Then by the inequality~$q > 100 N$ and Lemma~\ref{l:Hecke-independence}, we have
  \begin{equation}
    \label{eq:54}
    \begin{split}
      \deg(T_q(E_0)|_{S_q})
      & \ge
        \deg(T_q(E_0)) - \sum_{q' \in P, q' \neq q} \deg(T_q(E_0)|_{I(q, q')})
      \\ & \ge
           q + 1 - 24(\# P - 1)
      \\ & \ge
           \frac{q + 1}{2}.
    \end{split}
  \end{equation}
  On the other hand, note that the signed measure ${\mu_{\coset}^E - \mu_{\coset}^E(\{ E_0 \}) \delta_{E_0}}$ is nonnegative, thus the same holds for ${(T_q)_*(\mu_{\coset}^E - \mu_{\coset}^E(\{ E_0 \}) \delta_{E_0})}$.
  Combined with Theorem~\ref{t:Hecke-orbits}$(ii)$, this implies that for every~$E'$ in~$\Ell(\Cp)$ we have
  \begin{multline*}
    \mu_{\coset}^E (\{ E' \})
    =
    \left( \frac{1}{q + 1} (T_q)_* \mu_{\coset}^E \right)(\{ E' \})
    \ge
    \frac{\mu_{\coset}^E (\{ E_0 \})}{q + 1} \left( (T_q)_* \delta_{E_0} \right)(\{ E' \})
    \\ =
    \frac{\mu_{\coset}^E (\{ E_0 \})}{q + 1} \deg(T_q(E_0)|_{\{E'\}}).
  \end{multline*}
  Together with~\eqref{eq:54} this implies
  \begin{displaymath}
    1
    =
    \mu_{\coset}^E(\corbitc)
    \ge
    \sum_{q \in P} \mu_{\coset}^E(S_q)
    \ge
    \sum_{q \in P} \frac{\mu_{\coset}^E(\{ E_0 \})}{q + 1} \deg(T_q(E_0)|_{S_q})
    \ge
    N \mu_{\coset}^E(\{ E_0 \}).
  \end{displaymath}
  Since~$N$ is arbitrary, this implies that~$E_0$ is not an atom of~$\mu_{\coset}^E$ and completes the proof of the theorem.
\end{proof}

\begin{remark}
  \label{r:homogeneous-limits}
  A different strategy to prove Theorem~\ref{t:Hecke-orbits-nonatomic} is to use that for every~$E$ in~$\Sups$ and every coset~$\coset$ in~$\Qp^{\times} / \NE$ contained in~$\Z_p$, the measure~$\mu_{\coset}^E$ is the projection of a certain homogeneous measure under an analytic map of finite degree.
  Theorem~\ref{t:Hecke-orbits-nonatomic} then follows from the fact that the partial Hecke orbit~$\corbit$ is infinite.
\end{remark}

\subsection{On the limit measures of \CM{} points}
\label{ss:CM-nonatomic}
The goal of this section is to prove Theorem~\ref{t:CM-nonatomic}.
The proof is based on Theorem~\ref{t:Hecke-orbits-nonatomic} and on the description of all accumulation measures of~\eqref{eq:51} given in the companion papers~\cite{HerMenRiv20,HerMenRivII}.
We start recalling some results in the latter.

Recall from Section~\ref{ss:singular-moduli} that a discriminant is the discriminant of an order in a quadratic imaginary extension of~$\Q$.
A \emph{fundamental discriminant} is the discriminant of the ring of integers of a quadratic imaginary extension of~$\Q$.
For each discriminant~$D$, there is a unique fundamental discriminant~$d$ and a unique integer~$f \ge 1$ such that~$D = df^2$.
In this case, $d$ and~$f$ are the \emph{fundamental discriminant} and \emph{conductor of~$D$}, respectively.
A discriminant is \emph{prime}, if it is fundamental and divisible by only one prime number.
If~$d$ is a prime discriminant divisible by~$p$, then
\begin{displaymath}
  p \equiv -1 \mod 4
  \text{ and }
  d = -p,
  \text{ or }
  p = 2
  \text{ and~$d = -4$ or~$d = -8$.}
\end{displaymath}

A \emph{$p$\nobreakdash-adic quadratic order} is a $\Z_p$\nobreakdash-order in a quadratic extension of~$\Qp$, and a \emph{$p$\nobreakdash-adic discriminant} is a set formed by the discriminants of all $\Z_p$\nobreakdash-bases of a $p$\nobreakdash-adic quadratic order.
Every $p$\nobreakdash-adic discriminant is thus a coset in~$\Qp^{\times} / (\Z_p^{\times})^2$ contained in~$\Z_p$.

The \emph{$p$\nobreakdash-adic discriminant} of a formal \CM{} point~$E$, is the $p$\nobreakdash-adic discriminant of the $p$\nobreakdash-adic quadratic order~$\End(\FE)$.
Given a $p$\nobreakdash-adic discriminant~$\pd$, put
\begin{displaymath}
  \Lambda_{\pd}
  \=
  \{ E \in \Ell(\Cp) \colon \text{ formal \CM{} point of $p$\nobreakdash-adic discriminant~$\pd$} \}.
\end{displaymath}

\begin{theorem}[\textcolor{black}{\cite[Theorems~A and~B]{HerMenRivII}}]
  \label{t:CM}
  For every $p$\nobreakdash-adic discriminant~$\pd$, the following properties hold.
  \begin{enumerate}
  \item[$(i)$]
    The set~$\Lambda_{\pd}$ is a compact subset of~$\Ell(\Cp)$, and there is a Borel probability measure~$\nu_{\pd}$ on~$\Ell(\Cp)$ whose support is equal to~$\Lambda_{\pd}$ and such that the following equidistribution property holds.
    Let~$(D_n)_{n = 1}^{\infty}$ be a sequence of discriminants in~$\pd$ tending to~$- \infty$, such that for every~$n$ the fundamental discriminant of~$D_n$ is either not divisible by~$p$, or not a prime discriminant.
    Then we have the weak convergence of measures
    \begin{displaymath}
      \odelta_{D_n, p} \to \nu_{\pd}
      \text{ as }
      n \to \infty.
    \end{displaymath}
  \item[$(ii)$]
    Suppose that there is a prime discriminant~$d$ divisible by~$p$ and an integer ${m \ge 0}$ such that ${D \= d p^{2m}}$ is in~$\pd$.
    Then there are Borel probability measures~$\nu_{\pd}^+$ and~$\nu_{\pd}^-$ on~$\Ell(\Cp)$ such that the following equidistribution property holds.
    For every sequence~$(f_n)_{n = 0}^{\infty}$ in~$\N$ tending to~$\infty$ such that for every~$n$ we have~$\left( \frac{d}{f_n} \right) = 1$ (resp.~$\left( \frac{d}{f_n} \right) = - 1$), we have the weak convergence of measures
    \begin{displaymath}
      \odelta_{D (f_n)^2, p} \to \nu_{\pd}^+
      \text{ (resp.~$\odelta_{D (f_n)^2, p} \to \nu_{\pd}^-$)}
      \text{ as }
      n \to \infty.
    \end{displaymath}
  \end{enumerate}
\end{theorem}

The proof of Theorem~\ref{t:CM-nonatomic} is given after the following proposition, in which we gather further properties of the limit measures in Theorem~\ref{t:CM}.
To state it, we introduce some notation.

A $p$\nobreakdash-adic discriminant is \emph{fundamental}, if it is the $p$\nobreakdash-adic discriminant of the ring of integers of a quadratic extension of~$\Qp$.
Let~$\pfd$ be a fundamental $p$\nobreakdash-adic discriminant.
For~$\Delta$ in~$\pfd$, the field~$\Qp(\sqrt{\Delta})$ depends only on~$\pfd$, but not on~$\Delta$.
Denote it by~$\Qpfd$.
Choose a formal \CM{} point~$E_{\pfd}$ such that~$\End(\FE)$ is isomorphic to the ring of integers of~$\Qpfd$, as follows.
If~$\pfd$ does not contain a prime discriminant that is divisible by~$p$, then choose an arbitrary formal \CM{} point~$E_{\pfd}$ in~$\Lambda_{\pfd}$.
In the case where~$\pfd$ contains a prime discriminant~$d$ that is divisible by~$p$, then~$d$ is the unique fundamental discriminant in~$\pfd$ with this property and we choose~$E_{\pfd}$ in~$\Lambda_d$.
Note that if~$\Qpfd$ is unramified over~$\Qp$ then~$\Nr_{E_{\pfd}} = \Z_p^{\times}$, and that if~$\Qpfd$ is ramified over~$\Qp$ then~$\Nr_{E_{\pfd}}$ is a subgroup of~$\Z_p^{\times}$ of index two, see, \emph{e.g.}, \cite[Lemma~2.3]{HerMenRivII}.

Denote by~$\kval$ Katz' valuation on~$\Sups$, as defined in \cite[Section~4.1]{HerMenRiv20} and put
\begin{displaymath}
  N_p
  \=
  \left\{ E \in \Sups \colon \kval(E) < \frac{p}{p + 1} \right\}.
\end{displaymath}
For~$E$ in~$N_p$, denote by~$H(E)$ the canonical subgroup of~$E$ \cite[Theorem~3.10.7]{Kat73}.
The \emph{canonical branch of the Hecke correspondence~$T_p$} is the map ${\t \colon N_p \to \Sups}$ defined by ${\t(E) \= E / H(E)}$.
The map~$\t$ is analytic in the sense that it is given by a finite sum of Laurent series, each of which converges on all of~$N_p$, see, \emph{e.g.}, \cite[Theorem~B.1]{HerMenRiv20}.

Given a fundamental $p$\nobreakdash-adic discriminant~$\pfd$ and an integer~$m \ge 0$, define the affinoid
\begin{equation}
  \label{eq:55}
  A_{\pfd p^{2m}}
  \=
  \begin{cases}
    \kval^{-1}(\frac{1}{2} \cdot p^{-m})
    & \text{if~$\Qpfd$ is ramified over~$\Qp$};
    \\
    \kval^{-1} ([1, \infty])
    & \text{if~$\Qpfd$ is unramified over~$\Qp$ and~$m = 0$};
    \\
    \kval^{-1}(\frac{p}{p + 1} \cdot p^{-m})
    & \text{if~$\Qpfd$ is unramified over~$\Qp$ and~$m \ge 1$}.
  \end{cases}
\end{equation}

In the following proposition we summarize some of the results in \cite[Proposition~7.1, (7.13) and Sections~7.2 and~7.3]{HerMenRivII}.

\begin{proposition}
  \label{p:CM-orbits}
  For every fundamental $p$\nobreakdash-adic discriminant~$\pfd$ we have
  \begin{equation}
    \label{eq:56}
    \nu_{\pfd}
    =
    \begin{cases}
      \mu_{\Z_p^\times}^{E_{\pfd}}
      & \text{ if $\Qpfd$ is unramified over~$\Qp$};
      \\
      \frac{1}{2}\left( \mu_{\Npfd}^{E_{\pfd}} + \mu_{\Z_p^\times \ssetminus \Npfd}^{E_{\pfd}}\right)
      & \text{ if~$\Qpfd$ is ramified over~$\Qp$},
    \end{cases}
  \end{equation}
  and for every integer ${m \ge 1}$ we have
  \begin{equation}
    \label{eq:57}
    \nu_{\pfd p^{2m}}
    =
    \begin{cases}
      \frac{1}{p^{m - 1}(p + 1)} ( \t^m \bigm\vert_{A_{\pfd p^{2m}}} )^* \nu_{\pfd}
      & \text{if~$\Qpfd$ is unramified over~$\Qp$};\\
      \frac{1}{p^m} ( \t^m \bigm\vert_{A_{\pfd p^{2m}}} )^* \nu_{\pfd}
      & \text{if~$\Qpfd$ is ramified over~$\Qp$}.
    \end{cases}
  \end{equation}
  If in addition~$\pfd$ contains a prime discriminant, then we also have
  \begin{equation}
    \label{eq:58}
    \nu_{\pfd}^+
    =
    \mu_{\Npfd}^{E_{\pfd}}
    \text{ and }
    \nu_{\pfd}^-
    =
    \mu_{\Z_p^{\times} \ssetminus \Npfd}^{E_{\pfd}},
  \end{equation}
  and~\eqref{eq:57} holds for~$\nu_{\pfd p^{2m}}^+$ (resp.~$\nu_{\pfd p^{2m}}^-$), with~$\nu_{\pfd}$ replaced by~$\nu_{\pfd}^+$ (resp.~$\nu_{\pfd}^-$).
\end{proposition}

\begin{proof}[Proof of Theorem~\ref{t:CM-nonatomic}]
  Let~$(D_n)_{n = 1}^{\infty}$ be a sequence of discriminants tending to~$- \infty$ such that the sequence of measures~$(\odelta_{D_n, p})_{n = 1}^{\infty}$ converges weakly to a measure different from~$\delta_{\xcan}$.
  By \cite[Theorem~A]{HerMenRiv20}, there is a constant~$c > 0$ such that for every~$n$ we have~$|D_n|_p > c$ and~$\Lambda_{D_n} \subseteq \Sups$.
  This implies that~$(D_n)_{n = 1}^{\infty}$ is contained in a finite union of $p$\nobreakdash-adic discriminants, see, \emph{e.g.}, \cite[Lemmas~2.1 and~A.1]{HerMenRivII}.
  Taking a subsequence if necessary, assume that~$(D_n)_{n = 1}^{\infty}$ is contained in a $p$\nobreakdash-adic discriminant~$\pd$.
  Let~$\pfd$ be the fundamental $p$\nobreakdash-adic discriminant and~$m \ge 0$ the integer such that~$\pd = \pfd p^{2m}$, see, \emph{e.g.}, \cite[Lemma~A.1$(i)$]{HerMenRivII}.

  Passing to a subsequence if necessary, there are two cases.

  \partn{Case 1}
  For every~$n$ the fundamental discriminant of~$D_n$ is either not divisible by~$p$, or not a prime discriminant.
  In this case the sequence~$(\odelta_{D_n, p})_{n = 1}^{\infty}$ converges to~$\nu_{\pd}$ by Theorem~\ref{t:CM}$(i)$.
  Then~\eqref{eq:56} in Proposition~\ref{p:CM-orbits} and Theorem~\ref{t:Hecke-orbits-nonatomic} imply that~$\nu_{\pfd}$ is nonatomic.
  This is the desired assertion in the case where~$m = 0$.
  If~$m \ge 1$, then the fact that~$\nu_{\pd}$ is nonatomic follows from~\eqref{eq:57} in Proposition~\ref{p:CM-orbits}, together with the fact that~$\nu_{\pfd}$ is nonatomic and the analyticity of the canonical branch~$\t$ of~$T_p$.

  \partn{Case 2}
  There is a prime discriminant~$d$ that is divisible by~$p$ and a sequence~$(f_n)_{n = 1}^{\infty}$ in~$\N$ such that for every~$n$ we have~$D_n = d f_n^2$ and~${\left( \frac{d}{f_n} \right) = 1}$ (resp. ${\left( \frac{d}{f_n} \right) = -1}$).
  In this case the sequence~$(\odelta_{D_n, p})_{n = 1}^{\infty}$ converges weakly to~$\nu_{\pd}^+$ (resp.~$\nu_{\pd}^-$) by Theorem~\ref{t:CM}$(ii)$.
  Then~\eqref{eq:58} in Proposition~\ref{p:CM-orbits} and Theorem~\ref{t:Hecke-orbits-nonatomic} imply that~$\nu_{\pfd}^+$ and~$\nu_{\pfd}^-$ are both nonatomic.
  This is the desired assertion in the case where~$m = 0$.
  If~$m \ge 1$, then that~$\nu_{\pd}^+$ and~$\nu_{\pd}^-$ are both nonatomic follows from the fact that~$\nu_{\pfd}^+$ and~$\nu_{\pfd}^-$ are both nonatomic, from the last assertion of Proposition~\ref{p:CM-orbits} and from the fact that the canonical branch~$\t$ of~$T_p$ is analytic.
\end{proof}

\subsection{Proof of Theorem~\ref{t:disperse}}
\label{ss:proof-disperse}
In the case where~$f$ is the $j$\nobreakdash-invariant, the desired estimate is a direct consequence of~\eqref{eq:12} and \cite[\emph{Th{\'e}or{\`e}me}~2.4]{CloUll04} if~${v = \infty}$ and of Theorem~\ref{t:CM-nonatomic} and \cite[Theorems~A and~B]{HerMenRivII}, which are summarized in Theorem~\ref{t:CM} in Section~\ref{ss:CM-nonatomic}, if~$v$ is a prime number.

To prove Theorem~\ref{t:disperse} in the general case, let~$K$ be a finite extension of~$\Q$ inside~$\Qalg$, let~$f$ be a nonconstant modular function defined over~$K$ and let~$\Phi(X, Y)$ be a modular polynomial of~$f$ in~$K[X, Y]$.
Note that~$\Phi(X, Y)$ is irreducible in~$\C_v[X, Y]$.
Let~$C_3$, $\theta$ and~$\eta$ be given by Lemma~\ref{l:reduction-to-j} and denote by~$\delta_X$ (resp.~$\delta_Y$) the degree of~$\Phi(X, Y)$ in~$X$ (resp.~$Y$).
Furthermore, note that~$\Phi(X, \alpha)$ is nonzero (Proposition~\ref{p:omitted-and-cuspidal}) and denote by~$Z$ the (possibly empty) finite set of zeros of this polynomial in~$\Cv$.

Let~$\tau$ be a quadratic imaginary number in~$\H$ that is not a pole of~$f$, and put
\begin{equation}
  \label{eq:59}
  \sm
  \=
  j(\tau)
  \text{ and }
  \fsm
  \=
  f(\tau).
\end{equation}
Then, $\sm$ is a singular modulus of the $j$\nobreakdash-invariant and~$\fsm$ is a singular modulus of~$f$.
Noting that for every~$\sigma$ in~$\Gal(\Qalg|K)$ we have ${\Phi(\sigma(\sm), \sigma(\fsm)) = 0}$, by Lemma~\ref{l:reduction-to-j} there is ${r_0 > 0}$ independent of~$\tau$ such that for every~$r$ in~$]0, r_0[$ we have
\begin{multline}
  \label{eq:60}
  \# (\oO_{K}(\fsm) \cap \bfD_v(\alpha, r))
  \le
  \delta_Y \# \{ \sm' \in \oO_{K}(\sm) \colon |\sm'|_v \ge C_3^{-1} r^{- \eta} \}
  \\ \quad +
  \delta_Y \sum_{z_0 \in Z} \# (\oO_{K}(\sm) \cap \bfD_v(z_0, C_3 r^{\theta})).
\end{multline}
In the case where~$v$ is a prime number, we have
\begin{equation}
  \label{eq:61}
  \{ \sm' \in \oO_{K}(\sm) \colon |\sm'|_v > 1 \}
  =
  \emptyset,
\end{equation}
so the desired estimate for~$f$ follows from that for the $j$\nobreakdash-invariant, together with~\eqref{eq:60} and Proposition~\ref{p:singular-moduli}$(ii)$.
To prove the theorem in the case where ${v = \infty}$, we use the fact that the limit measure~$\mu_{\infty}$ in \cite[\emph{Th{\'e}or{\`e}me}~2.4]{CloUll04}, seen as a measure on~$\P^1(\C)$, is nonatomic.
Thus, there is ${R > 1}$ such that
\begin{equation}
  \label{eq:62}
  \mu_{\infty}( \{ z \in \Cv \colon |z|_v > R \} )
  \le
  \frac{\varepsilon}{2 \delta_Y},
\end{equation}
and, if~$\Phi(X, \alpha)$ is nonconstant, such that for every~$z_0$ in~$Z$ we have
\begin{equation}
  \label{eq:63}
  \mu_{\infty}(\bfD_v(z_0, R^{-1}))
  \le
  \frac{\varepsilon}{2 \delta_X \delta_Y}.
\end{equation}
Then, the desired estimate for~$f$ and ${v = \infty}$ follows from that for the $j$\nobreakdash-invariant, together with~\eqref{eq:60} and Proposition~\ref{p:singular-moduli}$(ii)$.

\section{$p$-Adic approximation by singular moduli}
\label{s:badly-approximable-Hauptmodul}

The goal of this section is to prove the following proposition, from which we derive Theorem~\ref{t:badly-approximable-Hauptmodul}.
Throughout this section, fix a prime number~$p$.

\begin{proposition}
  \label{p:badly-approximable-j}
  Let~$\sm_0$ be a singular modulus of the $j$\nobreakdash-invariant.
  Then, there exists a constant ${A > 0}$ such that for every singular modulus~$\sm$ of the $j$\nobreakdash-invariant that is different from~$\sm_0$ we have
  \begin{displaymath}
    -\log |\sm - \sm_0 |_p
    \le
    A \log |D_{\sm}|.
  \end{displaymath}
\end{proposition}

The archimedean counterpart of this estimate was shown by Habegger \cite[Lemmas~5 and~8 and formula~(11)]{Hab15}.
See also Conjecture~\ref{c:badly-approximable} in Section~\ref{ss:badly-approximable-Hauptmodul}.

After some preliminaries in Section~\ref{ss:deformations}, the proofs of Proposition~\ref{p:badly-approximable-j} and Theorem~\ref{t:badly-approximable-Hauptmodul} are given in Sections~\ref{ss:proof-badly-approximable-j} and~\ref{ss:proof-badly-approximable-Hauptmodul}, respectively.
To prove Proposition~\ref{p:badly-approximable-j}, we use that \CM{} points outside~$\Sups$ are isolated \cite[Corollary~B]{HerMenRiv20} to restrict to the case where the \CM{} points corresponding to~$\sm$ and~$\sm_0$ are both in~$\Sups$.
In the case where the conductors of~$D_{\sm}$ and~$D_{\sm_0}$ are both $p$\nobreakdash-adic units, we use an idea in the proof of \cite[Proposition~5.11]{Cha18}.
To deduce the general case from this particular case, we use a formula in~\cite{HerMenRivII} showing how the canonical branch of~$T_p$ relates \CM{} points whose conductors differ by a power of~$p$ (Theorem~\ref{t:CM-from-canonical}).

Denote by~$\tSups$ the (finite) set of isomorphism classes of supersingular elliptic curves over~$\Fpalg$.
For~$\ss$ in~$\tSups$, denote by~$\Dss$ the set of all~$E$ in~$\Ell(\Cp)$ having good reduction, and such that the reduced class is~$\ss$.
The set~$\Dss$ is a residue disc in~$\Ell(\Cp)$.

\subsection{Formal $\Z_p$-modules and elliptic curves}
\label{ss:deformations}
In this section, we briefly recall the work of Gross and Hopkins in~\cite{HopGro94b}, on deformation spaces of formal modules.
See also \cite[Sections~2.4 to~2.7]{HerMenRivII} for a more detailed account of the results needed here.
For every~$\ss$ in~$\tSups$, we describe an action of $(\End(\ss) \otimes \Z_p)^\times$ on a certain ramified covering of~$\Dss$.
In the proof of Proposition~\ref{p:badly-approximable-j} we use a relation between the metric on~$\Dss$ and the natural metric of the covering, which is stated as Theorem~\ref{t:Pi-properties} below.

Fix~$\ss$ in~$\tSups$ and a representative elliptic curve defined over~$\F_{p^2}$ that we also denote by~$\ss$.
Denote by~$\Fss$ the formal group of~$\ss$ endowed with its natural structure of formal $\Z_p$\nobreakdash-module and set
\begin{displaymath}
  \Bss
  \=
  \End_{\Fpalg}(\Fss)\otimes \Qp, \Rss
  \=
  \End_{\Fpalg}(\Fss)
  \text{ and }
  \Gss
  \=
  \Aut_{\Fpalg}(\Fss).
\end{displaymath}
Then, $\Bss$ is a division quaternion algebra over~$\Qp$ and the sets~$\Rss$ and~$\Gss$ embed in~$\Bss$ as the maximal order and its group of units, respectively.
Denote by~$g \mapsto \overline{g}$ the involution of~$\Bss$, and for~$g$ in~$\Bss$ denote by ${\nr(g) \= g \overline{g}}$ in~$\Qp$ its \emph{reduced norm}.
On the other hand, the function ${\ord_{\Bss} \colon \Bss \to \Z \cup \{\infty\}}$ defined for~$g$ in~$\Bss$ by ${\ord_{\Bss}(g) \= \ord_p(\nr(g))}$, is the unique valuation extending the valuation~$2 \ord_p$ on~$\Q_p$.
Identifying~$\Rss$ and~$\Gss$ with their images in~$\Bss$, we have
\begin{displaymath}
  \Rss = \{ g \in \Bss \colon \ord_{\Bss}(g) \ge 0 \}
  \text{ and }
  \Gss = \{ g \in \Bss \colon \ord_{\Bss}(g) = 0 \}.
\end{displaymath}
The function ${\dBss \colon \Bss \times \Bss \to \R}$ defined for~$g$ and~$g'$ in~$\Bss$ by
\begin{displaymath}
  \dBss(g, g')
  \=
  p^{- \frac{1}{2} \ord_{\Bss}(g - g')},
\end{displaymath}
defines an ultrametric distance on~$\Bss$ that makes~$\Bss$ into a topological algebra over~$\Qp$.

Identify the residue field of~$\Cp$ with an algebraic closure~$\Fpalg$ of~$\Fp$ and denote by ${\pi \colon \cO_p \rightarrow \Fpalg}$ the reduction map.
Moreover, denote by~$\Q_{p^2}$ the unique unramified quadratic extension of~$\Q_p$ inside~$\Cp$, and by~$\Z_{p^2}$ its ring of integers.

Let~$R_0$ be a complete, local, Noetherian~$\Z_p$\nobreakdash-algebra with maximal ideal~$\cM_0$ and residue field isomorphic to a subfield~$\rfk$ of~$\Fpalg$ that contains~$\F_{p^2}$.
Fix a reduction map~${R_0 \to \rfk}$.
We are mainly interested in the special case where~$R_0$ the ring of integers of a finite extension of~$\Qp$ contained in~$\Cp$ together with the restriction of~$\pi$, or a quotient of such ring of integers together with the morphism induced by the restriction of~$\pi$.
We stick to the general case for convenience.

A \emph{deformation of~$\Fss$ over~$R_0$} is a pair~$(\cF,\alpha)$, where~$\cF$ is a formal $\Z_p$\nobreakdash-module over~$R_0$ and ${\alpha \colon \tcF \to \Fss}$ is an isomorphism of formal~$\Z_p$\nobreakdash-modules defined over~$\rfk$.
Here, $\tcF$ is the formal group over~$\rfk$ obtained as the base change of~$\cF$ under the reduction map ${R_0 \to \rfk}$.
Two such deformations~$(\cF,\alpha)$ and~$(\cF',\alpha')$ are \emph{isomorphic}, if there exists an isomorphism~$\varphi$ in ${\Iso_{R_0} ( \cF , \cF')}$ with reduction~$\tvarphi$ such that ${\alpha'\circ \tvarphi = \alpha}$.
Denote by~$\Xss(R_0)$ the set of isomorphism classes of deformations of~$\Fss$ over~$R_0$.

For the rest of this section, we further assume that our choice of the representative elliptic curve~$\ss$ is such that~$\Fss$ is isomorphic over~$\F_{p^2}$ to the specialization of a universal formal~$\Z_p$\nobreakdash-module of height two, see~\cite[Lemma~2.5]{HerMenRivII}.
Then, a consequence of the work of Gross and Hopkins is that there exists a bijection
\begin{equation}
  \label{eq:parametrization-X-e}
  \cM_{0} \rightarrow \Xss(R_0)
\end{equation}
that is functorial in~$R_0$, see \cite[Section~12]{HopGro94b} and \cite[Section~2.5]{HerMenRivII} for details.
Moreover, we have the action
\begin{displaymath}
  \begin{array}{rcl}
    \Aut_{\rfk}(\Fss) \times \bfX_\ss(R_0)& \to & \bfX_\ss(R_0)
    \\
    (\beta, (\cF,\alpha)) & \mapsto & \beta\cdot(\cF,\alpha) \= (\cF,\beta \circ \alpha).
  \end{array}
\end{displaymath}

Let~$\cK$ be a finite extension of~$\Q_{p^2}$ inside~$\Cp$, with ring of integers~$\OK$ and residue field~$\rf$.
Consider the reduction map~${\OK \to \rf}$ obtained as the restriction of~$\pi$ to~$\OK$.
Denote by~$\bfY(\ss, \OK)$ the space of isomorphism classes of pairs~$(E,\alpha)$ formed by an elliptic curve~$E$ given by a Weierstrass equation with coefficients in~$\OK$ and having smooth reduction, and an isomorphism ${\alpha \colon \tE \to \ss}$ defined over~$\rf$.
Here, two pairs~$(E,\alpha)$ and~$(E',\alpha')$ are \emph{isomorphic} if there exists an isomorphism ${\psi \colon E \to E'}$ defined over~$\rf$ such that ${\alpha'\circ \widetilde{\psi} = \alpha}$.
Consider the natural map
\begin{equation}
  \label{eq:equivalence-formal}
  \bfY(\ss,\OK) \to \Xss(\OK)
\end{equation}
mapping a class in~$\bfY(\ss,\OK)$ represented by a pair~$(E,\alpha)$, to the class in~$\Xss(\OK)$ represented by the deformation~$(\FE, \halpha)$.
Here, ${\halpha \colon \tFE \to \Fss}$ is the isomorphism induced by~$\alpha$.
This map is known to be a bijection thanks to the so-called Woods Hole Theory, see \cite[Section~6]{LubSerTat64} or \cite[Theorem~4.1]{ColMcM10}.
We obtain a map
\begin{equation}
  \label{eq:64}
  \Pi_{\ss,\cK} \colon \Xss(\OK) \to \Ell_{\sups}(\Qpalg) \cap \Dss,
\end{equation}
by composing the inverse of~\eqref{eq:equivalence-formal} with the natural map from~$\bfY(\ss,\OK)$ to ${\Ell_{\sups}(\Qpalg) \cap \Dss}$.

Consider
\begin{displaymath}
  \sK
  \=
  \{ \text{finite extensions of~$\Q_{p^2}$ inside~$\Cp$} \}
\end{displaymath}
as a directed set with respect to the inclusion.
For each~$\cK$ in~$\sK$, consider the parametrization~\eqref{eq:parametrization-X-e} with ${R_0 = \OK}$.
Taking a direct limit over~$\sK$ and then a completion, we obtain a set~$\hDss$ that is parametrized by the maximal ideal of~$\Op$.
The action of~$\Gss$ on the system~$\{\Xss(\OK) \colon \cK \in \sK\}$  extends to a continuous map ${\Gss \times \hDss \to \hDss}$ that is analytic in the second variable, see \cite[Section~2.6]{HerMenRivII} for details.

In the following theorem, $\delta_{\ss} \= \# \Aut(\ss) /2$.
Note that~$\delta_{\ss} = 1$ if ${j(e) \neq 0, 1728}$ and that in all the cases we have~$1 \le \delta_{\ss} \le 12$, see, \emph{e.g.}, \cite[Appendix~A, Proposition~1.2(c)]{Sil09}.

\begin{theorem}[\textcolor{black}{\cite[Theorem~2.7]{HerMenRivII}}]
  \label{t:Pi-properties}
  Fix~$\ss$ in~$\tSups$.
  Then, the system ${\{\Pi_{\ss,\cK} \colon \cK \in \sK \}}$ given by~\eqref{eq:64} defines a ramified covering map
  \begin{displaymath}
    \Piss \colon \hDss \to \Dss,
  \end{displaymath}
  such that for every~$x$ in~$\hDss$ and every~$E$ in~$\Dss$ we have
  \begin{multline}
    \label{eq:65}
    \min \{ |x - x'|_p \colon x' \in \Piss^{-1}(E) \}^{\delta_{\ss}}
    \le
    |j(\Piss(x)) - j(E)|_p
    \\ \le
    \min \{ |x - x'|_p \colon x' \in \Piss^{-1}(E) \}.
  \end{multline}
\end{theorem}

\subsection{Proof of Proposition~\ref{p:badly-approximable-j}}
\label{ss:proof-badly-approximable-j}
The proof of Proposition~\ref{p:badly-approximable-j} is at the end of this section.

Let~$\ss$ be in~$\tSups$.
The set
\begin{displaymath}
  L(\ss)
  \=
  \{ \phi \in \Z + 2 \End(\ss) \colon \tr(\phi) = 0 \},
\end{displaymath}
is a $\Z$\nobreakdash-lattice of dimension three inside~$\End(\ss)$.
Define for each integer~$m \ge 1$,
\begin{displaymath}
  V_m(\ss)
  \=
  \{ \phi \in L(\ss) \colon \nr(\phi) = m \}.
\end{displaymath}

For each fundamental $p$\nobreakdash-adic discriminant~$\pfd$ and every discriminant~$D$ in~$\pfd$, the image of the set~$V_{|D|}(\ss)$ by the natural map~${\End(\ss) \to \End_{\Fpalg}(\Fss)}$, denoted by ${\phi \mapsto \hphi}$, is contained in
\begin{displaymath}
  \bfL_{\ss, \pfd}
  \=
  \{ \varphi \in \Z_p + 2 \Rss \colon \tr(\varphi) = 0, - \nr(\varphi) \in \pfd \},
\end{displaymath}
see \cite[Lemma~2.1]{HerMenRivII}.
Let ${U_{\ss, \pfd} \colon \bfL_{\ss, \pfd} \to \Gss}$ be the function defined by
\begin{align*}
  U_{\ss, \pfd}(\varphi)
  & \=
    \begin{cases}
      \frac{\varphi^2 + \varphi}{2}
      & \text{if~$\frac{\varphi^2 + \varphi}{2}$ belongs to~$\Gss$};
      \\
      1 + \frac{\varphi^2 + \varphi}{2}
      & \text{otherwise},
    \end{cases}
    \intertext{and for each~$\varphi$ in~$\bfL_{\ss, \pfd}$ define}
    \Fixss(\varphi)
  & \=
    \left\{ x \in \hDss \colon U_{\ss, \pfd}(\varphi) \cdot x = x \right\}.
\end{align*}

Given a fundamental $p$\nobreakdash-adic discriminant~$\pfd$, denote by~$\Q_{p^2}(\sqrt{\pfd})$ the compositum of~$\Q_{p^2}$ and~$\Q_p(\sqrt{\pfd})$.

\begin{proposition}[\textcolor{black}{\cite[Lemmas~4.5$(iv)$ and~4.15, and Proposition~5.6$(i)$]{HerMenRivII}}]
  \label{p:parametrization}
  Fix~$\ss$ in~$\tSups$ and a fundamental $p$\nobreakdash-adic discriminant~$\pfd$.
  \begin{enumerate}
  \item[$(i)$]
    For~$\varphi$ and~$\varphi'$ in~$\bfL_{\ss, \pfd}$ the sets~$\Fixss(\varphi')$ and~$\Fixss(\varphi)$ coincide if~$\varphi'$ belongs to~$\Qp(\varphi)$ and they are disjoint if~$\varphi'$ is not in~$\Qp(\varphi)$.
  \item [$(ii)$]
    We have ${\Piss^{-1}(\Lambda_{\pfd} \cap \Dss) \subseteq \Xss(\cO_{\Q_{p^2}(\sqrt{\pfd})})}$, and for every~$\Delta$ in~$\pfd$ we have
    \begin{displaymath}
      \Piss^{-1} \left( \Lambda_{\pfd} \cap \Dss \right)
      =
      \Fixss \left(\{ \varphi \in \bfL_{\ss, \pfd} \colon \nr(\varphi) = - \Delta \} \right).
    \end{displaymath}
  \end{enumerate}
\end{proposition}

For a fundamental $p$\nobreakdash-adic discriminant~$\pfd$, put~$\varepsilon_{\pfd} \= \frac{1}{2}$ if~$\Qpfd$ is ramified over~$\Qp$ and~$\varepsilon_{\pfd} \= 1$ if~$\Qpfd$ is unramified over~$\Qp$.

\begin{lemma}
  \label{l:fixed-to-fixer}
  For every prime number~$p$, every~$\ss$ in~$\tSups$ and every fundamental $p$\nobreakdash-adic discriminant~$\pfd$, the following property holds.
  For all~$\varphi$ and~$\widecheck{\varphi}$ in~$\bfL_{\ss, \pfd}$ and all~$x$ in~$\Fixss(\varphi)$ and~$\widecheck{x}$ in~$\Fixss(\widecheck{\varphi})$, we have
  \begin{displaymath}
    |x - \widecheck{x}|_p
    \ge
    p^{-\varepsilon_{\pfd}} \dist_{\Bss}(\varphi \widecheck{\varphi}, \widecheck{\varphi} \varphi).
  \end{displaymath}
\end{lemma}

\begin{proof}
  Let~$\varpi_0$ and~$\varpi$ be uniformizers of~$\cO_{\Qp{(\sqrt{\pfd}})}$ and~$\cO_{\Qp(\varphi)}$, respectively, and note that
  \begin{equation}
    \label{eq:66}
    \ord_p(\varpi_0)
    =
    \varepsilon_{\pfd}
    =
    \frac{1}{2} \ord_{\Bss}(\varpi).
  \end{equation}
  If~$x = \widecheck{x}$, then~$\widecheck{\varphi}$ is in~$\Qp(\varphi)$ by Proposition~\ref{p:parametrization}$(i)$  and therefore~$\varphi \widecheck{\varphi} = \widecheck{\varphi} \varphi$.
  So the desired property holds in this case.
  Assume~$x \neq \widecheck{x}$.
  By Proposition~\ref{p:parametrization}$(ii)$, $x$ and~$\widecheck{x}$ are both in~$\Piss^{-1}(\Lambda_{\pfd} \cap \Dss)$ and therefore in~$\Xss(\cO_{\Q_{p^2}(\sqrt{\pfd})})$.
  In particular, there is an integer~$N \ge 1$ such that~$|x - \widecheck{x}|_p = |\varpi_0|_p^N$.
  Let~$(\cF, \alpha)$ and~$(\widecheck{\cF}, \widecheck{\alpha})$ represent~$x$ and~$\widecheck{x}$, respectively, and denote by~$\cF_N$ and~$\widecheck{\cF}_N$ the base change of~$\cF$ and~$\widecheck{\cF}$ under the projection map ${\cO_{\Q_{p^2}(\sqrt{\pfd})} \to R_0\= \cO_{\Q_{p^2}(\sqrt{\pfd})} / \varpi_0^N \cO_{\Q_{p^2}(\sqrt{\pfd})}}$.
  Since~\eqref{eq:parametrization-X-e} is a bijection, there is an isomorphism ${\psi \colon \cF_N \to \widecheck{\cF}_N}$ defined over~$R_0$ such that~$\widecheck{\alpha}= \alpha \circ \tpsi$.
  This implies that the maps
  \begin{displaymath}
    \End_{R_0}(\cF_N) \to \Rss
    \text{ and }
    \End_{R_0}(\widecheck{\cF}_N) \to \Rss,
  \end{displaymath}
  given by
  \begin{displaymath}
    \phi \mapsto \alpha \circ \widetilde{\phi} \circ \alpha^{-1}
    \text{ and }
    \phi \mapsto \widecheck{\alpha} \circ \widetilde{\phi} \circ \widecheck{\alpha}^{-1},
  \end{displaymath}
  respectively, have the same image.
  Thus, by \cite[Proposition~3.3]{Gro86} we have
  \begin{displaymath}
    \cO_{\Qp(\varphi)} + \varpi^{N - 1} \Rss
    =
    \cO_{\Qp(\widecheck{\varphi})} + \varpi^{N - 1} \Rss.
  \end{displaymath}
  It follows that~$\varphi \widecheck{\varphi}-\widecheck{\varphi}\varphi$ is in~$\varpi^{N - 1} \Rss$.
  Together with~\eqref{eq:66}, this implies
  \begin{displaymath}
    \dist_{\Bss}(\varphi \widecheck{\varphi}, \widecheck{\varphi} \varphi)
    \le
    |\varpi_0|_p^{N - 1}
    =
    p^{\varepsilon_{\pfd}} |x - \widecheck{x}|_p.
    \qedhere
  \end{displaymath}
\end{proof}

In the following theorem we use the canonical branch~$\t$ of~$T_p$, recalled in Section~\ref{ss:CM-nonatomic}.

\begin{theorem}[\textcolor{black}{\cite[Theorem~4.6]{HerMenRivII}}]
  \label{t:CM-from-canonical}
  Let~$d$ be a fundamental discriminant such that ${\Lambda_d \subseteq \Sups}$.
  Then for every integer~$r \ge 1$ and every integer~$f \ge 1$ that is not divisible by~$p$, we have
  \begin{displaymath}
    \Lambda_{d (f p^r)^2}
    =
    \begin{cases}
      \t^{-1}\left(\Lambda_{d f^2} \right) \cap \kval^{-1}(\frac{1}{2p})
      & \text{if $r = 1$ and~$p$ ramifies in~$\Q(\sqrt{d})$};
      \\
      \t^{1-r} (\Lambda_{d (fp)^2})
      & \text{if $r \ge 2$ and~$p$ ramifies in~$\Q(\sqrt{d})$};
      \\
      \t^{-r} \left(\Lambda_{d f^2} \right)
      & \text{if~$r \ge 1$ and~$p$ is inert in~$\Q(\sqrt{d})$}.
    \end{cases}
  \end{displaymath}
\end{theorem}

The following lemma is \cite[Lemma~4.9]{HerMenRiv20}, see also \cite[Lemma~4.8]{ColMcM06} and \cite[Proposition~5.3]{Gro86}.

\begin{lemma}
  \label{l:vcm}
  Denote Katz' valuation by~$\kval$, as in Section~\ref{ss:CM-nonatomic}.
  Let~$D$ be a discriminant such that~$\Lambda_D \subseteq \Sups$ and let~$m \ge 0$ be the largest integer such that~$p^m$ divides the conductor of~$D$.
  Then for every~$E$ in~$\supp(\Lambda_D)$ we have
  \begin{displaymath}
    \min \left\{ \kval(E), \tfrac{p}{p + 1} \right\}
    =
    \begin{cases}
      \frac{1}{2} \cdot p^{-m}
      & \text{ if }p \text{ ramifies in } \Q(\sqrt{D});
      \\
      \frac{p}{p + 1} \cdot p^{-m}
      & \text{ if $p$  is inert in $\Q(\sqrt{D})$.}
    \end{cases}
  \end{displaymath}
\end{lemma}

\begin{proof}[Proof of Proposition~\ref{p:badly-approximable-j}]
  Let~$E_0$ be the \CM{} point such that~$j(E_0) = \sm_0$.
  Since \CM{} points outside~$\Sups$ are isolated \cite[Corollary~B]{HerMenRiv20}, we can assume that~$E_0$ is in~$\Sups$.
  Let~$\ss$ be the element of~$\tSups$ such that~$E_0$ is in~$\Dss$.

  Let~$\pfd$ be the fundamental $p$\nobreakdash-adic discriminant and~$m \ge 0$ the integer such that~$E_0$ is in~$\Lambda_{\pfd p^{2m}}$.
  Let~$\sm$ be a singular modulus different from~$\sm_0$ and let~$E$ be the \CM{} point satisfying~$j(E) = \sm$.
  Without loss of generality, assume that~$D_{\sm} \neq D_{\sm_0}$ and that~$E$ is in~$\Dss$.
  In view of Lemma~\ref{l:vcm}, we can also assume that there is a fundamental $p$\nobreakdash-adic discriminant~$\pfd'$ such that~$D_{\sm}$ is in~$p^{2m} \pfd'$, see also \cite[Lemma~2.1]{HerMenRivII}.
  Since~$\Lambda_{\pfd p^{2m}}$ and~$\Lambda_{\pfd' p^{2m}}$ are both compact by Theorem~\ref{t:CM}$(i)$ and they are disjoint if~$\pfd' \neq \pfd$, we can also assume that~$\pfd' = \pfd$.
  On the other hand, by Theorem~\ref{t:CM-from-canonical} and the fact that the canonical branch~$\t$ of~$T_p$ is analytic, it is sufficient to prove the lemma in the case where~$m = 0$, so~$E_0$ and~$E$ are both in~$\Lambda_{\pfd} \cap \Dss$.

  Using ${\delta_{\ss} \le 12}$ and Theorem~\ref{t:Pi-properties}, we can find~$x_0$ in~$\Piss^{-1}(E_0)$ and~$x$ in~$\Piss^{-1}(E)$ such that
  \begin{equation}
    \label{eq:67}
    | \sm - \sm_0 |_p
    \ge
    | x - x_0 |_p^{12}.
  \end{equation}
  On the other hand, by Proposition~\ref{p:parametrization}$(ii)$ there is~$\phi_0$ in~$\End(\ss)$ such that~$\hphi_0$ satisfies the equation~$X^2 - D_{\sm_0} = 0$, is in~$\bfL_{\ss, \pfd}$ and is such that~$x_0$ is in~$\Fixss(\hphi_0)$.
  Similarly, we can find~$\phi$ in~$\End(\ss)$ such that~$\hphi$ satisfies the equation~$X^2 - D_{\sm} = 0$, is in~$\bfL_{\ss, \pfd}$ and is such that~$x$ is in~$\Fixss(\hphi)$.
  Note that Proposition~\ref{p:parametrization}$(i)$ and our assumption~${D_{\sm} \neq D_{\sm_0}}$, imply that ${\phi_0 \phi - \phi \phi_0}$ is nonzero.
  Combined with the fact that~$\deg$ is a positive definite quadratic form on~$\End(\ss)$ and \cite[Chapter~V, Lemma~1.2]{Sil09}, this implies
  \begin{multline*}
    \ord_{\Bss}(\hphi_0 \hphi - \hphi \hphi_0 )
    =
    \ord_p(\nr(\hphi_0 \hphi - \hphi \hphi_0))
    =
    \ord_p(\deg(\phi_0 \phi - \phi \phi_0))
    \\ \le
    \log_p (\deg(\phi_0 \phi - \phi \phi_0))
    \le
    \log_p (4\deg(\phi_0)\deg(\phi))
    =
    \log_p(4 \deg(\phi_0) | D_{\sm}|).
  \end{multline*}
  Together Lemma~\ref{l:fixed-to-fixer}, \eqref{eq:67} and the inequality ${|D_{\sm}| \ge 2}$, this implies the desired estimate.
\end{proof}

\subsection{Proof of Theorem~\ref{t:badly-approximable-Hauptmodul}}
\label{ss:proof-badly-approximable-Hauptmodul}
In the case where ${v = \infty}$, Theorem~\ref{t:badly-approximable-Hauptmodul} is a direct consequence of the following proposition, and in the case where~$v$ is a prime number, Theorem~\ref{t:badly-approximable-Hauptmodul} is a direct consequence of Proposition~\ref{p:badly-approximable-j}, the following proposition and Lemma~\ref{l:Hauptmodul-sm} below.

\begin{proposition}
  \label{p:badly-approximable-conditional}
  Let~$f$ be a nonconstant modular function defined over~$\Qalg$, let~$\Phi(X, Y)$ be a modular polynomial of~$f$ in~$\Qalg[X, Y]$ and let~$v$ be in~$\VQ$.
  Furthermore, let~$\alpha$ in~$\Qalg$ be a non-cuspidal value of~$f$ if ${v = \infty}$, or such that every root of~$\Phi(X, \alpha)$ is badly approximable in~$\Cv$ by the singular moduli of the $j$\nobreakdash-invariant if~$v$ is a prime number.
  Then, $\alpha$ is badly approximable in~$\Cv$ by the singular moduli of~$f$.
\end{proposition}

In the case where ${v = \infty}$, the hypothesis that~$\alpha$ is a non-cuspidal value of~$f$ is necessary by Proposition~\ref{p:cuspidal-are-approximated}$(i)$.
In the case where~$v$ is a prime number and~$\alpha$ is an omitted value of~$f$, the hypothesis on~$\alpha$ is automatically satisfied because the polynomial~$\Phi(X, \alpha)$ is constant (Proposition~\ref{p:omitted-and-cuspidal}$(i)$).

\begin{proof}[Proof of Proposition~\ref{p:badly-approximable-conditional}]
  Note that~$\Phi(X, Y)$ is irreducible in~$\Cv[X, Y]$ and that~$\Phi(X, \alpha)$ is nonzero by Proposition~\ref{p:omitted-and-cuspidal}.
  Denote by~$Z$ the (possibly empty) finite set of zeros of~$\Phi(X, \alpha)$ in~$\Cv$.
  Furthermore, let~$C_3$, $\theta$ and~$\eta$ be given by Lemma~\ref{l:reduction-to-j}.

  Suppose ${v = \infty}$.
  By \cite[Lemmas~5 and~8 and formula~(11)]{Hab15} there are constants ${A > 0}$ and~$B$ such that for every~$z_0$ in~$Z$ and every singular modulus~$\sm$ of the $j$\nobreakdash-invariant different from~$z_0$, we have
  \begin{equation}
    \label{eq:68}
    - \log |\sm - z_0|_v
    \le
    A \log |D_{\sm}| + B.
  \end{equation}
  On the other hand, Proposition~\ref{p:omitted-and-cuspidal}$(ii)$, Lemma~\ref{l:reduction-to-j} and our hypothesis that~$\alpha$ is a non-cuspidal value of~$f$ imply that~$Z$ is nonempty and that for every quadratic imaginary number~$\tau$ in~$\H$ such that~$f(\tau)$ is sufficiently close to~$\alpha$, we have
  \begin{equation}
    \label{eq:69}
    \min \{ |j(\tau) - z_0|_v \colon z_0 \in Z \}
    \le
    C_3 |f(\tau) - \alpha|_v^{\theta}.
  \end{equation}
  Together with~\eqref{eq:68} and Proposition~\ref{p:singular-moduli}$(iii)$, this implies that~$\alpha$ is badly approximable in~$\C$ by the singular moduli of~$f$.

  It remains to consider the case where~$v$ is a prime number~$p$.
  Recall that~$C_3$ and~$\eta$ are given by Lemma~\ref{l:reduction-to-j} and put ${r \= C_3^{- \frac{1}{\eta}}}$.
  In the case where~$\alpha$ is an omitted value of~$f$, the desired assertion is given by Proposition~\ref{p:cuspidal-are-approximated}$(ii)$.
  Suppose that~$\alpha$ is a value of~$f$, so~$\Phi(X, \alpha)$ is nonconstant by Proposition~\ref{p:omitted-and-cuspidal}$(i)$.
  In particular, $Z$ is nonempty.
  Proposition~\ref{p:singular-moduli}$(iii)$ with~$f$ replaced by~$j$ and our hypotheses imply that there are constants ${A > 0}$ and~$B$ such that for every~$z_0$ in~$Z$ and every singular modulus~$\sm$ of the $j$\nobreakdash-invariant different from~$z_0$ we have~\eqref{eq:68}.
  On the other hand, reducing~$r$ if necessary Lemma~\ref{l:reduction-to-j} implies that for every quadratic imaginary number~$\tau$ in~$\H$ such that~$f(\tau)$ is in~$\bfD_p(\alpha, r)$ the singular modulus~$j(\tau)$ satisfies either~\eqref{eq:69} or
  \begin{equation}
    \label{eq:70}
    |j(\tau)|_p
    >
    C_3^{-1} |f(\tau) - \alpha|_p^{-\eta}
    >
    1.
  \end{equation}
  This last chain of inequalities is impossible since~$j(\tau)$ is an algebraic integer.
  We thus have~\eqref{eq:69}.
  Together with~\eqref{eq:68} and Proposition~\ref{p:singular-moduli}$(iii)$, this implies that~$\alpha$ is badly approximable in~$\Cp$ by the singular moduli of~$f$.
\end{proof}

\begin{lemma}
  \label{l:Hauptmodul-sm}
  Let~$h$ be a \emph{Hauptmodul} defined over~$\Qalg$ and let~$\Phi(X, Y)$ be a modular polynomial of~$h$ in~$\Qalg[X, Y]$.
  Then, for every singular modulus~$\hsm_0$ of~$h$, every root of~$\Phi(X, \hsm_0)$ is a singular modulus of the $j$\nobreakdash-invariant.
\end{lemma}

The proof of this lemma is after the following one.

\begin{lemma}
  \label{l:rationalizable}
  Let~$\gamma$ be an element of~$\SL(2, \R)$ that is contained in a subgroup of~$\SL(2, \R)$ commensurable to~$\SL(2, \Z)$.
  Then, there are integers $a$, $b$, $c$ and~$d$ such that ${ad - bc > 0}$ and such that for every~$\tau$ in~$\H$ we have ${\gamma(\tau) = \frac{a \tau + b}{c\tau + d}}$.
  In particular, the image by~$\gamma$ of a quadratic imaginary number in~$\H$ is also quadratic imaginary.
\end{lemma}

\begin{proof}
  Let~$\Gamma$ be a subgroup of~$\SL(2, \R)$ commensurable to~$\SL(2, \Z)$ containing~$\gamma$.
  Then, the set of cusps of~$\Gamma$ is equal to that of~$\SL(2, \Z)$, which is equal to~$\P^1(\Q)$ \cite[Proposition~1.30]{Shi71}.
  It follows that ${\gamma(\P^1(\Q)) = \P^1(\Q)}$, and in particular that~$\gamma(\infty)$ is in~$\P^1(\Q)$.
  In the case where ${\gamma(\infty) \neq \infty}$, let~$\hgamma$ be an element of~$\SL(2, \R)$ such that for every~$\tau$ in~$\H$ we have
  \begin{equation}
    \hgamma(\tau)
    =
    \frac{1}{\gamma(\tau) - \gamma(\infty)}.
  \end{equation}
  Otherwise, put ${\hgamma \= \gamma}$.
  In all of the cases, we have ${\hgamma(\Q) = \Q}$ and therefore there is~$\lambda$ in~$\Q$ such that ${\lambda > 0}$ and such that for every~$\tau$ in~$\H$ we have ${\hgamma(\tau) = \lambda \tau + \hgamma(0)}$.
  Since~$\hgamma(0)$ is in~$\Q$, this implies the desired assertion for~$\hgamma$ and therefore for~$\gamma$.
\end{proof}

\begin{proof}[Proof Lemma~\ref{l:Hauptmodul-sm}]
  Let~$\tau_0$ be a quadratic imaginary number in~$\H$ such that ${h(\tau_0) = \hsm_0}$.
  Note that~$\Phi(X, Y)$ is irreducible over~$\C$, so by Proposition~\ref{p:modular-functions} for each root~$\sm$ of~$\Phi(X, \hsm_0)$ there is~$\tau$ in~$\H$ such that
  \begin{equation}
    \label{eq:71}
    \sm = j(\tau)
    \text{ and }
    \hsm_0 = h(\tau).
  \end{equation}
  Since~$h$ is a \emph{Hauptmodul}, there is~$\gamma$ in the stabilizer of~$h$ in~$\SL(2, \R)$ such that ${\gamma(\tau_0) = \tau}$.
  By Lemma~\ref{l:rationalizable}, $\tau$ is also a quadratic imaginary number and therefore~$\sm$ is a singular modulus of the $j$\nobreakdash-invariant.
\end{proof}

\section{Proof of Theorems~\ref{t:S-units-genus-0} and~\ref{t:units-are-rare}}
\label{s:proof-S-units}

In this section we prove the following theorem and we deduce from it Theorems~\ref{t:S-units-genus-0} and~\ref{t:units-are-rare}.

\begin{custtheo}{A'}
  \label{t:S-units-conditional}
  Let~$f$ be a nonconstant modular function defined over~$\Qalg$ and~$\Phi(X, Y)$ a modular polynomial of~$f$ in~$\Qalg[X, Y]$.
  Moreover, let~$\alpha$ in~$\Qalg$ be a non-cuspidal value of~$f$ and let~$S$ be a finite set of prime numbers~$p$ such that every root of~$\Phi(X, \alpha)$ is badly approximable in~$\Cp$ by the singular moduli of the $j$\nobreakdash-invariant.
  Then, there are at most finitely many singular moduli~$\fsm$ of~$f$ such that ${\fsm - \alpha}$ is an $S$\nobreakdash-unit.
\end{custtheo}

Theorem~\ref{t:units-are-rare} is a direct consequence of Theorem~\ref{t:S-units-conditional} with ${S = \emptyset}$ and ${\alpha = 0}$ applied to~$f$ and to~$\frac{1}{f}$.
Another direct consequence of Theorem~\ref{t:S-units-conditional} is the following version of Theorem~\ref{t:units-are-rare} for $S$\nobreakdash-units, under the hypothesis that there is an affirmative solution to Conjecture~\ref{c:badly-approximable}.

\begin{coro}
  \label{c:S-units-conditional-non-weak}
  Let~$S$ be a finite set of prime numbers~$p$ such that every algebraic number is badly approximable in~$\Cp$ by the singular moduli of the $j$\nobreakdash-invariant.
  Moreover, let~$f$ be a nonconstant modular function defined over~$\Qalg$ that is not a weak modular unit.
  Then, there are at most a finite number of singular moduli of~$f$ that are $S$\nobreakdash-units.
\end{coro}

The proof of Theorem~\ref{t:S-units-conditional} is given in Section~\ref{ss:proof-S-units-conditional}.
In Section~\ref{ss:proof-S-units-genus-0} we prove Theorem~\ref{t:S-units-genus-0} and the following corollary of Theorem~\ref{t:S-units-conditional}.
To state it, recall that a subgroup of~$\SL(2, \R)$ is a \emph{congruence group}, if for some~$N$ in~$\N$ it contains
\begin{equation}
  \left\{ \bigl( \begin{smallmatrix} a & b \\ c & d \end{smallmatrix} \bigr) \in \SL(2, \Z) \colon a, d \equiv 1 \mod N \text{ and } b, c \equiv 0 \mod N \right\}
\end{equation}
as a finite index subgroup.
The following corollary shows that an affirmative solution to Conjecture~\ref{c:badly-approximable} would yield a version of Theorem~\ref{t:S-units-genus-0} for a general congruence or genus zero group and a general algebraic value.

\begin{coro}
  \label{c:S-units-conditional-non-unit}
  Let~$f$ be a nonconstant modular function defined over~$\Qalg$ for a congruence or a genus zero group and let~$\alpha$ in~$\Qalg$ be a value of~$f$.
  Suppose that for every prime number~$p$, every algebraic number is badly approximable in~$\Cp$ by the singular moduli of the $j$\nobreakdash-invariant.
  Then, there are at most a finite number of singular moduli~$\fsm$ of~$f$ such that ${\fsm - \alpha}$ is an $S$\nobreakdash-unit.
\end{coro}

Corollary~\ref{c:S-units-conditional-non-unit} applied to~$f$ and to~$\frac{1}{f}$ with ${\alpha = 0}$, shows that an affirmative solution to Conjecture~\ref{c:badly-approximable} would yield a version Theorem~\ref{t:units-are-rare} that holds under the weaker hypothesis that~$f$ is not a modular unit, but that is restricted to congruence or to genus zero groups.

\subsection{Proof of Theorem~\ref{t:S-units-conditional}}
\label{ss:proof-S-units-conditional}
Recall that~$\VQ$, and for each~$v$ in~$\VQ$, the norm field~$(\Cv, | \cdot |_v)$, are defined in Section~\ref{ss:disperse}.
Given a finite extension~$K$ of~$\Q$ inside~$\Qalg$, denote by~$\VK$ the set of all norms on~$K$ that for some~$v$ in~$\VQ$ coincide with~$| \cdot |_v$ on~$\Q$.
For such~$w$ and~$v$, write~$w \mid v$, let~$(\Cw, | \cdot |_w)$ be a completion of an algebraic closure of~$(K, w)$ and denote by~$K_w$ the closure of~$K$ inside~$\Cw$.
Note that~$(\Cw, |\cdot|_w)$ and~$(\Cv, | \cdot |_v)$ are isomorphic as normed fields and that~$(K_w, | \cdot |_w)$ is a completion of~$(K, w)$.
Moreover, identify the algebraic closure of~$K$ inside~$\Cw$ with~$\Qalg$, put
\begin{equation}
  \label{eq:72}
  \nu_w \= \frac{[K_w : \Q_v]}{[K : \Q]}
  \text{ and }
  \| \cdot \|_w
  \=
  | \cdot |_w^{\nu_w},
\end{equation}
and for all~$\alpha$ in~$\Cv$ and~$r > 0$ put
\begin{displaymath}
  \bfD_w(\alpha, r)
  \=
  \{ z \in \Cv \colon |z - \alpha|_w < r \}.
\end{displaymath}
Note that ${\nu_w \le 1}$, see, \emph{e.g.}, \cite[Corollary~1.3.2]{BomGub06}.

Let
\begin{equation}
  \label{eq:73}
  \log^+ \colon [0, \infty [ \to \R
  \text{ and }
  \log^- \colon [0, \infty [ \to \R \cup \{ - \infty \}
\end{equation}
be the functions defined by
\begin{displaymath}
  \log^+(x)
  \=
  \log \max \{1, x \}
  \text{ and }
  \log^-(x)
  \=
  \log \min \{1, x \}.
\end{displaymath}
Denote by ${\nh \colon \Qalg \to \R}$ the \emph{Weil} or \emph{naive height}, which for each finite extension~$K$ of~$\Q$ inside~$\Qalg$ and every~$\alpha$ in~$K$ is given by
\begin{equation}
  \label{eq:74}
  \nh(\alpha)
  =
  \sum_{w \in \VK} \log^+ \| \alpha \|_w.
\end{equation}
In this formula, the right side is independent of the finite extension~$K$ of~$\Q$ inside~$\Qalg$ containing~$\alpha$.
Note that by the triangle inequality, for all~$\alpha_0$ and~$\alpha$ in~$\Qalg$ we have
\begin{equation}
  \label{eq:75}
  \nh(\alpha - \alpha_0)
  \ge
  \nh(\alpha) - \nh(\alpha_0) - \log2.
\end{equation}

The proof of Theorem~\ref{t:S-units-conditional} is given after a couple of lemmas.

\begin{lemma}
  \label{l:Mahler}
  Let~$\alpha$ in~$\Qalg$ be given and let~$K$ be a finite extension of~$\Q$ inside~$\Qalg$ that does not necessarily contain~$\alpha$.
  Then, we have
  \begin{equation}
    \label{eq:76}
    \nh(\alpha)
    =
    - \frac{1}{\# \oO_{K}(\alpha)} \sum_{w \in \VK} \sum_{\alpha' \in \oO_{K}(\alpha)} \log^- \|\alpha'\|_w.
  \end{equation}
\end{lemma}

\begin{proof}
  Let~$\hK$ be a finite extension of~$K$ inside~$\Qalg$ containing~$\oO_K(\alpha)$, let
  \begin{equation}
    \label{eq:77}
    P(z) = z^d + a_{d - 1} z^{d - 1} + \cdots + a_0
  \end{equation}
  be the minimal polynomial of~$\alpha$ in~$K[x]$ and for every~$w$ in~$\VK$ put
  \begin{equation}
    \label{eq:78}
    \| P \|_w
    \=
    \max \{ \| a_j \|_w \colon j \in \{0, \ldots, d - 1 \} \}.
  \end{equation}

  For every archimedean~$w$ in~$\VK$ (resp.~$\whw$ in~$\VhK$) put ${S^1_w \= \{ z \in \C_w \colon |z|_w = 1 \}}$ (resp. ${S^1_{\whw} \= \{ z \in \C_{\whw} \colon |z|_{\whw} = 1 \}}$) and denote by~$\lambda_w$ (resp.~$\lambda_{\whw}$) the Haar measure of this group.
  Then, we have
  \begin{equation}
    \label{eq:79}
    \begin{split}
      \sum_{\alpha' \in \oO_K(\alpha)} \log^+ \| \alpha' \|_{w}
      & =
        \int \log \| P(z) \|_{w} \dd \lambda_{w}(z)
      \\ & =
           \sum_{\substack{\whw \in \VhK \\ \whw \mid w}} \int \log \| P(z) \|_{\whw} \dd \lambda_{\whw}(z)
      \\ & =
           \sum_{\substack{\whw \in \VhK \\ \whw \mid w}} \sum_{\alpha' \in \oO_K(\alpha)} \log^+ \| \alpha' \|_{\whw},
    \end{split}
  \end{equation}
  see, \emph{e.g.}, \cite[Corollary~1.3.2 and Proposition~1.6.5]{BomGub06}.
  Similarly, for every non-archimedean~$w$ in~$\VK$ we have
  \begin{equation}
    \label{eq:80}
    \begin{split}
      \sum_{\alpha' \in \oO_K(\alpha)} \log^+ \| \alpha' \|_{w}
      & =
        \log \| P \|_{w}
      \\ & =
           \sum_{\substack{\whw \in \VhK \\ \whw \mid w}} \log \| P \|_{\whw}
      \\ & =
           \sum_{\substack{\whw \in \VhK \\ \whw \mid w}} \sum_{\alpha' \in \oO_K(\alpha)} \log^+ \| \alpha' \|_{\whw},
    \end{split}
  \end{equation}
  see, \emph{e.g.}, \cite[Corollary~1.3.2 and Lemma~1.6.3]{BomGub06}.
  Combined with~\eqref{eq:79} and with the Galois invariance of the Weil height, see, \emph{e.g.}, \cite[Proposition~1.5.17]{BomGub06}, this implies
  \begin{equation}
    \label{eq:81}
    \begin{split}
      \# \oO_K(\alpha) \nh(\alpha)
      & =
        \sum_{\whw \in \VhK} \sum_{\alpha' \in \oO_{K}(\alpha)} \log^+ \|\alpha'\|_{\whw}
      \\ & =
           \sum_{w \in \VK} \sum_{\alpha' \in \oO_{K}(\alpha)} \log^+ \|\alpha'\|_{w}.
    \end{split}
  \end{equation}
  On the other hand, by the product formula applied to the element~$\prod_{\alpha' \in \oO_K(\alpha)} \alpha'$ of~$K$ we have
  \begin{equation}
    \label{eq:82}
    \prod_{w \in \VK} \prod_{\alpha' \in \oO_K(\alpha)} \| \alpha' \|_w
    =
    1,
  \end{equation}
  see, \emph{e.g.}, \cite[Proposition~1.4.4]{BomGub06}.
  Together with~\eqref{eq:81}, this implies the desired identity.
\end{proof}

The following lemma is an extension of~\cite[Lemma~3]{Hab15} to the more general setting considered here.

\begin{lemma}
  \label{l:Colmez-bound-general}
  Let~$K$ be a finite extension of~$\Q$ inside~$\Qalg$ and let~$f$ be a nonconstant modular function defined over~$K$.
  Then, for every singular modulus~$\fsm_0$ of~$f$ there are constants~$A_0 > 0$ and~$B_0$ such that for every singular modulus~$\fsm$ of~$f$ we have
  \begin{equation}
    \label{eq:Colmez-bound-general}
    \nh(\fsm - \fsm_0)
    \ge
    A_0 \log (\# \oO_{K}(\fsm)) + B_0.
  \end{equation}
\end{lemma}

\begin{proof}
  Let~$\Phi(X, Y)$ be a modular polynomial of~$f$ in~$\Qalg[X, Y]$ and denote by~$\delta_X$ and~$\delta_Y$ the degree of~$\Phi(X, Y)$ in~$X$ and~$Y$, respectively.
  For each~$k$ in~$\{0, \ldots, \delta_X\}$ let~$P_k(Y)$ in~$K[Y]$ be the coefficient of~$X^k$ in~$\Phi(X, Y)$, and let~$M_k > 0$ be such that for every~$\alpha$ in~$\Qalg$ we have
  \begin{equation}
    \label{eq:83}
    \nh(P_k(\alpha))
    \le
    \deg(P_k) \nh(\alpha) + M_k,
  \end{equation}
  see, \emph{e.g.}, \cite[Chapter~VIII, Theorem~5.6]{Sil09}.
  Thus, if we put
  \begin{equation}
    \label{eq:84}
    \Delta
    \=
    \sum_{k = 0}^{\delta_X} \deg(P_k)
    \text{ and }
    M
    \=
    \sum_{k = 0}^{\delta_X} M_k,
  \end{equation}
  then for every quadratic imaginary number~$\tau$ in~$\H$ that is not a pole of~$f$ we have
  \begin{equation}
    \label{eq:85}
    \nh(j(\tau)) - {\delta_X} \log 2
    \le
    \sum_{k = 0}^{{\delta_X}}\nh(P_k(f(\tau)))
    \le
    \Delta \nh(f(\tau)) + M,
  \end{equation}
  see, \emph{e.g.}, \cite[Chapter~VIII, Theorem~5.9]{Sil09}.
  Combined with~\eqref{eq:75}, Proposition~\ref{p:singular-moduli}$(iii)$ and \cite[Lemma~3]{Hab15}, which is based on Colmez' lower bound \cite[\emph{Th{\'e}or{\`e}me~1}]{Col98}, we obtain the desired estimate.
\end{proof}

\begin{proof}[Proof of Theorem~\ref{t:S-units-conditional}]
  Let~$K$ be a finite extension of~$\Q$ containing~$\alpha$ and the coefficients of~$\Phi$ and denote by~$S_0$ the set of all~$w$ in~$\VK$ such that for some~$v$ in ${S \cup \{ \infty \}}$ we have~$w \mid v$.
  Let~$A_0$ be the constant given by Lemma~\ref{l:Colmez-bound-general}.
  By Propositions~\ref{p:singular-moduli}$(ii)$ and~\ref{p:badly-approximable-conditional}, for every~$v$ in~$\VQ$ there are constants ${A_v > 0}$ and~$B_v$ such that for every singular modulus~$\fsm$ of~$f$ we have
  \begin{equation}
    \label{eq:86}
    - \log |\fsm - \alpha|_v
    \le
    A_v \log (\# \oO_K(\fsm)) + B_v.
  \end{equation}
  For every~$w$ in~$\VK$ such that~$w \mid v$, put ${A_w \= A_v}$ and ${B_w \= B_v}$.

  Suppose that there is a sequence of pairwise distinct singular moduli~$(\fsm_n)_{n = 1}^{\infty}$ of~$f$ such that for every~$n$ the difference~$\fsm_n - \alpha$ is an $S$\nobreakdash-unit.
  By Proposition~\ref{p:singular-moduli}$(iv)$, we have
  \begin{equation}
    \label{eq:87}
    \# \oO_{K}(\fsm_n) \to \infty
    \text{ as }
    n \to \infty.
  \end{equation}
  Together with Theorem~\ref{t:disperse} applied to each~$v$ in ${S \cup \{ \infty \}}$, this implies that there is~$r$ in~$]0, 1[$ such that for every~$w$ in~$S_0$ and every sufficiently large~$n \ge 1$, we have
  \begin{displaymath}
    \# (\oO_{K}(\fsm_n) \cap \bfD_w(\alpha, r))
    \le
    \frac{A_0}{2 A_w (\# S_0 + 1)} \# \oO_{K}(\fsm_n).
  \end{displaymath}
  Thus, for every sufficiently large~$n$ we have
  \begin{displaymath}
    \frac{\# (\oO_{K}(\fsm_n - \alpha) \cap \bfD_w(0, r))}{\# \oO_{K}(\fsm_n - \alpha)}
    =
    \frac{\# (\oO_{K}(\fsm_n) \cap \bfD_w(\alpha, r))}{\# \oO_{K}(\fsm_n)}
    \le
    \frac{A_0}{2 A_w (\# S_0 + 1) }.
  \end{displaymath}
  Combined with Lemma~\ref{l:Mahler}, \eqref{eq:86}, the fact that for every~$w$ in~$\VK$ we have~$\nu_{w} \le 1$, and our assumption that~$\fsm_n - \alpha$ is an $S$\nobreakdash-unit, this implies that for some constant~$B$ independent of~$n$ we have
  \begin{multline*}
    \nh(\fsm_n - \alpha)
    \\
    \begin{aligned}
      & =
        - \frac{1}{\# \oO_{K}(\fsm_n - \alpha)} \sum_{w \in S_0} \left( \sum_{\substack{\beta \in \oO_{K}(\fsm_n - \alpha) \\ |\beta|_w < r}} \log \|\beta\|_w + \sum_{\substack{\beta \in \oO_{K}(\fsm_n - \alpha) \\ r \le |\beta|_w < 1}} \log \|\beta\|_w \right)
      \\ & \le
           \sum_{w \in S_0} \frac{\# (\oO_{K}(\fsm_n - \alpha) \cap \bfD_v(0, r))}{\# \oO_{K}(\fsm_n - \alpha)} (A_w \log (\# \oO_{K}(\fsm_n)) + B_w) + (\# S_0 + 1) \log \frac{1}{r}
      \\ & \le
           \frac{A_0}{2} \log (\# \oO_{K}(\fsm_n)) + B.
    \end{aligned}
  \end{multline*}
  In view of~\eqref{eq:87}, letting ${n \to \infty}$ we obtain a contradiction with Lemma~\ref{l:Colmez-bound-general} that completes the proof of the theorem.
\end{proof}

\subsection{Proof of Theorem~\ref{t:S-units-genus-0} and Corollary~\ref{c:S-units-conditional-non-unit}}
\label{ss:proof-S-units-genus-0}

The proofs are given after a few of lemmas.

\begin{lemma}
  \label{l:rational-def-Qalg}
  Let~$V$ be a projective curve defined over~$\Qalg$ and let~$g_0$ be a rational function defined on~$V$ over~$\C$ such that every zero and every pole of~$g_0$ is in~$V(\Qalg)$.
  If there exists~$z_0$ in~$V(\Qalg)$ such that ${g_0(z_0) = 1}$, then~$g_0$ is defined over~$\Qalg$.
\end{lemma}

\begin{proof}
  If~$g_0$ is constant, then it is equal to~$1$ and the result follows.
  Assume~$g_0$ is nonconstant.
  Let~$\sigma$ in~$\Aut(\C|\Qalg)$ be given and denote by~$g_0^{\sigma}$ the image of~$g_0$ under the action of~$\sigma$ on rational functions.
  The hypothesis that every zero and every pole of~$g_0$ is in~$V(\Qalg)$ implies that~$g_0$ and~$g_0^{\sigma}$ have the same zeros and poles, and that the corresponding multiplicities are the same.
  This implies that~$\frac{g_0}{g_0^{\sigma}}$ is constant.
  Evaluating at~$z_0$ and using
  \begin{equation*}
    g_0^{\sigma}(z_0)
    =
    \sigma(g_0(z_0))
    =
    \sigma(1)
    =
    1
    =
    g_0(z_0),
  \end{equation*}
  we conclude that ${g_0 = g_0^{\sigma}}$.
  Since~$\sigma$ is arbitrary, we get that~$g_0$ is defined over~$\Qalg$.
\end{proof}

\begin{lemma}
  \label{l:Hauptmodulization}
  Let~$f$ be a nonconstant modular function defined over~$\Qalg$ for a genus zero subgroup of~$\SL(2, \R)$, such that~$0$ is a value of~$f$.
  Then, there is a holomorphic \emph{Hauptmodul}~$h$ defined over~$\Qalg$ and a nonconstant rational function~$R(X)$ in~$\Qalg(X)$, such that~$0$ is a (non-cuspidal) value of~$h$ and we have
  \begin{equation}
    \label{eq:88}
    R(0) = 0
    \text{ and }
    R(h) = f.
  \end{equation}
\end{lemma}

\begin{proof}
  Our hypotheses imply that the stabilizer~$\Gamma$ of~$f$ in~$\SL(2, \R)$ is of genus zero.
  Put ${\hGamma \= \Gamma \cap \SL(2, \Z)}$ and note that~$\hGamma$ has finite index in~$\Gamma$ and in~$\SL(2, \Z)$.
  Denote by ${\Pi \colon X(\hGamma) \to X(\Gamma)}$ the map induced by the identity on~$\H$ and by~$\whj$ and~$\whf$ the meromorphic functions defined on~$X(\hGamma)$ that are induced by~$j$ and~$f$, respectively.
  Note that the meromorphic function~$f_0$ defined on~$X(\Gamma)$ that is induced by~$f$ satisfies ${\whf = f_0 \circ \Pi}$.

  First, we show that~$X(\hGamma)$ can be defined over~$\Qalg$ in such a way that~$\whj$ and~$\whf$ correspond to rational functions defined over~$\Qalg$.
  Let~$\Phi(X, Y)$ in~$\Qalg[X, Y]$ be a modular polynomial for~$f$ and denote by~$Z(\Phi)$ the zero set of~$\Phi$ in~$\C^2$.
  Denote by~$P$ the finite subset of~$X(\hGamma)$ formed by the poles of~$\whj$ and those of~$\whf$ and let~$\hvarphi$ be the function defined by
  \begin{equation}
    \label{eq:89}
    \begin{array}{rcl}
      \hvarphi \colon X(\hGamma) \ssetminus P & \to & Z(\Phi)
      \\
      z & \mapsto & (\whj(z), \whf(z)).
    \end{array}
  \end{equation}
  By definition of~$\hGamma$, for every~$\gamma$ in~$\SL(2, \Z)$ outside~$\hGamma$ the meromorphic functions~$\whf$ and~$\whf \circ \gamma$ are different.
  Thus, the set~$E_{\gamma}$ of all points of~$X(\hGamma)$ at which these functions agree is finite.
  Let~$\cR$ be a set of representatives of the right cosets of~$\hGamma$ in~$\SL(2, \Z)$ that are different from~$\hGamma$ and put
  \begin{equation}
    \label{eq:90}
    E
    \=
    P \cup \bigcup_{\gamma \in \cR} E_{\gamma}.
  \end{equation}
  Then, $\cR$ and~$E$ are both finite and the restriction of~$\hvarphi$ to ${X(\hGamma) \ssetminus E}$ is injective.
  Thus, $\hvarphi$ induces a birational isomorphism between~$X(\hGamma)$ and~$Z(\Phi)$.
  By \cite[Chapter~I, Corollary~6.11]{Har77}, there exist a smooth projective curve~$V$ defined over~$\Qalg$ and birational isomorphisms
  \begin{equation}
    \phi \colon X(\hGamma) \dashrightarrow V(\C)
    \text{ and }
    \psi \colon V(\C) \dashrightarrow Z(\Phi),
  \end{equation}
  such that~$\psi$ is defined over~$\Qalg$ and~$\psi \circ \phi$ defines the same birational isomorphism as~$\hvarphi$.
  Note that~$\phi$ extends to an isomorphism ${X(\hGamma) \to V(\C)}$, see, \emph{e.g.} \cite[Chapter~I, Proposition~6.8]{Har77}.
  Under this isomorphism, $\whj$ and~$\whf$ correspond to the composition of~$\psi$ with the projections on the first and second coordinate on~$Z(\Phi)$, respectively, both of which are defined over~$\Qalg$.
  Thus, $\phi$ and~$V(\C)$ induce an algebraic structure on~$X(\hGamma)$ over~$\Qalg$ for which~$\whj$ and~$\whf$ correspond to rational functions defined over~$\Qalg$.
  In what follows we fix this algebraic structure on~$X(\hGamma)$.

  Next, we show that~$X(\Gamma)$ can be defined over~$\Qalg$ in such a way that~$\Pi$ corresponds to a rational function defined over~$\Qalg$.
  To do this, it is sufficient to show that there is a biholomorphic map ${h_0 \colon X(\Gamma) \to \P^1(\C)}$ for which the composition~$h_0 \circ \Pi$ is defined over~$\Qalg$.
  Choose pairwise distinct numbers~$\alpha_0$, $\alpha_1$ and~$\alpha_{\infty}$ in~$\Qalg$ and choose $z_0$, $z_1$ and~$z_{\infty}$ in~$f_0^{-1}(\alpha_0)$, $f_0^{-1}(\alpha_1)$ and~$f_0^{-1}(\alpha_{\infty})$, respectively.
  Since~$X(\Gamma)$ is of genus zero, there is a biholomorphic map ${h_0 \colon X(\Gamma) \to \P^1(\C)}$ mapping~$z_0$, $z_1$ and~$z_{\infty}$ to~$0$, $1$ and~$\infty$, respectively.
  Thus, $z_1$ and every zero and every pole of~$h_0 \circ \Pi$ is defined over~$\Qalg$ and therefore~$h_0 \circ \Pi$ is defined over~$\Qalg$ by Lemma~\ref{l:rational-def-Qalg}.
  We conclude that~$X(\Gamma)$ and~$\Pi$ are both defined over~$\Qalg$ with respect to the algebraic structure induced by~$h_0$.
  In what follows we fix this algebraic structure on~$X(\Gamma)$.
  Note that~$f_0$ is also defined over~$\Qalg$, because~$\whf$ is.
  Since the cuspidal values of~$f_0$ are defined over~$\Qalg$ (Proposition~\ref{p:omitted-and-cuspidal}), it follows that each cusp of~$X(\Gamma)$ is also defined over~$\Qalg$.

  To complete the proof of the lemma, note that our hypothesis that~$0$ is a value of~$f$ implies that there is~$\tau_0$ in~$\H$ such that ${f(\tau_0) = 0}$.
  The point~$z_0$ of~$X(\Gamma)$ defined by~$\tau_0$ is not a cusp of~$X(\Gamma)$ and is defined over~$\Qalg$.
  It follows that there is a biholomorphic map ${X(\Gamma) \to \P^1(\C)}$ defined over~$\Qalg$ mapping~$z_0$ to~$0$ and the cusp of~$X(\Gamma)$ defined by~$i\infty$ to~$\infty$.
  The lift~$h$ to~$\H$ of this function is a holomorphic \emph{Hauptmodul} for~$\Gamma$ that is defined over~$\Qalg$ and satisfies ${h(\tau_0) = 0}$.
  From the results proved in the previous paragraphs, it follows that there is~$R(X)$ in~$\Qalg(X)$ such that ${R(h) = f}$.
  Evaluating at~$\tau_0$, we conclude that ${R(0) = 0}$.
  Finally, note that since~$h$ is a \emph{Hauptmodul} and~$0$ is a value of~$h$, we have that~$0$ is a non-cuspidal value of~$h$.
\end{proof}

\begin{lemma}
  \label{l:factorization-genus-0}
  Let~$h$ be a holomorphic \emph{Hauptmodul} defined over~$\Qalg$, let~$R(X)$ in~$\Qalg(X)$ be nonconstant and such that~${R(0) = 0}$ and put ${f \= R(h)}$.
  Suppose that for every finite set of prime numbers~$S$, there are at most a finite number of singular moduli of~$h$ that are $S$\nobreakdash-units.
  Then, $f$ is a nonconstant modular function defined over~$\Qalg$ and for every finite set of prime numbers~$S$ there are at most a finite number of singular moduli of~$f$ that are $S$\nobreakdash-units.
\end{lemma}

\begin{proof}
  That~$f$ is a nonconstant modular function follows from the fact that~$h$ has the same property and that~$R(X)$ is nonconstant.
  To show that~$f$ is defined over~$\Qalg$, note that~$h$ is algebraically dependent with the $j$\nobreakdash-invariant over~$\Qalg$ because~$h$ is defined over~$\Qalg$.
  It follows that~$f$ is also algebraically dependent with the $j$\nobreakdash-invariant over~$\Qalg$ and therefore that it is defined over~$\Qalg$.

  Let~$S$ be a finite set of prime numbers.
  By Corollary~\ref{c:elliptic-S-units}$(i)$ there is a finite set of prime numbers~$S_0$ such that every singular modulus of~$h$ is an $S_0$\nobreakdash-integer.
  Our hypotheses that~$R(X)$ is nonconstant and ${R(0) = 0}$ imply that there are~$\ell$ in~$\N$, $a$ in~${\Qalg \ssetminus \{ 0 \}}$ and monic polynomials~$P(X)$ and~$Q(X)$ in~$\Qalg[X]$, such that
  \begin{equation}
    \label{eq:91}
    P(0)
    \neq
    0,
    Q(0)
    \neq
    0
    \text{ and }
    R(X)
    =
    a X^{\ell} \frac{P(X)}{Q(X)}.
  \end{equation}
  Let~$S_1$ be a finite set of prime numbers containing~$S$ and~$S_0$ and such that each of the coefficients of~$P(X)$ and of~$Q(X)$ is an $S_1$\nobreakdash-integer and each of the numbers~$a$, $P(0)$ and~$Q(0)$ is an $S_1$\nobreakdash-unit.

  By hypothesis, the set~$U$ of all those singular moduli of~$h$ that are~$S_1$\nobreakdash-units is finite.
  Let~$\fsm$ be a singular modulus of~$f$ outside the finite set~$R(U)$, let~$\tau$ be a quadratic imaginary number such that ${f(\tau) = \fsm}$ and put ${\hsm \= h(\tau)}$.
  Then, ${\fsm = R(\hsm)}$ and~$\hsm$ is a singular modulus of~$h$ outside~$U$.
  It follows that~$\hsm$ is an $S_1$\nobreakdash-integer that is not an $S_1$\nobreakdash-unit.
  That is, there is a prime number~$p$ outside~$S_1$ and~$\sigma$ in~$\Gal(\Qalg|\Q)$ such that ${|\sigma(\hsm)|_p < 1}$.
  Denote by~$P^{\sigma}(X)$ and~$Q^{\sigma}(X)$ the image of~$P(X)$ and~$Q(X)$ by the induced action of~$\sigma$ on~$\Qalg[X]$, respectively.
  In view of our choice of~$S_1$, we have
  \begin{equation}
    \label{eq:92}
    |\sigma(a)|_p
    =
    |P^{\sigma}(\sigma(\hsm))|_p
    =
    |Q^{\sigma}(\sigma(\hsm))|_p
    =
    1.
  \end{equation}
  Together with~\eqref{eq:91}, this implies
  \begin{equation}
    \label{eq:93}
    |\sigma(\fsm)|_p
    =
    \left| \sigma(a) \sigma(\hsm)^{\ell} \frac{P^{\sigma}(\sigma(\hsm))}{Q^{\sigma}(\sigma(\hsm))} \right|_p
    =
    |\sigma(\hsm)|_p^{\ell}
    <
    1.
  \end{equation}
  This proves that~$\fsm$ is not an $S_1$\nobreakdash-unit and therefore that it is not an $S$\nobreakdash-unit.
\end{proof}

\begin{proof}[Proof of Theorem~\ref{t:S-units-genus-0}]
  Put ${f_0 \= f - \fsm_0}$ and note that~$0$ is a value of~$f_0$.
  Let~$h$ and~$R$ be given by Lemma~\ref{l:Hauptmodulization} with~$f$ replaced by~$f_0$.
  Combining Proposition~\ref{p:badly-approximable-j}, Lemma~\ref{l:Hauptmodul-sm} and Theorem~\ref{t:S-units-conditional} with~$f$ replaced by~$h$ and with ${\alpha = 0}$, we obtain that for every finite set of prime numbers~$S$, there are at most a finite number of singular moduli of~$h$ that are $S$\nobreakdash-units.
  Together with Lemma~\ref{l:factorization-genus-0}, this implies that~$f_0$ has the same property.
  It follows that for every finite set of prime numbers~$S$, there are at most a finite number of singular moduli~$\fsm$ of~$f$ that ${\fsm - \fsm_0}$ is an $S$\nobreakdash-unit.
\end{proof}

The proof of Corollary~\ref{c:S-units-conditional-non-unit} is after the following lemma.

\begin{lemma}\label{l:factorization-congruence}
  Let~$f$ be a nonconstant modular function defined over~$\Qalg$ for a congruence group~$\Gamma$ contained in~$\SL(2, \Z)$.
  Then, for every cusp~$c$ of~$X(\Gamma)$ there exists~$m$ in~$\N$ and a modular unit~$g$ defined over~$\Qalg$ for~$\Gamma$ such that the following property holds.
  No cusp of~$X(\Gamma)$ different from~$c$ is a zero or a pole of the meromorphic function defined on~$X(\Gamma)$ induced by~$f^mg$.
\end{lemma}

\begin{proof}
  Let~$j_0$ and~$f_0$ be the meromorphic functions defined on~$X(\Gamma)$ induced by the $j$\nobreakdash-invariant and by~$f$, respectively.
  Moreover, let~$Z$ be the finite set of cusps of~$X(\Gamma)$ and for each~$z$ in~$Z$ denote by~$n_z$ the order of~$f_0$ at~$z$.
  If for every~$z$ in~$Z \ssetminus \{c\}$ we have ${n_z = 0}$, then the desired assertion holds with ${m = 1}$ and with~$g$ equal to the constant function equal to~$1$.
  Suppose this is not the case, so the divisor~$D$ on~$X(\Gamma)$ defined by
  \begin{equation}
    \label{eq:94}
    D
    \=
    \left(\sum_{z\in Z\ssetminus\{c\}}n_z\right)c-\sum_{z\in Z\ssetminus\{c\}}n_zz,
  \end{equation}
  is nonzero.
  Note that the degree of~$D$ is zero.
  Applying the Manin--Drinfel'd theorem repeatedly \cite[Theorem~1]{Dri73}, we obtain that there exists~$m$ in~$\N$ and a nonconstant meromorphic function~$g_0$ defined on~$X(\Gamma)$ such that the divisor of zeros and poles of~$g_0$ equals~$mD$.
  It follows that the modular function~$g$ induced by~$g_0$ is a modular unit for~$\Gamma$ such that~$f_0^mg_0$ has no zeros or poles in~$Z\ssetminus \{c\}$.
  To complete the proof of the lemma, it remains to show that there is a nonzero complex number~$s$ such that~$sg$ is defined over~$\Qalg$.
  To do this, note that the Riemann surface~$X(\Gamma)$ has a structure of projective variety defined over~$\Qalg$ for which~$j_0$ is given by a rational function defined over~$\Qalg$, see, \emph{e.g.}, \cite[Chapter~6.7]{Shi71}.
  In particular, each element of~$Z$ is defined over~$\Qalg$ with respect to this algebraic structure.
  Choose a point~$z_0$ in~$X(\Gamma) \ssetminus Z$ defined over~$\Qalg$, note that~$g_0(z_0)$ is a nonzero complex number and put ${s \= g_0(z_0)^{-1}}$.
  By Lemma~\ref{l:rational-def-Qalg} with~$g_0$ replaced by~$s g_0$, the function~$s g_0$ corresponds to a rational function on~$X(\Gamma)$ defined over~$\Qalg$.
  Since this is also the case for~$j_0$, we have that~$j_0$ and~$s g_0$ are algebraically dependent over~$\Qalg$.
  This implies that~$sg$ is defined over~$\Qalg$ and completes the proof of the lemma.
\end{proof}

\begin{proof}[Proof of Corollary~\ref{c:S-units-conditional-non-unit}]
  In the case where~$\alpha$ is non-cuspidal value of~$f$, the desired assertion follows from Theorem~\ref{t:S-units-conditional}.
  Suppose~$\alpha$ is a cuspidal value of~$f$ and let~$\Gamma$ be the stabilizer of~$f$ in~$\SL(2,\R)$.

  Suppose first that~$\Gamma$ is a congruence group, put ${\hGamma \= \Gamma \cap \SL(2,\Z)}$ and let~$c$ be a cusp of~$X(\hGamma)$ at which the meromorphic function~$f_0$ defined on~$X(\hGamma)$ induced by~$f$ takes the value~$\alpha$.
  Let~$m$ and~$g$ be given by Lemma~\ref{l:factorization-congruence} with~$f$ replaced by ${f - \alpha}$ and with~$\Gamma$ replaced by~$\hGamma$ and put ${\whf \= (f - \alpha)^m g}$.
  Then, $0$ is a non-cuspidal value of~$\whf$ or of~$\frac{1}{\whf}$ and Theorem~\ref{t:S-units-conditional} with~$\alpha$ replaced by~$0$ implies that there are at most a finite number of singular moduli of~$\whf$ that are $S$\nobreakdash-units.
  On the other hand, by Corollary~\ref{c:elliptic-S-units} there is a finite set of prime numbers~$S_0$ such that every singular modulus of~$g$ is an $S_0$\nobreakdash-unit.
  Putting ${S_1 \= S \cup S_0}$, we conclude that there are at most a finite number of singular moduli~$\fsm$ of~$f$ such that ${\fsm - \alpha}$ is an $S_1$\nobreakdash-unit.
  Since~$S_1$ contains~$S$, this implies the desired assertion.

  It remains to consider the case where~$\Gamma$ is of genus zero.
  Let~$h$ be the \emph{Hauptmodul} given by Lemma~\ref{l:Hauptmodulization} with~$f$ replaced by ${f - \alpha}$.
  Theorem~\ref{t:S-units-conditional} with~$f$ replaced by~$h$ and with~$\alpha$ replaced by~$0$, implies that for every finite set of prime numbers~$S$ there are at most a finite number of singular moduli of~$h$ that are $S$\nobreakdash-units.
  Together with Lemma~\ref{l:factorization-genus-0}, this implies that ${f - \alpha}$ has the same property.
\end{proof}

\appendix
\section{Fourier series expansion of modular functions}
\label{s:arithmetically-modular}

The goal of this section is to give conditions on a modular function to be defined over a given subfield of~$\C$.

A meromorphic function~$f$ defined on~$\H$ is \emph{periodic}, if there is~$h$ in~$\N$ such that for every~$\tau$ in~$\H$ we have ${f(\tau + h) = f(\tau)}$.
The \emph{period of~$f$} is the least~$h$ satisfying this property.
In this case, $f$ admits a Fourier series expansion at~$i \infty$ of the form
\begin{equation}
  \label{eq:95}
  f(\tau)
  =
  \sum_{n = - \infty}^{\infty} a_n \exp \left( \frac{2 \pi i n}{h} \tau \right).
\end{equation}
The function~$f$ is \emph{meromorphic} (resp. \emph{holomorphic}) \emph{at~$i \infty$}, if for every sufficiently large integer~$n$ (resp. every~$n$ in~$\N$) we have~${a_{- n} = 0}$.

Note that every modular function is periodic and therefore it admits a Fourier series expansion at~$i \infty$.
The goal of this appendix is to prove the following proposition.

\begin{proposition}
  \label{p:arithmetically-modular}
  Let~$f$ be a modular function whose Fourier series expansion at~$i \infty$ has coefficients in a subfield~$K$ of~$\C$.
  Then~$f$ is defined over~$K$.
\end{proposition}

The proof of this proposition is after the following lemma.

\begin{lemma}
  \label{l:transcendental-deformation}
  Let~$K$ be a subfield of~$\C$ and let~$\cA$ be a finite subset of~$\C$ that is not contained in~$K$.
  Then, there is a field homomorphism ${K(\cA) \to \C}$ that is the identity on~$K$ and that is different from the inclusion.
\end{lemma}

\begin{proof}
  Denote by~$\overline{K}$ the algebraic closure of~$K$ inside~$\C$.

  Suppose first that~$\cA$ is contained in~$\Kalg$.
  By the primitive element theorem, there is~$\alpha$ in~$\Kalg$ such that ${K(\cA) = K(\alpha)}$.
  Our assumption that~$\cA$ is not contained in~$K$ implies that the minimal polynomial of~$\alpha$ over~$K$ is of degree at least two.
  Thus, this polynomial has a root~$\alpha'$ different from~$\alpha$.
  It follows that there is a field homomorphism ${K(\alpha) \to \C}$ that is the identity on~$K$ and that maps~$\alpha$ to~$\alpha'$.
  It is thus different from the inclusion.

  It remains to consider the case where~$\cA$ is not contained in~$\Kalg$.
  In this case, there is a nonempty subset~$\cA_0$ of~$\cA$ that is algebraically independent over~$\Kalg$.
  Increasing~$\cA_0$ if necessary, assume it is maximal with this property.
  Then, $\Kalg(\cA)$ is a finite extension of~$\Kalg(\cA_0)$.
  Since~$\Kalg(\cA_0)$ is isomorphic to the field of rational functions with coefficients in~$\Kalg$ in~$\# \cA_0$ variables, there is a field isomorphism ${\sigma \colon \Kalg(\cA_0) \to \Kalg(\cA_0)}$ that is the identity on~$\Kalg$ and such that for some~$a_0$ in~$\cA_0$ we have ${\sigma(a_0) = 2 a_0}$.
  Since~$\Kalg(\cA)$ is a finite extension of~$\Kalg(\cA_0)$ and~$\C$ is algebraically closed, $\sigma$ extends to a field homomorphism ${\Kalg(\cA) \to \C}$.
\end{proof}

\begin{proof}[Proof of Proposition~\ref{p:arithmetically-modular}]
  We use that~$1/j$ is holomorphic at~$i \infty$, see, \emph{e.g.}, \cite[Chapter~4, Section~1]{Lan87}.
  Replacing~$f$ by~$1 / f$ if necessary, assume that~$f$ is also holomorphic at~$i \infty$.
  Let~$\Phi(X, Y)$ be a modular polynomial of~$f$ in~$\C[X, Y]$ (Proposition~\ref{p:modular-functions}).
  Replacing~$\Phi$ by a constant multiple if necessary, assume that one of the coefficients of~$\Phi$ is equal to~$1$.
  Denote by~$\delta$ the degree of~$X$ in~$\Phi(X, Y)$, and note that the polynomial
  \begin{equation}
    \label{eq:96}
    \Psi(X, Y)
    \=
    X^{\delta} \Phi(1/X, Y)
  \end{equation}
  in~$\C[X, Y]$ is also irreducible.

  For each pair of nonnegative integers~$(k, \ell)$, denote by~$A_{k, \ell}$ the coefficient of~$X^k Y^{\ell}$ in~$\Psi(X, Y)$.
  Moreover, denote by~$I$ the set of all~$(k, \ell)$ such that ${A_{k, \ell} \neq 0}$.
  By our normalization of~$\Phi$, there is ${(k_0, \ell_0)}$ such that ${A_{k_0, \ell_0} = 1}$.
  Suppose that~$\Psi(X, Y)$ is not in~$K[X, Y]$, so the set
  \begin{equation}
    \label{eq:97}
    \cA
    \=
    \{ A_{k, \ell} \colon (k, \ell) \in I \}
  \end{equation}
  is not contained in~$K$.
  By Lemma~\ref{l:transcendental-deformation}, there is a field homomorphism ${\sigma \colon K(\cA) \to \C}$ that is the identity on~$K$ and that is different from the inclusion.
  It follows that for some~$(k', \ell')$ in~$\cA$ we have ${\sigma(A_{k', \ell'}) \neq A_{k', \ell'}}$.

  For each integer ${n \ge 0}$, denote by~$a_n^{k, \ell}$ the coefficient of~$\exp \left( \frac{2 \pi i n}{h} \tau \right)$ in the Fourier series expansion of~$(1/j)^k f^{\ell}$.
  Since~$1/j$ is holomorphic at~$i \infty$ and its Fourier series expansion has coefficients in~$\Q$, see, \emph{e.g.}, \cite[Chapter~4, Section~1]{Lan87}, our hypothesis implies that~$a_n^{k, \ell}$ is in~$K$.
  On the other hand, the fact that the function~$\Psi(1/j, f)$ vanishes identically implies that for every integer ${n \ge 0}$ we have
  \begin{equation}
    \label{eq:98}
    \sum_{(k, \ell) \in I} A_{k, \ell} a_n^{k, \ell}
    =
    0
    \text{ and }
    \sum_{(k, \ell) \in I} \sigma(A_{k, \ell}) a_n^{k, \ell}
    =
    0.
  \end{equation}
  It follows that the polynomial
  \begin{equation}
    \label{eq:99}
    \Psi_0(X, Y)
    \=
    \sum_{(k, \ell) \in I} (\sigma(A_{k, \ell}) - A_{k, \ell}) X^k Y^{\ell}
  \end{equation}
  in~$\C[X, Y]$, is such that the function~$\Psi_0(1/j, f)$ vanishes identically.
  Note also that~$\Psi$ is nonzero, because the coefficient of~$X^{k'} Y^{\ell'}$ in~$\Psi_0(X, Y)$ is nonzero by our choice of~$\sigma$.
  Moreover, the coefficient of~$X^{k_0} Y^{\ell_0}$ in~$\Psi_0(X, Y)$ is zero, so~$\Psi_0$ is not a scalar multiple of~$\Psi$.

  Consider the polynomial
  \begin{equation}
    \label{eq:100}
    \Phi_0(X, Y)
    \=
    X^{\delta} \Psi_0(1/X, Y)
  \end{equation}
  in~$\C[X, Y]$.
  The functions~$\Phi(j, f)$ and~$\Phi_0(j, f)$ vanish identically.
  By Proposition~\ref{p:modular-functions}$(i)$, this implies that the polynomial~$\Phi_0$ vanishes on the zero set of~$\Phi$.
  Since~$\Phi$ is irreducible over~$\C$, we conclude that~$\Phi$ divides~$\Phi_0$.
  Since the degree of~$\Phi_0$ in~$X$ and in~$Y$ is less than or equal to the corresponding degree of~$\Phi$, we conclude~$\Phi_0$ is a scalar multiple of~$\Phi$.
  However, this would imply that~$\Psi_0$ is a scalar multiple of~$\Psi$, which is false.
  This contradiction proves that~$\Phi(X, Y)$ is in~$K[X, Y]$, and completes the proof of the proposition.
\end{proof}

\bibliographystyle{alpha}

\end{document}